\documentclass[11pt,reqno]{amsart}
\usepackage{indentfirst,enumerate,cite,amssymb,amsfonts,amsmath,amsthm,mathrsfs,dsfont}
\usepackage{color}
\usepackage{graphicx}
\usepackage{multicol}
\usepackage[colorlinks=true,linkcolor=blue,citecolor=blue]{hyperref}
\usepackage{enumerate}
\usepackage{geometry}
\numberwithin{equation}{section}

\newcommand{\dbar}{\mathrm{d}\hspace*{-0.15em}\bar{}\hspace*{0.1em}}

\newcommand{\dx}{{\mathrm{d}x}}
\newcommand{\dy}{{\mathrm{d}y}}

\newcommand{\dxi}{{\mathrm{d}\xi}}

\newtheorem{theorem}{Theorem}[section]
\newtheorem{definition}[theorem]{Definition}
\newtheorem{example}[theorem]{Example}
\newtheorem{lemma}[theorem]{Lemma}
\newtheorem{proposition}[theorem]{Proposition}
\newtheorem{corollary}[theorem]{Corollary}

\newtheorem{remark}[theorem]{Remark}

\newcommand{\R}{{\mathbb{R}}} 
\newcommand{\C}{{\mathbb{C}}} 
\newcommand{\Z}{{\mathbb{Z}}} 
\newcommand{\N}{{\mathbb{N}}} 

\newcommand{\caS}{{\mathcal{S}}}
\newcommand{\E}{{\mathcal{E}}}
\newcommand{\D}{{\mathcal{D}}}
\newcommand{\RR}{{\mathcal{R}}}
\newcommand{\ZZ}{{\mathcal{Z}}}

\newcommand{\Op}{{\mathrm{Op}}}
\newcommand{\SG}{{\mathrm{SG}}}
\newcommand{\AG}{{\mathrm{AG}}}
\newcommand{\ST}{{\mathrm{ST}}}
\newcommand{\FS}{{\mathrm{FS}}}

\newcommand{\Mj}{{\{M_j\}_{j\in\mathbb{N}_0}}} 
\newcommand{\Aj}{{\{A_j\}_{j\in\mathbb{N}_0}}} 
\newcommand{\Bj}{{\{B_j\}_{j\in\mathbb{N}_0}}} 

\newcommand{\M}{{\mathcal{M}}} 
\newcommand{\A}{{\mathcal{A}}} 
\newcommand{\B}{{\mathcal{B}}} 

\newcommand{\X}{{\overline{X}}}

\newcommand{\cl}{{\mathrm{cl}}}

\newcommand{\mj}{{\{m_j\}_{j\in\mathbb{N}_0}}}

\allowdisplaybreaks

\begin{document}
	
	\title[Gelfand-Shilov spaces and ultradifferential SG-operators on manifolds]{Gelfand-Shilov spaces\\and operators with ultradifferential weighted symbols\\
	  on non-compact manifolds}
	
	\author[S. Coriasco]{Sandro Coriasco}
	\email{sandro.coriasco@unito.it}
	\address{
		Dipartimento di Matematica ``G. Peano'',
		Universit\`a degli Studi di Torino,  
		V. C. Alberto, n. 10, I-10123 Torino 
		Italy}

	\author[P.M. Tokoro]{Pedro Meyer Tokoro}
	\email{pedro.tokoro@ufpr.br}
	\address{
		Programa de P\'os-Gradua\c c\~ao em Matem\'atica,  
		Universidade Federal do Paran\'a,  
		Caixa Postal 19096\\  CEP 81530-090, Curitiba, Paran\'a,  
		Brasil}
	

	
	\begin{abstract}
		We invariantly define Gelfand-Shilov spaces on classes of non-compact manifolds with a certain ``structure at infinity''.
		We also construct and study a global calculus of pseudodifferential operators on such manifolds, locally defined by symbols
		satisfying estimates associated with weight sequences. The operators so obtained act naturally on the previously
		defined Gelfand-Shilov spaces. These generalise analogous functional spaces and operators defined on Euclidean spaces.
	\end{abstract}
	
	\maketitle

	
\section{Introduction}\label{sec:intro}

When studying properties of solutions of equations in distributional setting on a non-compact manifold,
one is interested, in addition to the local regularity, also to the behaviour in \textit{neighbourhoods of points at infinity}, 
a notion which has itself to be made precise. In the microlocal approach, this reflects into defining calculi of operators involving
symbols satisfying global estimates also in the space variable, in contrast to the usual requirements to the symbols employed 
in the usual H\"ormander's pseudodifferential calculus on compact manifolds.

On $\R^n$, many calculi where the symbols satisfy global estimates exist, 
within the so-called Weyl-H\"ormander calculus, see \cite[Ch. XVIII]{Hormander_3}.
Among them, well known ones are Shubin's calculus and the SG-calculus, see \cite{Cordes76,cordes,NicRod,Parenti1972,Shubin1987}. 
While the former does not have good invariance properties in general (see \cite{Krainer} for an invariantly defined 
calculus on a manifold with ends $X$, which reduces to Shubin's calculus when $X=\R^n$), the latter can be invariantly 
defined on a wide class of non-compact manifolds, the so-called SG-manifolds, see \cite{Schrohe86}.
Scattering manifolds (interior of a compact manifold with boundary with suitable metric structure close to the boundary, which represents ``the points at infinity'') are 
another relevant class of non-compact manifolds where the calculus of operators associated with local classical SG-symbols can be invariantly defined as well, see \cite{Melrose_GST,Melrose1996}. A further class of manifolds with similar properties is given by the $\mathcal{S}$-manifolds, see \cite[Ch. 4]{cordes}.
In these three cases, the informal phrase ``at infinity'' is replaced by precise boundary components in a compactified phase space, and/or by a specific structure outside
compact submanifolds, and/or by requirements on the change of coordinates for large values of the variables when unbounded charts domains are involved. 
Ellipticity is then not only invertibility for large covectors over compact sets: it also includes invertibility at spatial infinity and, in the SG-classical case, 
``compatibility at the corner'', where both $|x|$ and $|\xi|$ are large (see \eqref{eq:sgell} and Sections \ref{sec:SGABHypoell}--\ref{sec:sgucl} below). 
In the sequel, we then focus on the SG-calculus and study its ultradifferential version on suitable 
analytic non-compact manifolds. We now summarize in short only a few of the
main features of the basic SG-calculus and the associated scale of (weighted) Sobolev spaces. 

Let $\langle\cdot\rangle = (1+|\cdot|)^{\frac{1}{2}}$. For $m,\mu\in\R$, the symbol class $\SG^{m,\mu}(\R^n)$ consists of all $a\in C^\infty(\R^n\times\R^n)$ such that 
\begin{equation}\label{eq:sg-symbol}
  |\partial_x^\beta\partial_\xi^\alpha a(x,\xi)|
  \leq C_{\alpha\beta}\,\langle x\rangle^{m-|\beta|}\langle \xi\rangle^{\mu-|\alpha|},
  \qquad x,\xi\in\R^n, \alpha,\beta\in\N^n_0,
\end{equation}
see, e.g., \cite{cordes,NicRod} for a detailed description of the associated calculus on $\R^n$. Thus the symbol has a two-components order: 
the usual ``differential order'' $\mu$ and 
an ``exit'' (or spatial growth/decay) order $m$. The quantization is the usual one, given by
\[
  [\Op(a)u](x)=(2\pi)^{-n}\iint e^{i(x-y)\cdot\xi}a(x,\xi)u(y)\,\dy\,\dxi=(2\pi)^{-n}\int e^{ix\cdot\xi}a(x,\xi)\widehat{u}(\xi)\,\dxi,
\]
initially on $u\in\mathcal S(\R^n)$ and then extended by duality to the whole $\mathcal S'(\R^n)$. Such operators are a closed class under composition and adjoint,
with asymptotic expansions formulae for the corresponding symbols completely analogous to those which hold true, for instance, for the standard H\"ormander calculus.

Defining, for $s,\sigma\in\R$, the so-called Sobolev-Kato spaces (or \textit{weighted Sobolev spaces}) as
\[
  H^{s,\sigma}(\R^n)=\{u\in\mathcal S'(\R^n): \langle{x}\rangle^{s}u\in H^\sigma(\R^n)\},
\]
where $H^\sigma(\R^n)$ is the usual Sobolev space of smoothness $\sigma$, equipped
with the naturally associated Hilbert norm, one has that any $a\in \SG^{m,\mu}(\R^n)$ gives a linear continuous operator
\[
	\Op(a)\colon H^{s,\sigma}(\R^n)\to H^{s-m,\sigma-\mu}(\R^n).
\]

Notice that we have to topological equalities
\[
	\mathcal{S}(\R^n)=\bigcap_{s,\sigma\in\R}H^{s,\sigma}(\R^n)
	\quad\text{and}\quad
	\mathcal{S}'(\R^n)=\bigcup_{s,\sigma\in\R}H^{s,\sigma}(\R^n).	
\]

A symbol is SG-elliptic if, outside a compact subset of the phase space, we have
\begin{equation}\label{eq:sgell}
  |a(x,\xi)|\geq C\,\langle x\rangle^{m}\langle \xi\rangle^{\mu},
\end{equation}
for some constant $C>0$. Then there exists $b\in \SG^{-\mu,-m}$ satisfying $\Op(a)\Op(b)=I+R_1$ and $\Op(b)\Op(a)=I+R_2$, where the remainders $R_1,R_2$ have Schwartz kernels in $\mathcal S(\R^{2n})$.  

As recalled above, in \cite{Schrohe86} it is proved that the SG-calculus and the weighted spaces on $\R^n$, as well as the analog of Schwartz's spaces of functions and 
distributions, can be invariantly defined on SG-manifolds. 
The latter are non-compact manifolds with 
a specific atlas structure, \textit{compatible} with the SG-operators and the Schwartz's spaces of functions and distributions structure (see \cite{MP2002},
\cite[Ch. 4]{cordes} and \cite{Melrose_GST} for similar results on a class of manifolds with ends, on $\mathcal{S}$-manifolds, and on scattering manifolds,
respectively). 
A relevant subclass of SG-manifolds are those which can be 
regarded as the interior of compact manifolds with boundary. Manifolds $X$ that admit such a compactification, which we denote by $\overline{X}$, 
are of great interest in geometric analysis. In fact, under the suitable choices of a Riemannian metric $g$ on $X$, these objects can model, among others, 
manifolds with conical or cylindrical ends, and, more generally, asymptotically Euclidean manifolds 
(see, e.g. \cite{cordes,Krainer,Melrose_GST,Melrose_SST}), as well as manifolds with conical, or more complicated, singularities 
(see, e.g., \cite{EgSchu97,LauterSeiler2001,Loya1998,Mazzeo1991,Melrose_APS,Melrose1981,SchulzeBook98}). 

The features of SG-calculus and its associated scale of weighted spaces allow to describe power-like decay properties, with respect to the weight $\langle\cdot\rangle$,
of solutions of equations in the temperate distributions space, 
in addition to the usual smoothness ones, as well as the propagation of ``decay and smoothness singularities'' in terms of global wave-front sets 
(see, e.g., \cite{CM2013,CJT_2016,Melrose_GST}). On Euclidean spaces, however, the Sobolev-Kato and/or Schwartz spaces, of course, are not the 
only possible choices when one wants simultaneous 
global control of decay and differentiability.  In particular, interesting and widely investigated settings are the Gelfand--Shilov spaces $\mathcal{S}_\theta^\theta(\R^n)$, $\theta>1$.
These may be described, in the so-called Roumieu form, as smooth functions satifying estimates of the type
\[
  |x^\alpha\partial^\beta u(x)|\leq C h^{|\alpha|+|\beta|}(\alpha!\,\beta!)^\theta,
  \qquad x\in\R^n,
\]
for $C,h>0$ independent on $\alpha,\beta\in\N_0^n$ (see  \cite{GelfandShilov1968}). These spaces are Fourier invariant and encode Gevrey regularity together with subexponential decay at infinity.  
In \cite{CapRod2006}, the authors developed a variant of the SG-calculus adapted to this setting, by replacing the pure polynomial derivative bounds 
\eqref{eq:sg-symbol} with corresponding Gevrey-type estimates, namely
\begin{equation}\label{eq:SGGS}
  |\partial_x^\beta\partial_\xi^\alpha a(x,\xi)|
  \leq C h^{|\alpha|+|\beta|}(\alpha!\,\beta!)^\theta
       \langle x\rangle^{m-|\beta|}\langle \xi\rangle^{\mu-|\alpha|},
  	\qquad x,\xi\in\R^n, \alpha,\beta\in\N^n_0,
\end{equation}
which gives a global counterpart of the local calculus in Gevrey spaces introduced by Boutet de Monvel and Kr\'ee in \cite{BMK1967} (see also \cite{Rod_Gevrey}).

The resulting operators $\Op(a)$ act continuously on $\mathcal{S}_\theta^\theta(\R^n)$ and on its dual space of tempered ultradistributions. 
Their hypoellipticity theory and SG-wavefront set 
keep track, similarly as above, of ordinary microlocal regularity and loss of rapid decay at infinity. It is then natural to ask if families of Gelfand-Shilov spaces,
compatible with a suitable analog of the SG-operators subfamilies in \eqref{eq:SGGS}, can also be invariantly defined on appropriate non-compact manifolds.
In this paper, we show that such question has a positive answer.

In particular, we study general ultradifferential SG-operators and we introduce and analyse in detail Gelfand-Shilov spaces of Roumieu and Beurling type, by means of the so-called weight sequences, extending the results obtained in \cite{CapRod2006}. Moreover, we show that these operators have an invariant meaning on an analytic subclass of the SG-manifolds.

We first recall some preliminary results in Section \ref{sec:prel}, about weight sequences and functional spaces defined by means of them. 
In Section \ref{sec:SGABRn} we introduce ultradifferential SG-symbols on $\R^n$ associated with weight sequences, and prove that they provide operator algebras
closed under composition and adjoint. We also study hypoellipticity, prove the existence of the parametrix to hypoelliptic elements in the calculus, 
describe the SG-classical subclasses on $\R^n$, and study the properties of the versions of global wave-front set adapted to the calculus. 
Section \ref{sec_coord_inv} paves the ground to the extension to manifolds, with the proof of the invariance of the ultradifferential SG-operators
under suitable analytic compatible diffeomorphisms. The extension to manifolds is then performed in the subsequent Sections \ref{sec:compmfs}, \ref{sec:GSspcsGSAmf}, 
and \ref{sec:SGABmf}, where SGA-manifolds are introduced, and the Gelfand-Shilov spaces and ultradifferentiable SG-calculus defined on them, respectively. 
The concluding Section \ref{sec:SGABmwe} is devoted to studying classical SG-operators on
analytic manifolds with ends.

\section*{Acknowledgements}
This study was financed in part by CAPES -- Brasil (Finance Code 001). The first author was supported in part by the Italian Ministry of the University and Research -- MUR, within the framework of the Call relating to the scrolling of the final rankings of the PRIN 2022 -- Project Code 2022HCLAZ8, CUP D53C24003370006 (PI A.~Palmieri, Local unit Sc.~Resp.~S.~Coriasco). Thanks are due to L. Mari, T. Pacini, and A. Raffero, for useful hints and discussions.


\section{Preliminaries}\label{sec:prel}
In this section we recall various concepts and results, which we will employ in the sequel.

\subsection{Weight sequences}
Following Komatsu \cite{Komatsu_ultra_1}, we consider classes of positive sequences, which will be involved
in the many definitions in the sequel.
\begin{definition}
	A weight sequence is a sequence $\M=\Mj$ of positive numbers such that:
	\begin{enumerate}
		\item[$(M0)$] $M_0=M_1=1$;
		\item[$(M1)$] $M_j^2\leq M_{j-1}M_{j+1}$, $j\in\N$;
		\item[$(M2)'$] There exist $A,H\geq 1$ such that $M_{j+1}\leq AH^j M_j$, for all $j\in\N_0$;
	\end{enumerate}
	We say that a weight sequence is non-quasianalytic if it satisfies
	\begin{enumerate}
		\item[$(M3)'$] $\displaystyle\sum_{j=1}^\infty\frac{M_{j-1}}{M_j}<+\infty$.
	\end{enumerate}
\end{definition}
Eventually, we may also need the following stronger conditions:
\begin{itemize}
	\item[$(M2)$] There exist $A,H\geq 1$ such that $\displaystyle M_j \leq AH^j\min_{0\leq k\leq j}\{M_{j-k}M_k\}$, for all $j\in\N_0$;
	
	\item[$(M4)$] $\dfrac{M_{j-k}}{(j-k)!} \dfrac{M_k}{k!} \leq \dfrac{M_j}{j!}$, for all $j\in\N$ and $k\leq j$.
\end{itemize}
It is clear that $(M2)$ implies $(M2)'$ and $(M4)$ implies $(M1)$. 
We remark that $(M3)'$ is a weaker version of Komatsu's condition $(M3)$, which will not be necessary in this work.
\begin{proposition}\label{MjMj-k}
	Let $\M=\Mj$ be a weight sequence. Then, we have $M_{k}M_{j-k}\leq M_{j}$, for all $j,k\in\N_0$ such that $k\leq j$.
\end{proposition}
\begin{proof}
	See \cite[Proposition 2.6]{LV2021}.
\end{proof}
Let $\A=\Aj$ and $\M=\Mj$ be two weight sequences. We write $\A\preceq \M$ if there exist $C,H>0$ such that $A_j\leq CH^jM_j$, for all $j\in\N_0$, and $\A\prec \M$ if, for every $H>0$, there exists $C>0$ such that $A_j\leq CH^jM_j$, for all $j\in\N_0$.
\begin{proposition}\label{incl_quasi_factorial}
	Suppose that $\M=\Mj$ satisfies $(M3)'$. Then $\{j!\}_{j\in\N_0}\prec \M$.
\end{proposition}
\begin{proof}
	If $\M$ satisfies $(M3)'$, by \cite[Lemma 4.1]{Komatsu_ultra_1}, we have $j/m_j\to 0$, $j\to\infty$, where $m_j = M_j/M_{j-1}$. Hence, as in the beginning of the proof of \cite[Theorem 7.2]{Komatsu_ultra_1}, for every $H>0$, we have
	\[\frac{j!}{H^j M_j} = \frac{1}{Hm_1} \cdot \frac{2}{Hm_2}\cdots \frac{j}{Hm_j} \to 0,\]
	which implies our claim.
\end{proof}

\begin{lemma}\label{lemma_mjj_Mj}
	Let $\M=\Mj$ be a weight sequence and consider the associated sequence $\mj$ defined by $m_0=0$ and $m_j=M_j/M_{j-1}$, $j\geq 1$. Then, $\mj$ is increasing and $m_j^j\geq M_j$ for all $j\in\N_0$, where we use the convention $m_0^0=1$.
\end{lemma}
\begin{proof}
	First, the fact that $\mj$ is increasing is equivalent to condition $(M1)$. Indeed, we have
	\[m_{j+1}\geq m_j\ \Leftrightarrow\ \frac{M_{j+1}}{M_j}\geq \frac{M_j}{M_{j-1}}\ \Leftrightarrow\ M_{j-1}M_{j+1}\geq M_j^2,\]
	for every $j\geq 1$, and the case $j=0$ is trivial. Also, since $m_1=1$, it follows $m_j\geq 1$ for all $j\geq 1$. Then, for every $j\geq 1$ we have
	\[M_j = \frac{M_1}{M_0}\cdot\frac{M_2}{M_1}\cdots\frac{M_j}{M_{j-1}} = m_1m_2\cdots m_j \leq m_j^j,\]
	and the case $j=0$ is trivial.
\end{proof} 

\begin{proposition}\label{exp_EM}
	Let $\Mj$ be a non-quasianalytic weight sequence satisfying $(M4)$. If $f\in\E_{[\M]}(\R^n)$ and $h(x)=\exp[f(x)]$, $x\in\R^n$, then $h\in \E_{[\M]}(\R^n)$.
\end{proposition}
\begin{proof}
	See \cite[Lemma 6.9]{wagner_thesis}.
\end{proof}

\subsection{Ultradifferentiable functions}

Let $\M=\Mj$ be a weight sequence and $U\subset\R^n$ be an open subset. For each $h>0$ and $K\subset\subset U$ compact (that is, as usual, $K\subset U$ and $K$ is compact), we define the spaces
\[\E_{\M}(K;h) = \{f\in C^\infty(U)\,:\, \sup_{x\in K} \sup_{\alpha\in\N_0^n} h^{|\alpha|} M_{|\alpha|}^{-1} |\partial_x^\alpha f(x)| < +\infty\},\]
and
\[\D_{\M}(K;h) = \{f\in C_c^\infty(K)\,:\, \sup_{x\in K} \sup_{\alpha\in\N_0^n} h^{|\alpha|} M_{|\alpha|}^{-1} |\partial_x^\alpha f(x)| < +\infty\},\]
which are Banach spaces with the norm
\[\|f\|_{K,h} = \sup_{x\in K} \sup_{\alpha\in\N_0^n} h^{|\alpha|} M_{|\alpha|}^{-1} |\partial_x^\alpha f(x)|.\]

Then, we define
\[\E_{\{\M\}}(U) = \underset{K\subset\subset U}{\mathrm{proj\,lim}}\ \underset{h>0}{\mathrm{ind\,lim}} \ \E_{\M}(K;h),
\quad
\mathcal{E}_{(\M)}(U) = \underset{K\subset\subset U}{\mathrm{proj\,lim}}\ \underset{h>0}{\mathrm{proj\,lim}} \ \E_{\M}(K;h),\]
\[\D_{\{\M\}}(U) = \underset{K\subset\subset U}{\mathrm{ind\,lim}}\ \underset{h>0}{\mathrm{ind\,lim}} \ \D_{\M}(K;h),
\quad
\D_{(\M)}(U) = \underset{K\subset\subset U}{\mathrm{ind\,lim}}\ \underset{h>0}{\mathrm{proj\,lim}} \ \D_{\M}(K;h).\]

We adopt the notation $\E_{[\M]}(U)$ and $\D_{[\M]}(U)$ when referring to both types of spaces. Clearly $\D_{[\M]}(U)\subset \E_{[\M]}(U)$. If condition $(M3)'$ holds, we remark that the spaces $\D_{[\M]}(U)$ are non-trivial, see \cite{BMT1990,BMM2007} for the details.

The two following results are straightforward consequences of the properties of weight sequences. We refer the reader to \cite{Komatsu_ultra_1} for the proofs.

\begin{proposition}\label{incl_E}
	Let $\A=\Aj$ and $\M=\Mj$ be two weight sequences.
	\begin{enumerate}
		\item[(a)] If $\A\preceq\M$, then $\mathcal{E}_{\{\A\}}(U) \subset \mathcal{E}_{\{\M\}}(U)$ and $\mathcal{D}_{\{\A\}}(U) \subset \mathcal{D}_{\{\M\}}(U)$.
		
		\item[(b)] If $\A\preceq\M$, then $\mathcal{E}_{(\A)}(U) \subset \mathcal{E}_{(\M)}(U)$ and $\mathcal{D}_{(\A)}(U) \subset \mathcal{D}_{(\M)}(U)$.
		
		\item[(c)] If $\A\prec\M$, then $\mathcal{E}_{\{\A\}}(U) \subset \mathcal{E}_{(\M)}(U)$ and $\mathcal{D}_{\{\A\}}(U) \subset \mathcal{D}_{(\M)}(U)$.
	\end{enumerate} 
\end{proposition}

\begin{proposition}\label{E_cl_prod_diff}
	Let $\M=\Mj$ be a weight sequence. Then, the spaces $\mathcal{E}_{[\M]}(U)$ and $\mathcal{D}_{[\M]}(U)$ are closed under products and differentiation.
\end{proposition}

\begin{example}
	Let $\M=\{j!^s\}_{j\in\N_0}$. If $s>1$, then $\mathcal{E}_{\{\M\}}(U)$ is the Gevrey space $\mathcal{G}^s(U)$, and if $s=1$, $\mathcal{E}_{\{\M\}}(U)$ is the space $C^\omega(U)$ of real-analytic functions on $U$.
\end{example}

Consider also the spaces $\D_{[\M]}'(U)$ and $\E_{[\M]}'(U)$ of $[\M]$-ultradistributions and compactly supported $[\M]$-ultradistributions, which are the topological duals of $\D_{[\M]}(U)$ and $\E_{[\M]}(U)$, respectively. Clearly $\E_{[\M]}'(U)\subset \D_{[\M]}'(U)$.

Every $f\in \E_{[\M]}(U)$ can be regarded as an element in $\D_{[\M]}'(U)$, defining, as usual,
\[\varphi\in \mathcal{D}_{[\M]}(U)\mapsto \langle f, \varphi \rangle = \int_{U} f(x)\varphi(x)\,\dx.\]

\begin{definition}
	Let $V\subset U$ be open sets. We say that $u\in \D_{[\M]}'(U)$ is $\E_{[\M]}$ in $V$ if there exists $f\in \E_{[\M]}(V)$ such that
	\[\langle u,\varphi\rangle = \int_{V}f(x)\varphi(x)\,\dx,\quad \forall\varphi\in \D_{[\M]}(V).\]
\end{definition}

We can now define the $[\M]$-singular supports for $[\M]$-ultradistributions.

\begin{definition}
	Let $u\in \mathcal{D}_{[\M]}'(U)$. We say that $x\in U$ does not belong to $[\M]\text{-\,}\mathrm{sing\,supp}$ if there exists a neighbourhood $V\subset U$ of $x$ such that $u$ is $\mathcal{E}_{[\M]}$ in $V$.
\end{definition}

\begin{theorem}\label{inv_analytic}
	If $f\in \mathcal{E}_{[\M]}(U)$ (respectively, $f\in \mathcal{D}_{[\M]}(U)$) and $g:V\to U$ is an analytic map between open sets $V\subset\R^m$ and $U\subset\R^n$, then $f\circ g\in \mathcal{E}_{[\M]}(V)$ (respectively, $f\circ g\in \mathcal{D}_{[\M]}(V)$).
\end{theorem}
\begin{proof}
	See \cite[Proposition 8.4.1]{Hormander_1}.
\end{proof}

By the previous theorem, we can define the spaces of ultradifferentiable functions on analytic manifolds. We recall that a manifold is analytic if it admits an atlas whose coordinate changes are analytic maps. Recalling Whitney \cite{Whitney}, every smooth manifold admits an analytic structure,
compatible with its smooth structure. Then, from now on, let $X$ be a paracompact manifold with an analytic countable atlas $\{(\Omega_k,\varphi_k)\}_{k\in\N_0}$. 

\begin{definition}
	We define $\mathcal{E}_{[\M]}(X)$ as the space of all $f\in C^\infty(X)$ such that, for every local chart $(\Omega_k,\varphi_k)$, the function $f\circ\varphi_k^{-1}$ belongs to $\mathcal{E}_{[\M]}(U_k)$, where $U_k=\varphi_k(\Omega_k)\subset\R^n$.
\end{definition}
By the paracompactness, we can also assume that $\Omega_k$ has compact closure and we can extend the domains of each chart to a slightly larger open set.

\subsection{Gelfand-Shilov spaces on $\R^n$}
Given a weight sequence $\{\M\}_{j\in\N_0}$ and $h>0$, we define
\[ \caS_{\M}(\R^n;h) = \{f\in C^\infty(\R^n) \, : \, \|f\|_h = \sup_{x\in\mathbb{R}^n} \sup_{\alpha,\beta\in\N_0^n} h^{|\alpha|+|\beta|} M_{|\alpha|}^{-1} M_{|\beta|}^{-1} |x^\beta\partial_x^\alpha f(x)| < +\infty\}, \]
which is a Banach space with the norm $\|\cdot\|_h$. Then, we define
\[\mathcal{S}_{\{\M\}}(\R^n) = \underset{h>0}{\mathrm{ind\,lim}}\ \caS_{\M}(\R^n;h),
\quad \mathcal{S}_{(\M)}(\R^n) = \underset{h>0}{\mathrm{proj\,lim}}\ \caS_{\M}(\R^n;h).\]

The inclusion maps $\caS_{\M}(\R^n;h)\hookrightarrow \caS_{\M}(\R^n;h_+)$ are continuous when $h\leq h_+$ and compact when $h<h_+$. In particular, $\mathcal{S}_{\{\M\}}(\R^n)$ and $\mathcal{S}_{(\M)}(\R^n)$ are DFS and FS spaces, respectively. The proofs of the next two Propositions \ref{prop:gsspcs_a} and \ref{prop:gsspcs_b} are analogous to those found in \cite{Komatsu_ultra_1} for the spaces of ultradifferential functions. For the sake of brevity, we omit them.

\begin{proposition}\label{prop:gsspcs_a}
	If $\A\preceq \M$, then $\mathcal{S}_{\{\A\}}(\R^n) \subset \mathcal{S}_{\{\M\}}(\R^n)$ and $\mathcal{S}_{(\A)}(\R^n) \subset \mathcal{S}_{(\M)}(\R^n)$. Moreover, if $\A\prec\M$, then $\mathcal{S}_{\{\A\}}(\R^n) \subset \mathcal{S}_{(\M)}(\R^n)$.
\end{proposition}
\begin{proposition}\label{prop:gsspcs_b}
	The spaces $\mathcal{S}_{[\M]}(\R^n)$ are closed under products and differentiation.
\end{proposition}

We define the spaces $\mathcal{S}_{\{\M\}}'(\R^n)$ and $\mathcal{S}_{(\M)}'(\R^n)$ of tempered ultradistributions as the topological duals of $\mathcal{S}_{\{\M\}}(\R^n)$ and $\mathcal{S}_{(\M)}(\R^n)$. In this case, $\mathcal{S}_{\{\M\}}'(\R^n)$ and $\mathcal{S}_{(\M)}'(\R^n)$ with the strong dual topology are FS and DFS spaces, respectively. As a consequence of the topological properties of $\caS_{[\M]}(\R^n)$, we can characterize the elements of the spaces $\mathcal{S}_{[\M]}'(\R^n)$ as follows.

\begin{proposition}
	A linear functional $u:\mathcal{S}_{\{\M\}}(\R^n)\to\C$ (respectively, $u:\mathcal{S}_{(\M)}(\R^n)\to\C$) belongs to $\mathcal{S}_{\{\M\}}'(\R^n)$ (respectively, $\mathcal{S}_{(\M)}'(\R^n)$) if and only if, for every $h>0$, there exists $C>0$ (respectively, there exist $C,h>0$) such that
	\[ |\langle u,f\rangle| \leq C\sup_{x\in\R^n}\sup_{\alpha,\beta\in\N_0^n} h^{|\alpha|+|\beta|}M_{|\alpha|}M_{|\beta|}|x^\beta\partial_x^\alpha f(x)|.\]
\end{proposition}

\begin{proposition}
	The spaces $\mathcal{S}_{[\M]}(\R^n)$ are nuclear and we have the following isomorphisms of locally convex topological vector spaces:
	\[\mathcal{S}_{[\M]}(\R^n)\,\widehat{\otimes}\, \mathcal{S}_{[\M]}(\R^m) \simeq \mathcal{S}_{[\M]}(\R^{n+m}) \simeq \mathcal{L}(\mathcal{S}_{[\M]}'(\R^n),\mathcal{S}_{[\M]}(\R^m)),\]
	and
	\[\mathcal{S}_{[\M]}'(\R^n)\,\widehat{\otimes}\, \mathcal{S}_{[\M]}'(\R^m) \simeq \mathcal{S}_{[\M]}'(\R^{n+m}) \simeq \mathcal{L}(\mathcal{S}_{[\M]}(\R^n),\mathcal{S}_{[\M]}'(\R^m)).\]
	
	In particular, for every continuous linear operator $T:\mathcal{S}_{[\M]}(\R^n)\to \mathcal{S}_{[\M]}'(\R^n)$, there exists a unique tempered ultradistribution $K_T\in \mathcal{S}_{[\M]}'(\R^{2n})$ satisfying
	\[\langle T f,g\rangle = \langle K_T, f\otimes g\rangle,\]
	for all $f,g\in \mathcal{S}_{[\M]}(\R^n)$.
\end{proposition}
\begin{proof}
	See \cite[Proposition 2]{Pran2013}.
\end{proof}

\begin{lemma}\label{est_GS_bracket}
	Let $f\in\caS_{\{\M\}}(\R^n)$ (respectively, $f\in\caS_{(\M)}(\R^n)$) and fix $N\in\N_0$. Then, there exist $C,h>0$ (respectively, for every $h>0$, there exist $C>0$) depending on $N$ such that
	\[\sup_{x\in \R^n}|\langle x\rangle^{2N}\partial_x^\alpha f(x)|\leq Ch^{|\alpha|}M_{|\alpha|},\quad\forall \alpha\in\N_0^n.\]
\end{lemma}
\begin{proof}
	Let $f\in\caS_{\{\M\}}(\R^n)$ and $C_f,h_f>0$ be such that
	\[\sup_{x\in\R^n}|x^\beta\partial_x^\alpha f(x)|\leq C_fh_f^{|\alpha|+|\beta|}M_{|\alpha|}M_{|\beta|},\quad\forall \alpha,\beta\in\mathbb{N}_0^n.\]
	
	Then, given $\alpha\in\mathbb{N}_0^n$ and $N\in\mathbb{N}_0$, we have
	\begin{align*}
		|\langle x\rangle^{2N}\partial_x^\alpha f(x)| & 
		= \hspace{4mm} |(1+|x|^2)^N\partial_x^\alpha f(x)|
		&& 
		\hspace{-13mm} \leq \hspace{3mm}
		\sum_{\ell=0}^N\binom{N}{\ell}|x|^{2\ell}|\partial_x^\alpha f(x)|\\
		& \leq \hspace{4mm} \sum_{\ell=0}^N\binom{N}{\ell}\sum_{|\beta|=2\ell}\frac{(2\ell!)}{\beta!}|x^\beta\partial_x^\alpha f(x)|
		&& \hspace{-13mm}
		\leq  \hspace{3mm}
		\sum_{\ell=0}^N\binom{N}{\ell}\sum_{|\beta|=2\ell}\frac{(2\ell)!}{\beta!}C_fh_f^{|\alpha|+|\beta|} M_{|\alpha|}M_{|\beta|}\\
		& \leq C_f h_0^{2N} h_f^{|\alpha|}M_{|\alpha|}M_{2N}\sum_{\ell=0}^N\binom{N}{\ell}\sum_{|\beta|=2\ell}\frac{(2\ell)!}{\beta!}\\
		& \leq C_f h_0^{2N} h_f^{|\alpha|}n^{2N}M_{|\alpha|}M_{2N} \sum_{\ell=0}^N\binom{N}{\ell}\sum_{|\beta|=2\ell} 1\\
		& \leq C h_f^{|\alpha|}M_{|\alpha|},
	\end{align*}
	where $h_0=\min\{h_f,1\}$ and $C= C_f2^n(8h_0^2n^2)^N M_{2N}$. Also, we used the estimates and identities (1.2.1), (1.2.3), (1.2.6), and (1.2.7) from \cite{Rod_Gevrey}, and the fact that the number of multi-indices $\beta \in \mathbb{N}_0^n$ such that $|\beta|=2\ell$ is bounded by $2^{n+2\ell}$ (see the comment in the end of page 11 of \cite{Rod_Gevrey}). The projective case is analogous. The proof is complete.
\end{proof}

\begin{lemma}\label{lemma_poly_GS}
	Assume that $\M=\Mj$ is a non-quasianalytic weight sequence, that is, $(M3)'$ holds. Let $f\in\caS_{\{\M\}}(\R^n)$ (respectively, $f\in\caS_{(\M)}(\R^n)$). Then, for every $\gamma\in\mathbb{N}_0^n$, we have $x^\gamma f \in\caS_{\{\M\}}(\R^n)$ (respectively, $x^\gamma f \in\caS_{(\M)}(\R^n)$) and there exists $C,h>0$ (respectively, for every $h>0$, there exists $C>0$) do not depending on $\gamma$ such that
	\[\sup_{x\in\R^n}|\partial_x^\alpha(x^\gamma f(x))|\leq Ch^{|\alpha|+|\gamma|}M_{|\alpha|}M_{|\gamma|},\quad\forall \alpha\in\mathbb{N}_0^n.\]
\end{lemma}
\begin{proof}
	The proof follows the same lines of the one of Lemma \ref{est_GS_bracket}, employing the Leibniz rule.
\end{proof}

\subsection{Weight functions}

Following Braun, Meise, and Taylor \cite{BMT1990} and Petzsche and Vogt \cite{PV1984}, we consider a useful class of increasing functions.

\begin{definition}
	We say that an increasing map $\omega:[0,+\infty)\to [0,+\infty)$ is a weight function if it satisfies the following conditions:
	\begin{enumerate}
		\item[$(W1)$] $\displaystyle\sup_{\varepsilon\geq 1}\limsup_{t\to+\infty}\frac{\omega(\varepsilon t)}{\varepsilon\omega(t)} < +\infty$;
		
		\item[$(W2)$] $\displaystyle \limsup_{t\to+\infty}\frac{\omega(t)}{t}<+\infty$;
		
		\item[$(W3)$] $\displaystyle\lim_{t\to +\infty}\dfrac{\omega(t)}{\log(t)}=+\infty;$
		
		\item[$(W4)$] $\displaystyle\int_1^{+\infty}\dfrac{\omega(t)}{t^2}\,\mathrm{d}t<+\infty$.
	\end{enumerate}
\end{definition}

Notice that condition $(W1)$ is exactly the Petzsche and Vogt's condition $(\alpha_1)$, which is stronger than Braun, Meise, and Taylor's condition $(\alpha)$.

Given a non-quasianalytic weight sequence $\M=\Mj$, consider the associated function $\omega_\M$ given by $\omega_\M(0)=0$ and
\begin{equation}\label{weight_seq}
	\omega_\M(t) = \sup_{j\in\N_0}\log\left(\frac{t^j}{M_j}\right),\quad t>0.
\end{equation}

By \cite[Proposition 13]{BMM2007} and \cite[Lemma 5.5]{PV1984}, if $\M$ satisfies also $(M2)$ and $(M4)$, then $\omega_\M(t)$ is a weight function. We also have that $\omega_\M(t)$ is continuous, non-negative, monotonically increasing, and $e^{\omega_\M(t)}$ grows faster than any polynomial.

\begin{proposition}\label{prop_KMR}
	For every $t_1,t_2\geq 0$, we have
	\begin{enumerate}
		\item $e^{-\omega_\M(2t_1)}e^{-\omega_\M(2t_2)}\leq e^{-\omega_\M(t_1+t_2)}$;
		
		\item $e^{\omega_\M(t_1)}e^{\omega_\M(t_2)}\leq Ae^{\omega_\M(H(t_1+t_2))}$, where $A,H\geq 1$ are given by condition $(M2)$.
	\end{enumerate}
\end{proposition}
\begin{proof}
	See \cite[Proposition 1]{KMR_2022}. 
\end{proof}

As shown in \cite[Section 2.5]{CKP_book}, we can obtain an equivalent definitions for the spaces $\mathcal{S}_{[\M]}(\R^n)$ by means of the associated weight function.

\begin{proposition}\label{equiv_GS}
	For each $h,\varepsilon>0$, consider the Banach space
	\[\mathcal{S}_{\M}(\R^n;h,\varepsilon) = \{f\in C^\infty(\R^n)\,:\, \|f\|_{h,\varepsilon}<+\infty\},\]
	where
	\begin{equation}\label{GS_exp}\|f\|_{h,\varepsilon} = \sup_{x\in\R^n}\sup_{\alpha\in\N_0^n}\frac{e^{\omega_\M(\varepsilon\langle x\rangle)}|\partial_x^\alpha f(x)|}{h^{|\alpha|}M_{|\alpha|}}.\end{equation}
	
	Then, we have
	\[\mathcal{S}_{\{\M\}}(\R^n) = \underset{h,\varepsilon>0}{\mathrm{ind\,lim}}\ \mathcal{S}_{\M}(\R^n;h,\varepsilon),
	\quad
	\mathcal{S}_{(\M)}(\R^n) = \underset{h,\varepsilon>0}{\mathrm{proj\,lim}}\ \mathcal{S}_{\M}(\R^n;h,\varepsilon).\]
\end{proposition}

\begin{proposition}
	We have that $f\in C^\infty(\R^n)\cap L^1(\R^n)$ belongs to $\caS_{\{\M\}}(\R^n)$ (respectively, $\caS_{(\M)}(\R^n)$) if and only if there exist $C,L>0$ (respectively, for every $L>0$, there exists $C>0$) such that
	\[|f(x)|\leq Ce^{-\omega_\M(\varepsilon x)}\quad \text{and}\quad |\widehat{f}(\xi)|\leq Ce^{-\omega_\M(\varepsilon\xi)},\quad \forall x,\xi\in\R^n.\]
\end{proposition}

\begin{remark}
	In particular, the Fourier transform is an automorphism of $\caS_{[\M]}(\R^n)$ and extends to an automorphism of $\caS_{[\M]}'(\R^n)$.
\end{remark}

Given two non-negative functions $f$ and $g$, we write $f\lesssim g$ if there exists $K,R>0$ such that $f(t) \leq Kg(t)$, for all $|t|\geq R$. If $f\lesssim g\lesssim f$, we write $f\asymp g$. Let $\omega_0$ be such that $\omega_0\asymp\omega_\M$. Then, if we employ $\omega_0$ in the place of $\omega_\M$ in \eqref{GS_exp}, we clearly obtain again the same space with equivalent topology.

Notice that $\omega_\M(t)$ is not smooth in principle. However, assuming also condition $(M4)$, we are able to prove that any such weight function admits
an equivalent smooth one, in addition to further properties.

\begin{proposition}\label{prop_wf_E}
	Let $\M=\Mj$ be a non-quasianalytic weight sequence satisfying $(M2)$ and $(M4)$, and $\omega_\M(t)$ be the associated weight function. Then, there exists a weight function $\omega_0(t)$ such that
	\begin{enumerate}[(i)]
		\item For every $t_1,t_2\geq 0$, we have
		\[\omega_0(t_1+t_2)\leq \omega_0(2t_1)+\omega_0(2t_2).\]
		
		\item There exists $C>0$ such that, for every $t_1,t_2\geq 0$,
		\[\omega_0(t_1)+\omega_0(t_2)\leq \omega_0(2H(t_1+t_2))+C,\]
		where $H\geq 1$ is the constant from condition $(M2)'$;
		
		\item $\omega_0\asymp \omega_\M$;
		
		\item Given $\delta>0$, there exist $C,h>0$ (respectively, for every $h>0$, there exists $C>0$) such that %
		\[|\partial_{t}^{k}\omega_0(t)|\leq Ch^{k}M_{k} \omega_0(t),\]
		for all $k\in\N_0$ and $t>\delta$. In particular, $\omega_0 \in \E_{[\M]}((0,+\infty))$.
	\end{enumerate}
\end{proposition}
\begin{proof}
	Let $\chi\in\D_{[\M]}(\R)$ be a non-negative function such that $\mathrm{supp}(\chi)\subset [-1,1]$ and $\chi\equiv 1$ in $\left[-\frac{1}{2},\frac{1}{2}\right]$. Then, consider
	\[\omega_0(t) = \int_{-1}^1 \omega_\M(t+1-s)\chi(s)\,\mathrm{d}s.\]
	
	Applying the logarithm on Proposition \ref{prop_KMR}, we obtain
	\[\omega_\M(t_1)+\omega_\M(t_2) \leq \omega_\M(H(t_1+t_2))+\log(A)\]
	and
	\[\omega_\M(t_1+t_2) \leq \omega_\M(2t_1)+\omega_\M(2t_2),\]
	for all $t_1,t_2\geq 0$, where $A,H\geq 1$ are given by property $(M2)'$.
	
	\noindent\textit{(i)} Given $t_1,t_2\geq 0$, we have
	\begin{align*}
		\omega_0(t_1+t_2) & = \int_{-1}^1 \omega_\M(t_1+t_2+1-s)\chi(s)\,\mathrm{d}s\\
		& = \int_{-1}^1 \omega_\M\left(t_1+\frac{1-s}{2}+t_2 + \frac{1-s}{2}\right)\chi(s)\,\mathrm{d}s\\
		& \leq \int_{-1}^1\omega_\M(2t_1+1-s)\chi(s)\,\mathrm{d}s + \int_{-1}^1\omega_\M(2t_2+1-s)\chi(s)\,\mathrm{d}s\\
		& = \omega_0(2t_1)+\omega_0(2t_2).
	\end{align*}
	
	\noindent\textit{(ii)} Given $t_1,t_2\geq 0$, we have
	\begin{align*}
		\omega_0(t_1)+\omega_0(t_2) & = \int_{-1}^1 \left( \omega_\M(t_1+1-s)+\omega_\M(t_2+1-s)\right)\chi(s)\,\mathrm{d}s\\
		& \leq \int_{-1}^1 \omega_\M (H(t_1+t_2+2-2s))\chi(s)\,\mathrm{d}s + \log(A)\int_{-1}^1\chi(s)\,\mathrm{d}s\\
		& = \int_{-1}^1 \omega_\M \left(H(t_1+t_2)+\frac{1-s}{2}+\frac{(2H-1)(1-s)}{2}\right)\chi(s)\,\mathrm{d}s + C_0\\
		& \leq \omega_0(2H(t_1+t_2))+C,
	\end{align*}
	where
	\[C_0 = \log(A)\int_{-1}^1\chi(s)\,\mathrm{d}s \quad\text{and}\quad C = \int_{-1}^{1}\omega_\M((2H-1)(1-s))\chi(s)\,\mathrm{d}s + C_0.\]
	
	\noindent\textit{(iii)} We have
	\begin{equation*}
		\omega_\M(t) \leq \omega_\M(t+1-s) = \omega_\M(t+1-s)\chi(s) + \omega_\M(t+1-s)(1-\chi(s)),
	\end{equation*}
	for all $s\in[-1,1]$. Then
	\begin{align*}
		\omega_\M(t) & = \int_{-\frac{1}{2}}^{\frac{1}{2}}\omega_\M(t)\,\mathrm{d}s\\
		& \leq \int_{-\frac{1}{2}}^{\frac{1}{2}}\omega_\M(t+1-s)\chi(s)\,\mathrm{d}s + \int_{-\frac{1}{2}}^{\frac{1}{2}}\omega_\M(t+1-s)(1-\chi(s))\,\mathrm{d}s\\
		& \leq \omega_0(t),
	\end{align*}
	for all $t\geq 0$, noticing that $1-\chi(s)$ vanishes for $|s|\leq \frac{1}{2}$ and $\omega_\M(t+1-s)\chi(s)\geq 0$.
	
	On the other hand, by \textit{(ii)} and property $(W1)$ of $\omega_\M$, we can obtain $C_1,C_2>0$ such that
	\begin{align*}
		\omega_0(t) & = \int_{-1}^1 \omega_\M(t+1-s)\chi(s)\,\mathrm{d}s
		\leq \omega_\M(t+2)\int_{-1}^1\chi(s)\,\mathrm{d}(s)\\
		& \leq \left(\omega_\M(2t)+\omega_\M(4)\right)\int_{-1}^1 \chi(s)\,\mathrm{d}s
		\leq C_1\omega_\M(t)+C_2,
	\end{align*}
	for all $t\geq 0$. Hence $\omega_0\asymp \omega_\M$.

	\noindent\textit{(iv)} Finally, by the properties of the convolution, condition $(W2)$, and the fact that $\chi\in\E_{[\M]}(\R)$, we have that there exist $C_1,h>0$ (respectively, for every $h>0$, there exists $C_1>0$) such that
	\begin{align*}
		|\partial_t^k\omega_0(t)|\leq\ & \int_{-1}^1 \omega_\M(t+1-s)|\partial_s^k\chi(s)|\,\mathrm{d}s
		\leq \omega_\M(t+1)Ch^kM_k\\
		& \leq C_1h^{k}M_k(\omega_\M(t)+\omega_\M(1))
		\leq Ch^kM_k \omega_0(t),
	\end{align*}
	for $t\geq\delta$ and $C>0$ not depending on $k$. Here, we used $\omega_\M\asymp\omega_0$. Finally, since $\omega_0(t)$ is bounded on compact subsets (it is smooth by the regularizing property of the convolution), the previous estimate yields $\omega_0\in\E_{[\M]}((0,+\infty))$. The proof is complete.	
\end{proof}

\begin{remark}
	Assuming that conditions $(M2)$ and $(M4)$ hold, we can assume that the weight function $\omega_\M(t)$ has the properties $(i)-(iii)$ of Proposition \ref{prop_wf_E}. In particular, for each $\varepsilon>0$, the functions $e^{\omega_\M(\varepsilon t)}$ and $e^{-\omega_\M(\varepsilon t)}$ belong to $\E_{[\M]}((0,+\infty))$.
\end{remark}

\subsection{Some technical results}\label{sec:app}
For the convenience of the reader, we recall here below a few results,
that we employ in some of the proofs.

Given $k\in\N_0$, the set $\Delta(k)$ of multi-indices 
$\gamma\in \mathbb{N}_0^k $ is defined by
\begin{equation}\label{Delta_k}
	\Delta(k) = \left\{\gamma=(\gamma_1,\dots,\gamma_k)\in\mathbb{N}_0^k \, : \, \sum_{\ell=1}^{k}\ell\gamma_\ell=k\right\}.
\end{equation}

\begin{lemma}[Fa\`a di Bruno Formula]\label{faa}
	Given $f\in C^\infty(\mathbb{R})$ and $k\in\mathbb{N}$, we have that
	\begin{equation*}
		\dfrac{d^k}{dt^k}e^{f(t)}=
		e^{f(t)}\displaystyle\sum_{\gamma\in\Delta(k)}\dfrac{k!}{\gamma!}\prod_{\ell=1}^{k}\left(\dfrac{1}{\ell!}\dfrac{d^\ell}{dt^\ell}f(t)\right)^{\gamma_\ell}.
	\end{equation*}
\end{lemma}

\begin{lemma}\label{lemma_sum}
	Given $k\in\mathbb{N}_0$ and $R>0$, we have that
	\begin{equation*} 
		\sum_{\gamma\in\Delta(k)}\dfrac{|\gamma|!}{\gamma!}R^{|\gamma|}=R(1+R)^{k-1},
	\end{equation*}
	where $|\gamma|=\gamma_1+\cdots+\gamma_k$.
\end{lemma}
\begin{proof}
	See \cite[Lemma 1.4.1]{krantz}.
\end{proof}

\begin{lemma}\label{lemma_prod_M}
	Let $\M=\Mj$ be a sequence of positive numbers satisfying $(M1)$. For each $\gamma\in\Delta(k)$, we have
	\begin{equation*}M_{|\gamma|}M_1^{\gamma_1}\cdots M_k^{\gamma_k} \leq M_k.\end{equation*}
\end{lemma}
\begin{proof}
	See \cite[Proposition 4.4]{bier}.
\end{proof}

	
\section{A global calculus of SG-symbols of ultradifferential type on $\mathbb{R}^n$}\label{sec:SGABRn}

Let $\A=\Aj$ and $\B=\Bj$ be two non-quasianalytic weight sequences. Given $m,\mu\in\R$ and $h>0$, we denote by $\SG_{\A,\B}^{m,\mu}(\R^n;h)$ the space of all $p\in C^\infty(\R^{2n})$ such that
\[\|p\|_h = \sup_{\alpha,\beta\in\N_0^n}\sup_{x,\xi\in\R^n} h^{|\alpha|+|\beta|} \frac{\langle \xi \rangle^{-\mu+|\alpha|}\langle x\rangle^{-m+|\beta|}}{A_{|\alpha|} B_{|\beta|}}|\partial_x^\beta\partial_\xi^\alpha p(x,\xi)|<+\infty,\]
which is a Banach space with the norm $\|\cdot\|_h$. Then, we define
\[{\SG}_{\{\A,\B\}}^{m,\mu}(\R^n) = \underset{h>0}{\mathrm{ind\,lim}}\ {\SG}_{\A,\B}^{m,\mu}(\R^n;h),
\quad 
{\SG}_{(\A,\B)}^{m,\mu}(\R^n) = \underset{h>0}{\mathrm{proj\,lim}}\ {\SG}_{\A,\B}^{m,\mu}(\R^n;h).\]

In a completely similar fashion, it is possible to define the wider symbol class $\ST_{[\A,\B]}(\R^n)$ (cf. \cite[Ch. 1, \S 1.1]{cordes}, of which 
${\SG}_{[\A,\B]}^{m,\mu}(\R^n)$ is a subclass, see Section \ref{subs:STclasses} below.

\begin{theorem}\label{cont_SG}
	Let $m,\mu\in\R$. Then, a symbol $p\in {\SG}_{[\A,\B]}^{m,\mu}(\R^n)$ defines a continuous operator $\mathrm{Op}(p):\mathcal{S}_{[\M]}(\R^n)\to \mathcal{S}_{[\M]}(\R^n)$ via the usual left-quantization
	\begin{equation}\label{Op_int} 
	\mathrm{Op}(p) u(x) = \int_{\R^n} e^{i\xi\cdot x}p(x,\xi)\widehat{u}(\xi)\,\dbar \xi,\quad u \in \mathcal{S}_{[\M]}(\R^n),
	\end{equation}
	where $\dbar\xi=(2\pi)^{-\frac{n}{2}}\mathrm{d}\xi$.
\end{theorem}
\begin{proof}
	Let $\mathscr{B}\subset\caS_{\{\M\}}(\R^n)$ be a bounded subset, that is, there exist $C,h>0$ such that
	\[\sup_{x\in\R^n} |x^\beta\partial_x^\alpha u(x)| \leq Ch^{|\alpha|+|\beta|}M_{|\alpha|}M_{|\beta|},\]
	for all $\alpha,\beta\in\N_0^n$ and $u\in\mathscr{B}$. Then, given $u\in \mathscr{B}$, Leibniz rule and integration by parts yield
	\begin{align*}
		x^\beta D_x^\alpha[\mathrm{Op}(p)u](x) & = x^\beta\sum_{\gamma\leq\alpha}\binom{\alpha}{\gamma}\int_{\R^n} (D_x^{\gamma}e^{i\xi\cdot x})(D_x^{\alpha-\gamma}p(x,\xi))\widehat{u}(\xi)\,\dbar\xi\\
		& = \sum_{\gamma\leq\alpha}\binom{\alpha}{\gamma}\int_{\R^n} (D_\xi^{\beta}e^{i\xi\cdot x})\xi^{\gamma}(D_x^{\alpha-\gamma}p(x,\xi))\widehat{u}(\xi)\,\dbar\xi\\
		& = (-1)^{|\alpha|}\sum_{\gamma\leq\alpha}\binom{\alpha}{\gamma} \sum_{\delta\leq\beta}\binom{\beta}{\delta}\int_{\R^n} e^{i\xi\cdot x} (D_x^{\alpha-\gamma}D_\xi^{\beta-\delta}p(x,\xi)) D_\xi^{\delta}(\xi^\gamma\widehat{u}(\xi))\,\dbar\xi.
	\end{align*}
	
	Now, notice that, for every $N\in\N_0$, we have
	\begin{align*}
		\int_{\R^n} e^{i\xi\cdot x} (D_x^{\alpha-\gamma}D_\xi^{\beta-\delta}&p(x,\xi)) D_\xi^{\delta}(\xi^\gamma\widehat{u}(\xi))\,\dbar\xi\\
		=\ & \langle x\rangle^{-2N}\int_{\R^n} e^{i\xi\cdot x}(1-\Delta_\xi)^N[(D_x^{\alpha-\gamma}D_\xi^{\beta-\delta}p(x,\xi)) D_\xi^{\delta}(\xi^\gamma\widehat{u}(\xi))]\,\dbar\xi.
	\end{align*}
	
	Given $\ell\in\N_0$, we set
	\[\Pi(\ell)=\{\lambda=(\lambda_1,\dots,\lambda_n)\in\N_0^n\,:\, \lambda_1+\cdots+\lambda_n = \ell\}.\]
	By the Lebniz rule and the multinomial theorem, given $f,g\in C^\infty(\R^n)$, we have
	\begin{align*}
		(1-\Delta_\xi)^N(fg) & = \sum_{\ell=0}^N\binom{N}{\ell}(-1)^\ell\Delta_\xi^\ell(fg)
		 = \sum_{\ell=0}^N\binom{N}{\ell}(-1)^\ell (\partial_{\xi_1}^2+\cdots+\partial_{\xi_n}^2)^\ell(fg)\\
		&= \sum_{\ell=0}^N\binom{N}{\ell}(-1)^\ell \sum_{\lambda\in\Pi(\ell)}\dfrac{\ell!}{\lambda_1!\cdots\lambda_n!}\partial_\xi^{2\lambda}(fg)\\
		& = \sum_{\ell=0}^N\binom{N}{\ell}(-1)^\ell \sum_{\lambda\in\Pi(\ell)}\dfrac{\ell!}{\lambda_1!\cdots\lambda_n!} \sum_{\nu\leq 2\lambda}\binom{2\lambda}{\nu} (\partial_{\xi}^{\nu}f)(\partial_{\xi}^{2\lambda-\nu} g).
	\end{align*}
	Hence, for any $N\in\N_0$, we obtain
	\begin{equation*}
		x^\beta\partial_x^\alpha[\mathrm{Op}(p)u](x) = \sum \int_{\R^n} e^{i\xi\cdot x}\langle\xi\rangle^{-2N} (D_x^{\alpha-\gamma}\partial_\xi^{2\lambda-\nu}D_\xi^{\beta-\delta}p(x,\xi))(\langle\xi\rangle^{2N}\partial_\xi^\nu D_\xi^\delta (\xi^\gamma\widehat{u}(\xi)))\,\dbar\xi,
	\end{equation*}
	where the summation symbol above denotes
	\[(-1)^{|\alpha|}\sum_{\gamma\leq\alpha}\binom{\alpha}{\gamma} \sum_{\delta\leq\beta}\binom{\beta}{\delta}\sum_{\ell=0}^N\binom{N}{\ell}(-1)^\ell \sum_{\lambda\in\Pi(\ell)}\dfrac{\ell!}{\lambda_1!\cdots\lambda_n!} \sum_{\nu\leq 2\lambda}\binom{2\lambda}{\nu}.\]
	By the multinomial theorem, one can easily check that
	\begin{equation}\label{bound_sum_binom}
		\sum_{\gamma\leq\alpha}\binom{\alpha}{\gamma} \sum_{\delta\leq\beta}\binom{\beta}{\delta}\sum_{\ell=0}^N\binom{N}{\ell} \sum_{\lambda\in\Pi(\ell)}\dfrac{\ell!}{\lambda_1!\cdots\lambda_n!} \sum_{\nu\leq 2\lambda}\binom{2\lambda}{\nu} \leq (8n)^N 2^{|\alpha|+|\beta|}.
	\end{equation}
	Since $p\in {\SG}_{\{\A,\B\}}^{m,\mu}(\R^n)$ and $\A,\B\preceq\M$, there exist $C_p,h_p>0$ such that
	\begin{equation}\label{est_p}
		|D_x^{\alpha-\gamma}\partial_\xi^{2\lambda-\nu}D_\xi^{\beta-\delta}p(x,\xi)| \leq C_ph_p^{|\alpha-\gamma|+|\beta-\delta|}h_p^{2N}M_{|\alpha-\gamma|}M_{|\beta-\delta|}M_{2N}\langle\xi\rangle^{\mu}\langle x\rangle^{m}.
	\end{equation}
	Combining Lemmas \ref{est_GS_bracket} and \ref{lemma_poly_GS}, we obtain constants $C_N,h_N>0$ depending on $N$ such that
	\begin{equation}\label{est_xi_u}
		|\langle\xi\rangle^{2N}\partial_\xi^\nu D_\xi^\delta (\xi^\gamma\widehat{u}(\xi))| \leq C_Nh_N^{|\gamma|+|\delta|}M_{|\gamma|}M_{|\delta|}.
	\end{equation}
	It follows, choosing $N\in\N_0$ such that $2N\geq \max\{\mu+n,m\}$,
	$\langle x\rangle^{m-2N}\leq 1$ and
	\begin{equation}\label{const_int}
	\tilde C = \int_{\R^n}\langle\xi\rangle^{\mu-2N}\,\dbar\xi<+\infty.
	\end{equation}
	
	Then, by \eqref{bound_sum_binom}, \eqref{est_p}, and \eqref{est_xi_u}, we obtain
	\[|x^\beta\partial_x^\alpha u(x)| \leq C_0h_0^{|\alpha|+|\beta|}M_{|\alpha|}M_{|\beta|},\]
	for all $x\in\R^n$, $\alpha,\beta\in\N_0^n$, and $u\in\mathscr{B}$, where the constants
	\[C_0 = \tilde{C}C_NC_p(8h_p^{2}n)^NM_{2N}\quad\text{and}\quad h_0 = 2\max\{1,h_N,h_p\}\]
	do not depend on $u\in\mathscr{B}$. Hence, $\mathrm{Op}(p)$ is bounded
	from $\caS_{\{\M\}}(\R^n)$ to itself. The projective case is analogous. 
\end{proof}

The next Corollary \ref{cont_SG_ext} is straightforward.

\begin{corollary}\label{cont_SG_ext}
	Given $p\in \SG_{[\A,\B]}^{m,\mu}(\mathbb{R}^n)$, $\mathrm{Op}(p):\mathcal{S}_{[\M]}(\mathbb{R}^n)\to \mathcal{S}_{[\M]}(\mathbb{R}^n)$ extends to a continuous linear operator from $\mathcal{S}_{[\M]}'(\mathbb{R}^n)$ to itself.
\end{corollary}
By the continuity of $\Op(p)$, we can associate to it a ultradistributional kernel $K_p\in\caS_{[\M]}'(\R^{2n})$, which is given by the oscillatory integral
\[K_p(x,y) = \int_{\R^n} e^{i\xi\cdot(x-y)}p(x,\xi)\,\dbar\xi.\]

\begin{lemma}\label{lemma_partition}
	Let $\M=\Mj$ be a weight sequence satisfying $(M_0)$, $(M_1)$, and $(M_3)'$, and $R>2$ an arbitrary constant. Then, there exists a sequence $\{\psi_N\}_{N\in\N_0}$ in $\mathcal{D}_{[\M]}(\R^n)$ such that
	\begin{equation*}
		\sum_{N\in\N_0}\psi_N = 1,\quad \mathrm{supp}(\psi_0) \subset\{\xi\in\R^n\,:\, \langle\xi\rangle < 3R\},
	\end{equation*}
	\begin{equation*}
		\mathrm{supp}(\psi_N)\subset\{\xi\in\R^n\,:\, 2Rm_N < \langle\xi\rangle < 3Rm_{N+1}\},
	\end{equation*}
	for $N\geq 1$, and there exist $C,h>0$ (respectively, for every $h>0$, there exists $C>0$) such that
	\begin{equation*}
		|D_\xi^\alpha\psi_0(\xi)|\leq C\left(\frac{h}{R}\right)^{|\alpha|}M_{|\alpha|} \quad \text{and}\quad |D_\xi^\alpha\psi_0(\xi)|\leq C\left(\frac{h}{Rm_N}\right)^{|\alpha|}M_{|\alpha|},
	\end{equation*}
	for all $N\in\N$, $\xi\in\R^n$, and $\alpha\in\N_0^n$.
\end{lemma}
\begin{proof}
	See \cite[Lemma 6]{Pran2013}.
\end{proof}

Given $k\in (0,1)$, set
\begin{equation*}
	\Omega_k = \{(x,y)\in\R^{2n}\,:\,|x-y|>k\langle x\rangle\}.
\end{equation*}

\begin{proposition}\label{est_kernel}
	Let $p\in {\SG}_{[\A,\B]}^{m,\mu}(\R^n)$. Then, for every $k\in(0,1)$, the kernel $K_p$ of $\Op(p)$ is smooth on $\Omega_k$ and there exist $C,h,\varepsilon>0$ (respectively, for every $h,\varepsilon>0$, there exists $C>0$) such that
	\begin{equation*}
		|D_x^\beta D_y^\gamma K_p(x,y)|\leq Ch^{|\beta|+|\gamma|}B_{|\beta|}B_{|\gamma|}e^{-\omega_{\M}(\varepsilon|x|)}e^{-\omega_{\M}(\varepsilon|y|)},
	\end{equation*}
	for all $\beta,\gamma\in\N_0^n$ and $(x,y)\in\overline{\Omega}_k$.
\end{proposition}
\begin{proof}
	For each $N\in\N_0$, set
	\[K_N(x,y)=\int_{R^n}e^{i\xi\cdot(x-y)}p(x,\xi)\psi_N(\xi)\,\dbar\xi,\]
	where the $\psi_N$ are given by Lemma \ref{lemma_partition}. In this case, we obtain
	\[K_p(x,y) = \sum_{N\in\N_0}K_N(x,y).\]
	
	Let $k\in(0,1)$ and $(x,y)\in\overline{\Omega}_k$ and consider $\ell\in\{1,\dots,n\}$ such that
	\begin{equation}\label{x_ell_y_ell}
		|x_\ell- y_\ell| \geq \frac{k}{n\langle x\rangle},
	\end{equation}
	where $x_\ell$ and $y_\ell$ denote the $\ell$-th coordinate of $x$ and $y$, respectively. Then, by the Leibniz rule and integration by parts, we obtain
	\begin{align*}
		D_x^\beta D_y^\gamma K_N(x,y) & = (-1)^{|\gamma|}\sum_{\delta\leq\beta}\binom{\beta}{\delta}\int_{\R^n}e^{i\xi\cdot (x-y)}\xi^{\delta+\gamma}\psi_N(\xi) D_x^{\beta-\delta}p(x,\xi)\,\dbar\xi\\
		& = (-1)^{|\gamma|+N}(x_\ell-y_\ell)^{-N}\sum_{\delta\leq\beta}\binom{\beta}{\delta}\int_{\R^n}e^{i\xi\cdot (x-y)}D_{\xi_\ell}^N\left(\xi^{\delta+\gamma}\psi_N(\xi) D_x^{\beta-\delta}p(x,\xi)\right)\,\dbar\xi.
	\end{align*}
	
	Given $\rho>0$, we define the operator
	\begin{equation*}
		\mathcal{L}_\rho = \frac{1}{\lambda_\rho(\langle x-y\rangle)}\sum_{j\in\N_0}\frac{\rho^{2j}}{M_{j}^2}(1-\Delta_\xi)^j,
	\end{equation*}
	where
	\[\lambda_\rho(t) = \sum_{j\in\N_0}\frac{\rho^{2j}}{M_{2j}} t^{2j},\quad t > 0.\]
	
	By \eqref{weight_seq}, we have
	\begin{equation}\label{est_lambda_rho}
		e^{\omega_{\M}(\rho t)} \leq \lambda_\rho(t),
	\end{equation}
	for all $\rho,t>0$.	
	
	It is not difficult to check that $\mathcal{L}_\rho e^{i\xi(x-y)} = e^{i\xi(x-y)}$. Hence, integrating by parts, we obtain
	\begin{align*}
		D_x^\beta D_y^\gamma K_N(x,y) =\ & \frac{(-1)^{|\gamma|+N}(x_\ell-y_\ell)^{-N}}{\lambda_\rho(x,y)}\sum_{\delta\leq\beta}\binom{\beta}{\delta} \sum_{j\in\N_0}\frac{\rho^{2j}}{M_j^2}\\
		& \times \int_{\R^n}e^{i\xi\cdot (x-y)}(1-\Delta_\xi)^j D_{\xi_\ell}^N\left(\xi^{\delta+\gamma}\psi_N(\xi) D_x^{\beta-\delta}p(x,\xi)\right)\,\dbar\xi.
	\end{align*}
	
	Proceeding as in the proof of Theorem \ref{cont_SG}, we can apply the binomial and multinomial theorems to expand $(1-\Delta_\xi)^j$, which yields
	\begin{align*}
		\left| (1-\Delta_\xi)^j D_{\xi_\ell}^N\left(\xi^{\delta+\gamma}\psi_N(\xi) D_x^{\beta-\delta}p(x,\xi)\right) \right| \leq\ & \sum_{j'=0}^j\binom{j}{j'}\sum_{\lambda\in\Pi(j')}\frac{j'!}{\lambda_1!\cdots\lambda_n!}\sum_{\nu\leq 2\lambda}\binom{2\lambda}{\nu}\sum_{N'=0}^N\binom{N}{N'}\\
		& \times |\partial_\xi^\nu\partial_{\xi_\ell}^{N'}(\xi^{\gamma+\delta}\psi_N(\xi))||\partial_\xi^{2\lambda-\nu}\partial_{\xi_\ell}^{N-N'}\partial_{x}^{\beta-\delta}p(x,\xi)|.
	\end{align*}
	
	Now, by Lemmas \ref{lemma_poly_GS} and \ref{lemma_partition} and the fact that $\mathcal{D}_{[\M]}(\R^n)\subset\caS_{[\M]}(\R^n)$, given $R>0$, there exist $C_\psi,h_\psi>0$ (respectively, for every $h_\psi>0$, there exists $C_\psi>0$) such that
	\begin{equation*}
		|\partial_\xi^\nu\partial_{\xi_\ell}^{N'}(\xi^{\gamma+\delta}\psi_N(\xi))| \leq C_\psi h_\psi^{|\gamma|+|\delta|} M_{|\nu|}M_{N'}M_{|\gamma|+|\delta|}\left(\frac{h_\psi}{Rm_N}\right)^{|\nu|+N'}.
	\end{equation*}
	
	Also, since $p\in {\SG}_{[\A,\B]}^{m,\mu}(\R^n)$ and $\A,\B\preceq \M$, there exist $C_p,h_p>0$ (respectively, for every $h_p>0$, there exists $C_p>0$) such that
	\begin{align*}
		|\partial_\xi^{2\lambda-\nu}\partial_{\xi_\ell}^{N-N'}\partial_{x}^{\beta-\delta}p(x,\xi)| \leq\ & C_ph_p^{|2\lambda|-|\nu|+N-N'+|\beta|-|\delta|}M_{|2\lambda|-|\nu|}M_{N-N'}M_{|\beta-\delta|}\\
		& \times \langle\xi\rangle^{\mu-|2\lambda-\nu|-N+N'} \langle x\rangle^{m-|\beta-\delta|}.
	\end{align*}
	
	We have the following bound for the binomial and multinomial sums:
	\begin{equation}
		\sum_{j'=0}^j\binom{j}{j'}\sum_{\lambda\in\Pi(j')}\frac{j'!}{\lambda_1!\cdots\lambda_n!}\sum_{\nu\leq 2\lambda}\binom{2\lambda}{\nu}\sum_{N'=0}^N\binom{N}{N'} \leq 8^N(2n)^j.
	\end{equation}
	
	Combining the previous estimates, we obtain constants $C_1,h_1>0$ (respectively, for every $h_1>0$, we obtain $C_1>0$) such that
	\begin{align*}
		\left| (1-\Delta_\xi)^j D_{\xi_\ell}^N\left(\xi^{\delta+\gamma}\psi_N(\xi) D_x^{\beta-\delta}p(x,\xi)\right) \right| \leq\ & C_1h_1^{2j+N+|\beta|+|\gamma|}M_{j}^2M_NM_{|\beta|}M_{|\gamma|}\left(\frac{1}{Rm_N}\right)^{|\nu|-N'}\\
		& \times \langle\xi\rangle^{\mu - |2\lambda-\nu|-(N-N')} \langle x\rangle^{m}.
	\end{align*}
	
	Now, since $2Rm_N < \langle\xi\rangle < 3Rm_{N+1}$ on $\mathrm{supp}(\psi_N)$, if $\mu> 0$, we obtain
	\begin{equation*}
		\langle\xi\rangle^{\mu - |2\lambda-\nu|-(N-N')} \leq (3Rm_{N+1})^{\mu}\left(\frac{1}{2Rm_N}\right)^{|2\lambda-\nu|+(N-N')},
	\end{equation*}
	and
	\[(3Rm_{N+1})^{\mu} \leq (3R)^{\mu}\tilde{C}(N+1)^{\mu} \leq (3R)^{\mu}\tilde{C}_{\mu}e^N,\]
	where $\tilde{C}>0$ do not depend on $N$ and $\tilde{C}_{\mu}$ depends only on $\tilde{C}$ and $\mu$.
	
	If $\mu\leq 0$, $(3Rm_{N+1})^{\mu}$ is bounded by a constant $\tilde{C}>0$ independent of $N$. Hence, we obtain new constants $C_2,h_2>0$ (respectively, for every $h_2>0$, there exists $C_2>0$) such that
	\begin{align*}
		\left| (1-\Delta_\xi)^j D_{\xi_\ell}^N\left(\xi^{\delta+\gamma}\psi_N(\xi) D_x^{\beta-\delta}p(x,\xi)\right) \right| \leq\ & C_2 h_2^{|\beta|+|\gamma|}M_{|\beta|}M_{|\gamma|}h_2^{2j}M_{j}^2M_N \langle x\rangle^{m}  \left(\frac{h_2}{Rm_N}\right)^{|2\lambda|+N}\\
		\leq\ & C_2 h_2^{|\beta|+|\gamma|}M_{|\beta|}M_{|\gamma|}h_2^{2j}M_{j}^2 \langle x\rangle^{m} \left(\frac{h_2}{R}\right)^{|2\lambda|+N},\\
		\leq\ & C_2 h_2^{|\beta|+|\gamma|}M_{|\beta|}M_{|\gamma|}h_2^{2j}M_{j}^2 \langle x\rangle^{m},
	\end{align*}
	for a sufficiently big $R>0$, and we used Lemma \ref{lemma_mjj_Mj} in the second inequality.
	
	Hence, by the previous estimates and \eqref{x_ell_y_ell} and \eqref{est_lambda_rho},  we obtain $C,h>0$ (respectively, for every $h>0$, we obtain $C>0$) such that
	\begin{align*}
		|D_x^\beta D_y^\gamma K_N(x,y)| \leq\ & \frac{|x_\ell-y_\ell|^{-N}}{\lambda_\rho(\langle x-y\rangle)}\sum_{\delta\leq\beta}\binom{\beta}{\delta}\sum_{j\in\N_0}(\rho h_2)^{2j} C_0h_2^{|\beta|+|\gamma|}M_{|\beta|}M_{|\gamma|}\langle x\rangle^{m}\\
		\leq & Ch^{|\beta|+|\gamma|}M_{|\beta|}M_{|\gamma|}\langle x\rangle^{m-N}e^{-\omega_{\M}(\rho\langle x-y\rangle)},
	\end{align*}
	where we choose $\rho>0$ small enough such that the series $\sum_{j\in\N_0}(\rho h_2)^{2j}$ is convergent. 
	
	Finally, since $(x,y)\in\overline{\Omega}_k$, we obtain $C'>0$ such that $\langle x-y\rangle \geq C'(\langle x\rangle+\langle y\rangle) \geq C'(|x|+|y|)$. Then, for such $(x,y)$ and the properties of $\omega_{\M}$, we obtain
	\[\langle x\rangle^{m-N}e^{-\omega_{\M}(\rho\langle x-y\rangle)} \leq e^{-\omega_{\M}(\rho'|x|)}e^{-\omega_{\M}(\rho'|y|)},\]
	where $\rho'>0$ is chosen in such a way that the exponential absorbs the polynomial $\langle x\rangle^{m}$. Such $\rho'$ can be obtained by Propositions 1 and 2 of \cite{KMR_2022}. The proof is complete.	
\end{proof}

\begin{definition}
	We say that a continuous linear operator $P:\caS_{[\M]}'(\R^n)\to \caS_{[\M]}'(\R^n)$ is $[\M]$-pseudolocal if
	\begin{equation*}
		[\M]\text{-\,}\mathrm{sing\,supp} (Pu) \subset [\M]\text{-\,}\mathrm{sing\,supp}(u),
	\end{equation*}
	for all $u\in \caS_{[\M]}'(\R^n)$.
\end{definition}

As a consequence of the kernel estimate given by Proposition \ref{est_kernel}, we can obtain that, for every $p\in {\SG}_{[\A,\B]}^{m,\mu}(\R^n)$, the operator $\mathrm{Op}(p)$ is $[\M]$-pseudolocal.

\subsection{Ultradifferential symbolic calculus}

Given $t_1,t_2\geq 0$, we set
\begin{equation*}
	Q_{t_1,t_2}=\{(x,\xi)\in\R^{2n}\,:\, \langle \xi\rangle < t_1,\ \langle x\rangle<t_2\}\quad\text{and}\quad Q_{t_1,t_2}^c = \R^{2n}\setminus Q_{t_1,t_2}.
\end{equation*}

If $t_1=t_2=t$, we write simply $Q_{t_1,t_2}=Q_t$ and $Q_{t_1,t_2}^c = Q_t^c$.

\begin{definition}
	We define $\FS_{\{\A,\B\}}^{m,\mu}(\R^n)$ (respectively, $\FS_{(\A,\B)}^{m,\mu}(\R^n)$) as the space of formal sums $\sum_{j\in\N_0}p_j$, where each $p_j(x,\xi)$ is smooth in the interior of $Q_{Bm_j}^c$, for some $B>0$, and there exists $h>0$ such that (respectively, for every $h>0$, we have)
	\begin{equation*}
		\sup_{j\in\N_0}\sup_{\alpha,\beta\in\N_0^n}\sup_{(x,\xi)\in Q_{B(j)}^c} \frac{\langle\xi\rangle^{-\mu+j+|\alpha|}\langle x\rangle^{-m+j+|\beta|}|\partial_\xi^\alpha\partial_x^\beta p_j(x,\xi)|}{h^{|\alpha|+|\beta|+2j} A_{|\alpha|}A_jB_{|\beta|}B_j} < +\infty,
	\end{equation*}
	where $Q_{B(j)}^c = Q_{Ba_j,Bb_j}^c$, with $a_j=A_j/A_{j-1}$, $b_j=B_j/B_{j-1}$, for $j\geq 1$, and $a_0=b_0=0$.
\end{definition}

\begin{definition}\label{def_equiv_FS}
	We say that two formal sums $\sum_{j\in\N_0}p_j, \sum_{j\in\N_0}\tilde{p}_j \in \FS_{[\A,\B]}^{m,\mu}(\R^n)$ are equivalent if there exist $B,h>0$ such that (respectively, there exists $B>0$ such that, for every $h>0$, we have)
	\begin{equation*}
		\sup_{N\in\N_0}\sup_{\alpha,\beta\in\N_0^n}\sup_{(x,\xi)\in Q_{B(N)}^c} \frac{\langle\xi\rangle^{-\mu+N+|\alpha|}\langle x\rangle^{-m+N+|\beta|} \left| \sum_{j<N}\partial_\xi^\alpha\partial_x^\beta(p_j(x,\xi)-\tilde{p}_j(x,\xi))\right|}{h^{|\alpha|+|\beta|+2N} A_{|\alpha|}A_NB_{|\beta|}B_N} < +\infty.
	\end{equation*}
	In this case, we write $\sum_{j\in\N_0}p_j\sim\sum_{j\in\N_0}\tilde{p}_j$.
\end{definition}

Every $p\in {\SG}_{[\A,\B]}^{m,\mu}(\R^n)$ can be regarded as an element of $\FS_{[\A,\B]}^{m,\mu}(\R^n)$ by setting 
$p_0=p$ and $p_j=0$, $j\geq 1$. Conversely, we have the following Proposition \ref{prop_p_pj}.

\begin{proposition}\label{prop_p_pj}
	Let $\sum_{j\in\N_0}p_j\in \FS_{[\A,\B]}^{m,\mu}(\R^n)$. Then, there exists $p\in {\SG}_{[\A,\B]}^{m,\mu}(\R^n)$ such that $p\sim \sum_{j\in\N_0}p_j$ in $\FS_{[\A,\B]}^{m,\mu}(\R^n)$.
\end{proposition}
\begin{proof}
	Let $\varphi_1\in\mathcal{D}_{[\A]}(\R^n)$ and $\varphi_2\in\mathcal{D}_{[\B]}(\R^n)$ be such that $0\leq\varphi_\ell\leq 1$, $\varphi_\ell(t)\equiv 1$ if $\langle t\rangle \leq 2$, and $\varphi_\ell(t)\equiv 0$ if $\langle t\rangle \geq 3$, for $\ell=1,2$. Set $\chi(x,\xi)=1-\varphi_1(\xi)\varphi_2(x)$. Then, for $j\in\N$ and $R>0$, we define
	\[\chi_j(x,\xi)=\chi\left(\frac{x}{Rb_j},\frac{\xi}{Ra_j}\right),\]
	and put $\chi_0(x,\xi)\equiv 1$. By definition, there exist $C_\chi,h_\chi>0$ (respectively, for every $h_\chi>0$, there exists $C_\chi>0$) such that
	\begin{equation}\label{der_chi_j}
		\sup_{(x,\xi)\in\R^{2n}}|D_\xi^\alpha D_x^\beta \chi_j(x,\xi)| \leq \frac{C_\chi h_\chi^{|\alpha|+|\beta|}A_{|\alpha|}B_{|\beta|}}{R^{|\alpha|+|\beta|}a_j^{|\alpha|}b_j^{|\beta|}},
	\end{equation}
	for all $\alpha,\beta\in\N_0^n$ and $j\in\N_0$.
	
	Let us show that
	\begin{equation}\label{p_chi_pj}
	p(x,\xi) = \sum_{j\in\N_0}\chi_j(x,\xi)p_j(x,\xi)
	\end{equation} 
	belongs to ${\SG}_{[\A,\B]}^{m,\mu}(\R^n)$ for a sufficiently big $R>0$, and $p\sim \sum_{j\in\N_0}p_j$ in $\FS_{[\A,\B]}^{m,\mu}(\R^n)$.
	
	Given $\alpha,\beta\in\N_0^n$, differentiating \eqref{p_chi_pj} yields
	\begin{equation*}
		D_\xi^\alpha D_\xi^\beta p(x,\xi) = \sum_{j\in\N_0} \sum_{\gamma\leq\alpha}\binom{\alpha}{\gamma} \sum_{\delta\leq\beta}\binom{\beta}{\delta} D_\xi^{\alpha-\gamma}D_x^{\beta-\delta}p_j(x,\xi) D_\xi^\gamma D_x^\delta \chi_j(x,\xi).
	\end{equation*}
	
	Let us consider the inductive case. Since $\sum_{j\in\N_0} p_j \in \FS_{\{\A,\B\}}^{m,\mu}(\R^n)$, there exist $C,h>0$ such that
	\begin{equation}
		|D_\xi^{\alpha-\gamma}D_x^{\beta-\delta}p_j(x,\xi)| \leq Ch^{|\alpha-\gamma|+|\beta-\delta|+2j}A_{|\alpha-\gamma|}B_{|\beta-\delta|}A_jB_j\langle\xi\rangle^{m_\xi-|\alpha-\gamma|-j}\langle x\rangle^{m-|\beta-\delta|-j},
	\end{equation}
	for $(x,\xi)\in Q_{B(j)}^c$. Then, taking $R>B$, we obtain
	\begin{align*}
		|D_\xi^\alpha D_x^\beta p(x,\xi)| & \leq  C\langle\xi\rangle^{\mu-|\alpha|}\langle x\rangle^{m-|\beta|}\sum_{j\in\N_0}\sum_{\gamma\leq\alpha}\binom{\alpha}{\gamma} \sum_{\delta\leq\beta}\binom{\beta}{\delta} h^{|\alpha-\gamma|+|\beta-\delta|+2j}\langle\xi\rangle^{|\gamma|-j}\langle x\rangle^{|\delta|-j}\\
		& \times A_{|\alpha-\gamma|}B_{|\beta-\delta|}A_jB_j |D_\xi^\gamma D_x^\delta\chi_j(x,\xi)|\\
		& \leq Ch^{|\alpha|+|\beta|} A_{|\alpha|}B_{|\beta|} \langle\xi\rangle^{\mu-|\alpha|}\langle x\rangle^{m-|\beta|}  \sum_{j\in\N_0} \frac{h^{2j}A_jB_j}{\langle\xi\rangle^{j}\langle x\rangle^{j}}|\chi_j(x,\xi)|\\
		& + C \langle\xi\rangle^{\mu-|\alpha|}\langle x\rangle^{m-|\beta|}\sum_{j\in\N_0}\sum_{(\gamma,\delta)\neq(0,0)}\binom{\alpha}{\gamma}\binom{\beta}{\delta}h^{|\alpha-\gamma|+|\beta-\delta|+2j}A_{|\alpha-\gamma|}B_{|\beta-\delta|}\\
		& \times \frac{\langle\xi\rangle^{|\gamma|}\langle x\rangle^{|\delta|} A_jB_j}{\langle\xi\rangle^j\langle x\rangle^j}|D_\xi^\gamma D_x^\delta\chi_j(x,\xi)|\\
		& = S_1 + S_2.
	\end{align*}
	
	To estimate $S_1$, notice that $\langle\xi\rangle\langle x\rangle \geq R^2a_jb_j$ on $\mathrm{supp}(\chi_j)$. Hence
	\begin{equation*}
		\sum_{j\in\N_0} \frac{h^{2j}A_jB_j}{\langle\xi\rangle^{j}\langle x\rangle^{j}}|\chi_j(x,\xi)| \leq \sum_{j\in\N_0} \frac{h^{2j}A_jB_j}{R^{2j}a_j^jb_j^j} \leq \sum_{j\in\N_0}\frac{h^{2j}}{R^{2j}} = C_{R} < +\infty,
	\end{equation*}
	for a large enough $R>0$, where we used Lemma \ref{lemma_mjj_Mj}. It follows
	\[S_1 \leq CC_Rh^{|\alpha|+|\beta|}A_{|\alpha|}B_{|\beta|}\langle\xi\rangle^{\mu-|\alpha|}\langle x\rangle^{m-|\beta|},\]
	for all $\alpha,\beta\in\N_0^n$.
	
	Now, for $S_2$, since $(\gamma,\delta)\neq (0,0)$, observe that $D_\xi^\delta D_x^\gamma\chi_j\equiv 0$ on $Q_{3R(j)}^c$. For $(x,\xi)\in Q_{3R(j)}$, we have $\langle\xi\rangle^{|\gamma|}\langle x\rangle^{|\delta|}\leq (3R)^{|\gamma|+|\delta|}a_j^{|\gamma|}b_j^{|\delta|}$. Hence, using \eqref{der_chi_j}, we obtain
	\begin{align*}
		& h^{|\alpha-\gamma|+|\beta-\delta|+2j}A_{|\alpha-\gamma|}B_{|\beta-\delta|} A_jB_j \frac{\langle\xi\rangle^{|\gamma|}\langle x\rangle^{|\delta|} }{\langle\xi\rangle^j\langle x\rangle^j}|D_\xi^\gamma D_x^\delta\chi_j(x,\xi)|\\
		\leq\ & h^{|\alpha-\gamma|+|\beta-\delta|+2j} A_{|\alpha-\gamma|} B_{|\beta-\delta|}A_jB_j\frac{ (3R)^{|\gamma|+|\delta|}a_j^{|\gamma|}b_j^{|\delta|}}{R^{2j}a_j^j b_j^j} \frac{C_\chi h_\chi^{|\gamma|+|\delta|} A_{|\gamma|}B_{|\delta|}}{R^{|\gamma|+|\delta|}a_j^{|\gamma|}b_j^{|\delta|}}\\
		\leq\ & C_\chi h_0^{|\alpha|+|\beta|}A_{|\alpha|}B_{|\beta|} \frac{h^{2j}}{R^{2j}},
	\end{align*}
	where $h_0 = \max\{h,h_\chi,3\}$. Hence, we have
	\begin{equation*}
		S_2 \leq CC_\chi C_R (2h_0)^{|\alpha|+|\beta|}A_{|\alpha|}B_{|\beta|}\langle\xi\rangle^{\mu-|\alpha|}\langle x\rangle^{m-|\beta|},
	\end{equation*}
	where the factor $2$ multiplicating $h_0$ appeared from the sum over binomial coefficients. 
	Combining the estimates for $S_1$ and $S_2$, we conclude that $p\in {\SG}_{\{\A,\B\}}^{m,\mu}(\R^n)$.
	
	Now, notice that, for every $N\in\N_0$,
	\[p(x,\xi)-\sum_{j<N}p_j(x,\xi) = \sum_{j\geq N}p_j(x,\xi)\chi_j(x,\xi),\]
	for $(x,\xi)\in Q_{3R(j)}^c$. Proceeding as above, we can show that $p(x,\xi)-\sum_{j<N}p_j(x,\xi)$ satisfies the estimates of Definition \ref{def_equiv_FS}, hence $p\sim \sum_{j\in\N_0}p_j$ in $\FS_{\{\A,\B\}}^{m,\mu}(\R^n)$.
	
	The projective case can be proved analogously by doing the suitable choices on the constants $h$ and $h_\chi$.	 
\end{proof}

\begin{definition}
	We say that continuous linear operator $P:\caS_{[\M]}(\R^n)\to \caS_{[\M]}(\R^n)$ is $[\M]$-regularizing if it extends to a continuous linear operator from $\caS_{[\M]}'(\R^n)$ to $\caS_{[\M]}(\R^n)$.
\end{definition}

\begin{remark} 
	If a continuous linear operator $P:\caS_{[\M]}(\R^n)\to \caS_{[\M]}(\R^n)$ has kernel $K_P$ in $\caS_{[\M]}(\R^{2n})$, then $P$ is $[\M]$-regularizing.
\end{remark}

\begin{lemma}\label{lemma_pran}
	Let $\ell\in (0,1)$ and $B>1$, and set
	\[\mathscr{M}(t) = \inf\left\{\frac{M_j}{\ell^j t^j}\,:\, j\in\N_0,\ t\geq Bm_j\right\}.\]
	
	Then, there exist $C=C(B,\ell,\Mj)>0$ and $\tilde{h}=\tilde{h}(B,\Mj)>0$ such that
	\[\mathscr{M}(t) \leq Ce^{-\omega_{\M}(\ell \tilde h t)},\]
	for all $t\geq B$.
\end{lemma}
\begin{proof} 
	See \cite[Lemma 8]{Pran2013}.
\end{proof}

\begin{proposition}\label{exp_0_reg}
	Let $p\in {\SG}_{[\A,\B]}^{0,0}(\R^n)$. If $p\sim 0$ in $\FS_{[\A,\B]}^{0,0}$, then $\Op(p)$ is $[\M]$-regularizing.
\end{proposition}
\begin{proof}
	Since $p\sim 0$, there exist $C,h,B>0$ (respectively, there exists $B>0$ and, for every $h>0$, there exists $C>0$) such that, for all $N\in\N_0$ and $(x,\xi)\in Q_{B(N)}^c$, we have
	\begin{equation*}
		|D_\xi^\alpha D_x^\beta p(x,\xi)| \leq Ch^{|\alpha|+|\beta|}M_{|\alpha|}M_{|\beta|}\langle\xi\rangle^{-|\alpha|}\langle x\rangle^{-|\beta|}\frac{M_N}{h^N}.
	\end{equation*}
	
	If $h>1$, let $\ell=h^{-1}$. For $\langle y\rangle\geq Bm_N$, we have
	\[\frac{h^NM_N}{\langle y\rangle^N} \leq \frac{M_N}{\ell^N|y|^N}.\]
	
	If $h\leq 1$, choose any $h_0>1$ and let $\ell=h_0^{-1}$. In this case, for $\langle y\rangle\geq Bm_N$, we have
	\[\frac{h^NM_N}{\langle y\rangle^N} \leq \frac{h_0^NM_N}{|y|^N} = \frac{M_N}{\ell^N|y|^N}.\]
	
	Hence, applying Lemma \ref{lemma_pran}, we obtain $C',\varepsilon>0$ such that
	\[ |D_\xi^\alpha D_x^\beta p(x,\xi)| \leq CC'h^{|\alpha|+|\beta|}M_{|\alpha|}M_{|\beta|} e^{-\omega_{\M}(\ell\varepsilon|x|)}e^{-\omega_{\M}(\ell\varepsilon|\xi|)},\]
	for all $(x,y)\in Q_B^c$ and $\alpha,\beta\in\N_0$. By Proposition \ref{equiv_GS}, we have $p\in\caS_{[\M]}(\R^{2n})$, which implies that the kernel $K_p$ of $\Op(p)$ also belongs to $\caS_{[\M]}(\R^{2n})$. We conclude that $\Op(p)$ is $[\M]$-regularizing.
\end{proof}

\begin{remark}
	If $\A$, $\B$, and $\M$ are non-quasianalytic weight sequences such that $\A,\B\preceq \M$, then any operator with symbol in $\bigcap_{m,\mu\in\R}{\SG}_{[\A,\B]}^{m,\mu}(\R^n)$ is $[\M]$-regularizing.
\end{remark}

We now introduce a useful subclass of $\SG_{[\A,\B]}(\R^n)$. Given $m_x,m_y,\mu\in\R$ and $h>0$, we denote by $\AG_{\A,\B}^{m_x,m_y,\mu}(\R^n;h)$ 
the space of all $p\in C^\infty(\R^{3n})$ such that
\[\|p\|_h = \sup_{\alpha,\beta,\gamma\in\N_0^n}\sup_{x,y,\xi\in\R^n} h^{|\alpha|+|\beta|+|\gamma|} \frac{\langle \xi \rangle^{-\mu+|\alpha|}\langle x\rangle^{-m_x+|\beta|}\langle y\rangle^{m_y-|\gamma|}}{A_{|\alpha|} B_{|\beta|}B_{|\gamma|}}|\partial_\xi^\alpha\partial_x^\beta\partial_t^\gamma p(x,y,\xi)|<+\infty,\]
which is a Banach space with the norm $\|\cdot\|_h$. Then, we define
\[\AG_{\{\A,\B\}}^{m_x,m_y,\mu}(\R^n) = \underset{h>0}{\mathrm{ind\,lim}}\ \AG_{\A,\B}^{m_x,m_y,\mu}(\R^n;h),
\quad
\AG_{(\A,\B)}^{m_x,m_y,\mu}(\R^n) = \underset{h>0}{\mathrm{proj\,lim}}\ \AG_{\A,\B}^{m_x,m_y,\mu}(\R^n;h).\]

An element $p\in \AG_{[\A,\B]}^{m_x,m_y,\mu}(\R^n)$ is said to be an amplitude. It follows by Theorem \ref{thm_ST} that it defines a continuous linear operator on 
$\caS_{[\M]}(\R^n)$ by the formula
\begin{equation}\label{ampl_formula}
	[\Op(p)u](x) = \int_{\R^n}\int_{\R^n} e^{i\xi\cdot (x-y)}p(x,y,\xi)u(y)\,\dbar y\dbar \xi,\quad u\in \caS_{[\M]}(\R^n),
\end{equation}
and it extends to a continuous operator from $\caS_{[\M]}'(\R^n)$ to itself. Moreover, $\Op(p)$ has ultradistributional kernel $K_p(x,y) \in \caS_{[\M]}'(\R^{2n})$ given by
\begin{equation*}
	K_p(x,y) = \int_{\R^n} e^{i\xi\cdot(x-y)} p(x,y,\xi)\,\mathrm{d}\xi,
\end{equation*}
where the integral above is understood as an oscillatory integral. Moreover, $K_p$ satisfies the same estimates of Proposition \ref{est_kernel}.

\begin{theorem}\label{symbol_ampl}
	Let $p\in \AG_{[\A,\B]}^{m_x,m_y,\mu}(\R^n)$. Then, there exists a symbol $\tilde{p}\in {\SG}_{[\A,\B]}^{m,\mu}(\R^n)$, with $m=m_x+m_y$, such that $\Op(p)=\Op(\tilde{p})+R$, where $R$ is a $[\M]$-regularizing operator. Moreover, $\tilde{p}$ has asymptotic expansion
	\[\tilde{p}(x,\xi) \sim \sum_{j\in\N_0}\sum_{|\alpha|=j}\frac{1}{\alpha!}\partial_\xi^\alpha D_y^\alpha p(x,y,\xi)|_{y=x}.\]
\end{theorem}
\begin{proof}
	The proof follows the same lines of \cite[Theorem 6.3.14]{NicRod}, . 
\end{proof}

Also the next Propositions \ref{asympt_exp_tr} and \ref{asympt_exp_adj} follow by standard arguments.

\begin{proposition}\label{asympt_exp_tr}
	Given $p\in {\SG}_{[\A,\B]}^{m,\mu}(\mathbb{R}^n)$, the transposed ${}^t\mathrm{Op}(p)$ of $\mathrm{Op}(p)$, defined by
	\begin{equation*}
		\langle \mathrm{Op}(p)u, v\rangle = \langle u,{}^t\mathrm{Op}(p)v\rangle,\ u,v\in\mathcal{S}_{[\M]}(\mathbb{R}^n),
	\end{equation*}
	is a pseudodifferential operator with symbol ${}^tp\in {\SG}_{[\A,\B]}^{m,\mu}(\mathbb{R}^n)$ satisfying
	\begin{equation*}
		{}^tp(x,\xi)\sim \sum_{j=0}^\infty\sum_{|\alpha|=j}\frac{1}{\alpha!}\partial_\xi^\alpha D_x^\alpha p(x,-\xi).
	\end{equation*}
\end{proposition}

\begin{proposition}\label{asympt_exp_adj}
	Given $p\in {\SG}_{[\A,\B]}^{m,\mu}(\R^n)$, the formal adjoint $\mathrm{Op}(p)^*$ of $\mathrm{Op}(p)$, defined by
	\begin{equation*}
		\langle \mathrm{Op}(p)u,\overline{v}\rangle = \langle u,\overline{\mathrm{Op}(p)^*v}\rangle,\ u,v\in\caS_{[\M]}(\mathbb{R}^n),
	\end{equation*}
	is a pseudodifferential operator with symbol $p^*\in {\SG}_{[\A,\B]}^{m,\mu}(\R^n)$ satisfying
	\begin{equation*}
		p^*(x,\xi)\sim \sum_{j=0}^\infty\sum_{|\alpha|=j}\frac{1}{\alpha!}\partial_\xi^\alpha D_x^\alpha \overline{p(x,\xi)}.
	\end{equation*}
\end{proposition}

\begin{theorem}\label{asympt_exp_comp}
	Given two symbols $p\in {\SG}_{[\A,\B]}^{m,\mu}(\mathbb{R}^n)$ and $q\in {\SG}_{[\A,\B]}^{\mu',m'}(\mathbb{R}^n)$, there exists a symbol $s\in {\SG}_{[\A,\B]}^{\mu+\mu',m+m'}(\mathbb{R}^n)$ and a $[\M]$-regularizing operator $R$ such that $\mathrm{Op}(p)\mathrm{Op}(q)=\Op(s)+R$. Moreover, $s$ has asymptotic expansion
	\begin{equation*}
		s(x,\xi)\sim \sum_{j=0}^\infty\sum_{|\alpha|=j}\frac{1}{\alpha!}\partial_\xi^\alpha p(x,\xi)D_x^\alpha q(x,\xi).
	\end{equation*}
\end{theorem}
\begin{proof}
	We have ${}^t({}^t\Op(q))=\Op(q)$ and ${}^t\Op(q)=\Op(q_1)+R_1$, where $R_1$ is $[\M]$-regularizing and $q_1\in {\SG}_{[\A,\B]}^{\mu',m'}(\R^n)$ has asymptotic expansion
	\begin{equation}\label{exp_q1}
		q_1(x,\xi) \sim \sum_{j\in\N_0}\sum_{|\alpha|=j}\frac{1}{\alpha!}\partial_\xi^\alpha D_x^\alpha q(x,-\xi),
	\end{equation}
	by Proposition \ref{asympt_exp_tr}.
	Moreover, $\Op(p)\Op(q) = \Op(p)\,{}^t\Op(q_1)+R_2$, where $R_2$ is $[\M]$-regularizing. Notice that
	\[{}^t\Op(q)u(x) = \int_{\R^n} e^{i\xi\cdot x}\left(\int_{\R^n}e^{-i\xi\cdot y}q_1(y,-\xi)u(y)\,\dbar y\right)\dbar\xi,\]
	which implies
	\[\widehat{({}^t\Op(q_1)u)}(\xi) = \int_{\R^n}e^{-i\xi\cdot y}q_1(y,-\xi)u(y)\,\dbar y,\]
	and
	\[\Op(p)\,{}^t\Op(q_1) = \int_{\R^n}\int_{\R^n} e^{i\xi\cdot(x-y)}p(x,\xi)q_2(y,\xi)u(y)\,\dbar y\dbar\xi,\]
	where $q_2(x,\xi)=q_1(x,-\xi)$. By \eqref{exp_q1}, we obtain
	\begin{equation}\label{exp_q2}
		q_2(x,\xi) \sim \sum_{j\in\N_0}\sum_{|\alpha|=j}\frac{(-1)^{|\alpha|}}{\alpha!}\partial_\xi^\alpha D_x^\alpha q(x,\xi).
	\end{equation}
	
	In this case, $\Op(p)\,{}^t\Op(q_1)$ is an operator with amplitude $a(x,y,\xi) = p(x,\xi)q_2(y,\xi)\in \AG_{[\A,\B]}^{\mu+\mu',m,m'}(\R^n)$. By Theorem \ref{symbol_ampl}, we have $\Op(p)\,{}^t\Op(q_1) = \Op(s)+R$, where $R$ is $[\M]$-regularizing and $s\in {\SG}_{[\A,\B]}^{\mu+\mu',m+m'}(\R^n)$ satisfies
	\[s(x,\xi) \sim \sum_{j\in\N_0}\sum_{|\alpha|=j}\frac{1}{\alpha!}\partial_\xi^\alpha\left( p(x,\xi) D_x^\alpha q(x,\xi)\right).\]
	
	It only remains to show that
	\[\sum_{j\in\N_0}\sum_{|\alpha|=j}\frac{1}{\alpha!}\partial_\xi^\alpha\left( p(x,\xi) D_x^\alpha q_2(x,\xi)\right) \sim \sum_{j=0}^\infty\sum_{|\alpha|=j}\frac{1}{\alpha!}\partial_\xi^\alpha p(x,\xi)D_x^\alpha q(x,\xi).\]
	
	By the Leibniz rule and \eqref{exp_q2}, we have
	\begin{align*}
		\sum_{\alpha}\frac{1}{\alpha!}\partial_\xi^\alpha\left( p(x,\xi) D_x^\alpha q_2(x,\xi)\right) & \sim \sum_{\alpha,\beta}\frac{1}{\alpha!\beta!}\partial_\xi^\alpha\left( p(x,\xi) \partial_\xi^\beta D_x^{\alpha+\beta}q(x,\xi)\right)\\
		& \sim \sum_{\alpha,\beta}\frac{(-1)^{|\beta|}}{\alpha!\beta!}\sum_{\gamma+\delta=\alpha}\frac{\alpha!}{\gamma!\delta!}\partial_\xi^{\gamma}p(x,\xi)\partial_\xi^{\beta+\delta}D_x^{\beta+\gamma+\delta}q(x,\xi)\\
		& \sim \sum_{\gamma,\nu}\left(\sum_{\beta+\delta=\nu}\frac{(-1)^{|\beta|}}{\beta!\delta!}\right) \frac{1}{\gamma!}\partial_\xi^{\gamma}p(x,\xi)\partial_\xi^{\beta+\delta}D_x^{\beta+\gamma+\delta}q(x,\xi).
	\end{align*}
%
	
	By the binomial formula, with $\boldsymbol{e}=(1,\dots,1)\in\R^n$, we have that
	\begin{equation*}
		\sum_{\beta+\delta=\nu}\frac{(-1)^{|\beta|}}{\beta!\delta!} = \frac{1}{\nu!}\sum_{\beta+\delta=\nu}\frac{\nu!}{\beta!\delta!}(-1)^{|\beta|}\boldsymbol{e}^\beta \boldsymbol{e}^\delta
		=
		\frac{1}{\nu!} (\boldsymbol{e}-\boldsymbol{e})^{\nu}
	\end{equation*}
	vanishes for $\nu\not=0$.
	We conclude that the only non-vanishing terms in the asymptotic expansion are those with $\nu=0$. The proof is complete.	
\end{proof}

Theorem \ref{asympt_exp_comp} implies the usual expected result about localization on left and right of an ultradifferentiable operator by compatible cutoffs.

\begin{proposition}\label{disj_supp}
	Let $p\in {\SG}_{[\A,\B]}^{m,\mu}(\R^n)$, and $\varphi,\psi\in\mathcal{E}_{[\B]}(\R^n)$ be two functions such that $\mathrm{dist}(\mathrm{supp}(\varphi),\mathrm{supp}(\psi))>0$. Assume also that there exist $C,h>0$ (respectively, for every $h>0$, there exists $C>0$) such that
	\[|\partial^{\alpha}\varphi(x)|,|\partial^\alpha\psi(x)|\leq Ch^{|\alpha|}M_{|\alpha|},\quad \forall x\in\R^n,\ \forall \alpha\in\N_0^n.\]
	
	Then, the operator given by
	\[u\mapsto  \varphi\,\mathrm{Op}(p)(\psi u),\quad u\in\mathcal{S}_{[\M]}(\mathbb{R}^n),\]
	is $[\M]$-regularizing.
\end{proposition}


\subsection{Ultradifferentiable ST-classes}\label{subs:STclasses}
We now introduce a more general class of symbols that induce continuous linear operators on Gelfand-Shilov spaces, inspired by those considered by H. O. Cordes in \cite[Ch. 1, \S 1.1]{cordes}. Operators with symbol in these classes will be especially useful in the proof of the coordinate invariance of symbols in $\SG_{[\M]}^{m,\mu}(\R^n)$ on Section \ref{sec_coord_inv}.

\begin{definition}\label{ST_class}
	Let $\Aj,\Bj$ be two weight sequences. We define $\ST_{[\A,\B]}(\mathbb{R}^n)$ as the space of all $p(x,y,\xi)\in C^\infty(\mathbb{R}^{3n})$ such that there exist two non-decreasing functions $\kappa,\lambda$ satisfying
	\[\lim_{j\to\infty}[\kappa(j)-j]=-\infty\quad\text{and}\quad \lim_{j\to\infty}[\lambda(j)-j]=-\infty,\]
	and $C,h>0$ (respectively, for every $h>0$, there exists $C>0$)
	so that
	\[ |\partial_\xi^\alpha\partial_x^\beta\partial_y^\gamma p(x,y,\xi)| \leq Ch^{|\alpha|+|\beta|+|\gamma|} A_{|\alpha|}B_{|\beta|}B_{|\gamma|} \langle \xi\rangle^{\kappa(|\beta|+|\gamma|)}\langle (x,y)\rangle^{\lambda(|\alpha|)},\]
	for all $(x,y,\xi)\in\mathbb{R}^{3n}$ and $\alpha,\beta,\gamma\in\mathbb{N}_0$.
\end{definition}

\begin{proposition}\label{der_ST}
	If $p\in \ST_{[\A,\B]}(\mathbb{R}^n)$, then $\partial_\xi^\alpha\partial_x^\beta\partial_y^\gamma p\in \ST_{[\A,\B]}(\mathbb{R}^n)$, for every fixed $\alpha,\beta,\gamma\in\mathbb{N}_0^n$.
\end{proposition}
\begin{proof}
	Let $\alpha',\beta',\gamma'\in\mathbb{N}_0^n$. Then, there exist $C_0,h_0>0$ such that
	\begin{align*}
		|\partial_\xi^{\alpha'}\partial_x^{\beta'}\partial_y^{\gamma'} \partial_\xi^{\alpha}\partial_x^{\beta}\partial_y^{\gamma} p(x,y,\xi)| \leq\ & C_0h_0^{|\alpha|+|\beta|+|\gamma|}h^{|\alpha'|+|\beta'|+|\gamma'|} A_{|\alpha|+|\alpha'|}B_{|\beta|+|\beta'|}B_{|\gamma|+|\gamma'|} \\
		& \times \langle\xi\rangle^{\kappa(|\beta|+|\gamma|+|\beta'|+|\gamma'|)} \langle (x,y)\rangle^{\lambda(|\alpha|+|\alpha'|)} \\
		\leq &\ \left( C_0 (2h_0)^{|\alpha|+|\beta|+|\gamma|} A_{|\alpha|} B_{|\beta|} B_{|\gamma|} \right) (2h_0)^{|\alpha'|+|\beta'|+|\gamma'|}A_{|\alpha|'} B_{|\beta'|} B_{|\gamma'|}\\
		& \times \langle\xi\rangle^{\kappa(|\beta|+|\gamma|+|\beta'|+|\gamma'|)} \langle (x,y)\rangle^{\lambda(|\alpha|+|\alpha'|)}.
	\end{align*}
	Then, to prove the desired estimates, it is enough to take $C = C_0(2h_0)^{|\alpha|+|\beta|+|\gamma|}A_{|\alpha|} B_{|\beta|} B_{|\gamma|}$, $h=2h_0$, $\kappa'(j)=\kappa(|\beta|+|\gamma|+j)$, and $\lambda'(j)=\lambda(|\alpha|+j)$.
\end{proof}

\begin{proposition}\label{poly_ST}
	If $p\in \ST_{[\A,\B]}(\mathbb{R}^n)$ and $q$ is a polynomial in $(x,y,\xi)$, then $pq\in \ST_{[\A,\B]}(\mathbb{R}^n)$.
\end{proposition}
\begin{proof}
	First, since $q$ is a polynomial in $(x,y,\xi)$, we can find a constant $C_q>0$ such that
	\[|\partial_\xi^{\alpha'}\partial_x^{\beta'}\partial_y^{\gamma'} q(x,y,\xi)|\leq C_q\langle\xi\rangle^{d_1}\langle(x,y)\rangle^{d_2},\]
	for every $\alpha',\beta',\gamma'\in\mathbb{N}_0^n$ and $x,y,\xi\in\mathbb{R}^n$, where $d_1,d_2\in\mathbb{N}_0$ are the $\xi$-degree and the $(x,y)$-degree of $q$, respectively.
	
	Now, let $\alpha,\beta,\gamma\in\mathbb{N}_0^n$. Then,
	\begin{align*}
		|\partial_\xi^{\alpha} \partial_x^{\beta} \partial_y^{\gamma} &(pq)(x,y,\xi)| \leq
		\\ &\leq \sum_{\alpha'\leq\alpha}\binom{\alpha}{\alpha'} \sum_{\beta'\leq\beta}\binom{\beta}{\beta'} \sum_{\gamma'\leq\gamma}\binom{\gamma}{\gamma'} |\partial_\xi^{\alpha'}\partial_x^{\beta'}\partial_y^{\gamma'} q(x,y,\xi)||\partial_\xi^{\alpha-\alpha'}\partial_x^{\beta-\beta'}\partial_y^{\gamma-\gamma'} p(x,y,\xi)|\\
		& \leq \sum_{\alpha'\leq\alpha}\binom{\alpha}{\alpha'} \sum_{\beta'\leq\beta}\binom{\beta}{\beta'} \sum_{\gamma'\leq\gamma}\binom{\gamma}{\gamma'} C_q\langle\xi\rangle^{d_1}\langle(x,y)\rangle^{d_2} Ch^{|\alpha|+|\beta|+|\gamma|} A_{|\alpha|}B_{|\beta|}B_{|\gamma|} \\
		& \times \langle \xi\rangle^{\kappa(|\beta|+|\gamma|)}\langle(x,y)\rangle^{\lambda(|\alpha|)}\\
		& \leq CC_q(2h)^{|\alpha|+|\beta|+|\gamma|} A_{|\alpha|}B_{|\beta|}B_{|\gamma|} \langle \xi\rangle^{d_1+\kappa(|\beta|+|\gamma|)}\langle(x,y)\rangle^{d_2+\lambda(|\alpha|)},
	\end{align*}
	where the term $2^{|\alpha|+|\beta|+|\gamma|}$ is a consequence of the sums over binomial coefficients. Taking $\kappa'(j)=N+\kappa(j)$ and $\lambda'(j)=M+\lambda(j)$, we conclude the proof.    
\end{proof}

\begin{theorem}\label{thm_ST}
	Let $\A=\Aj,\B=\Bj\preceq\M=\Mj$ be three non-quasianalytic weight sequences and $p\in \ST_{[\A,\B]}(\mathbb{R}^n)$. Then, the map $\mathrm{Op}(p):\caS_{[\M]}(\R^n)\to\caS_{[\M]}(\R^n)$ given by the formula
	\begin{equation}\label{pseudo_def}
		\mathrm{Op}(p)u(x) = \int_{\mathbb{R}^n}\int_{\mathbb{R}^n}e^{i\xi\cdot (x-y)}p(x,y,\xi){u}(y)\,\dbar y\,\dbar\xi,\quad u\in\caS_{[\M]}(\R^n),
	\end{equation}
	is a continuous linear operator. Moreover, $\mathrm{Op}(p)$ extends to a continuous linear operator $\mathcal{S}_{[\M]}'(\R^n) \to \mathcal{S}_{[\M]}'(\R^n)$.
\end{theorem}
\begin{proof}
	Let $u \in \caS_{[\M]}(\R^n)$ and $N_1,N_2\in\mathbb{N}_0$. Integrating by parts, we obtain that $[\mathrm{Op}(p)u](x)$ equals
	\begin{equation}\label{I_x_xi}
		\langle\xi\rangle^{-2N_1} \langle x-y\rangle^{-2N_2} \int_{\mathbb{R}^n}\int_{\mathbb{R}^n} e^{i\xi\cdot(x-y)} (1-\Delta_\xi)^{N_1} (1-\Delta_y)^{N_2} (p(x,y,\xi)u(y))\,\dbar y\,\dbar \xi.
	\end{equation}
	
	By the Leibniz rule, $(1-\Delta_\xi)^{N_1}(1-\Delta_y)^{N_2}(p(x,y,\xi)u(y))$ is a linear combination of terms of the form
	\[(1-\Delta_\xi)^{N_1}\partial_y^{\gamma_1}p(x,y,\xi)\partial_y^{\gamma_2}u(y),\quad |\gamma_1+\gamma_2|\leq 2N_2.\]
	For such terms, given $N_3\in\mathbb{N}_0$, we have
	\begin{align*}
		|(1-\Delta_\xi)^{N_1}\partial_y^{\gamma_1}p(x,y,\xi)\partial_y^{\gamma_2}u(y)| & = \langle y\rangle^{-2N_3}|(1-\Delta_\xi)^{N_1}\partial_y^{\gamma_1}p(x,y,\xi)\langle y\rangle^{2N_3}\partial_y^{\gamma_2}u(y)|\\
		& \leq Ch^{2N_1+2N_2+2N_3} A_{2N_1}B_{2N_2} B_{2N_3}\\
		& \times \langle\xi\rangle^{\kappa(2N_2)}\langle (x,y)\rangle^{\lambda(2N_1)}\langle y\rangle^{-N_3},
	\end{align*}
	where we used Lemma \ref{est_GS_bracket}, the fact that $|\gamma_1|\leq 2N_2$ and that the function $\kappa$ in non-decreasing. Then, for each $N_3\in\mathbb{N}_0$, there exist $C,h>0$ (respectively, for every $h>0$, there exists $C>0$) such that
	\begin{align*}|(1-\Delta_\xi)^{N_1}(1-\Delta_y)^{N_2}(p(x,y,\xi)u(y))| & \leq 2^{2N_2}Ch^{2N_1+2N_2+N_3} A_{2N_1}B_{2N_2} B_{2N_3}\\
		& \times \langle\xi\rangle^{\kappa(2N_2)}\langle (x,y)\rangle^{\lambda(2N_1)}\langle y\rangle^{-2N_3},
	\end{align*}
	where the factor $2^{2N_2}$ can be incorporated into $h$. Also, notice that
	\begin{align*}
		\int_{\mathbb{R}^n}\langle (x,y)\rangle^{\lambda(2N_1)}\langle y\rangle^{-2N_3}\,\dbar y & = \langle x\rangle^{\lambda(2N_1)+n}\int_{\mathbb{R}^n}\langle y\rangle^{\lambda(2N_1)}(1+\langle x\rangle^2|y|^2)^{-N_3}\,\dbar y\\
		& \leq C'\langle x\rangle^{\lambda(2N_1)+n},
	\end{align*}
	where we used the substitution $y=y'\langle x\rangle$, $\dbar y = \langle x\rangle^n\dbar y'$, and %
	\[C'=\int_{\mathbb{R}^n}\langle y\rangle^{\lambda(2N_1)-2N_3}\,\dbar y,\]
	since we can choose $N_3$ depending on $N_1$ to be sufficiently large such that the integral converges.
	
	Now, since we have $\lim_{j\to\infty}\lambda(j)-j=-\infty$, there are just finitely many values of $j\in\mathbb{N}_0$ such that $\lambda(j)-j$ is positive. Hence, we can obtain $\tilde N\in\mathbb{N}$ such that $\tilde N>\lambda(j)-j$ for every $j\in\mathbb{N}_0$. In this case, for each $N_1\in\mathbb{N}_0$, we can choose $N_3=N_1+\tilde N+n$. Then we have
	\[\int_{\mathbb{R}^n}\langle y\rangle^{\lambda(2N_1)-2N_3}\,\dbar y = \int_{\mathbb{R}^n}\langle y\rangle^{\lambda(2N_1)-2N_1-2\tilde N}\langle y\rangle^{-2n}\,\dbar y,\]
	where $\langle y\rangle^{\lambda(2N_1)-2N_1-2\tilde N}\leq 1$ for all $y\in\mathbb{R}^n$ and for all $N_1\in\mathbb{N}_0$, and the integral
	\[\int_{\mathbb{R}^n}\langle y\rangle^{-2n}\,\dbar y\]
	is finite. Hence, we obtain a constant $C'>0$, independent from $N_1$, such that
	\[\int_{\mathbb{R}^n}\langle (x,y)\rangle^{\lambda(2N_1)}\langle y\rangle^{-2N_3}\,\dbar y \leq C'\langle x\rangle^{\lambda(2N_1)+n}\]
	for all $N_1\in\mathbb{N}_0$, where $N_3$ is chosen as before. In this case, \eqref{I_x_xi} is dominated by
	\begin{equation}\label{est_I}
		C_0\langle\xi\rangle^{\kappa(2N_2)-2N_2}\langle x\rangle^{n+\lambda(2N_1)-2N_1},
	\end{equation}
	where
	\[C_0 = CC'h^{2N_1+2N_2+2N_3} A_{2N_1}B_{2N_2} B_{2N_3}.\]
	
	Then, for a sufficiently large $N$, the integral
	\[\int_{\mathbb{R}^n}\int_{\mathbb{R}^n}e^{i\xi\cdot(x-y)}p(x,y,\xi)u(y)\,\dbar y\,\dbar\xi\]
	is finite. In particular, the previous estimates allows us to differentiate under the integral sign. Hence, for each $u\in\caS_{[\M]}(\R^n)$, $\Op(p)u$ defines a function in $\caS_{[\M]}(\R^n)$.
	
	We now show the continuity of $\Op(p)$. Let $\mathscr{B}\subset\caS_{[\M]}(\R^n)$ be a bounded subset. Then, there exist $C,h>0$ (respectively, for every $h>0$, there exists $C>0$) such that
	\[\sup_{x\in\R^n}|x^\beta\partial_x^\alpha u(x)|\leq Ch^{|\alpha|+|\beta|}M_{|\alpha|}M_{|\beta|},\]
	for all $\alpha,\beta\in\mathbb{N}_0^n$ and $u\in\mathscr{B}$. Let us show that $\mathrm{Op}(p)(\mathscr{B})$ is bounded in $\mathcal{S}_{[\M]}(\mathbb{R}^n)$. Indeed, for every $\alpha,\beta\in\mathbb{N}_0^n$, we have
	\begin{align*}
		x^\beta D_x^\alpha\mathrm{Op}(p)u(x) & = x^\beta D_x^\alpha\int_{\mathbb{R}^n}\int_{\mathbb{R}^n}e^{i\xi\cdot (x-y)}p(x,y,\xi) u(y)\,\dbar y\,\dbar\xi\\
		& = x^\beta\sum_{\gamma\leq\alpha}\binom{\alpha}{\gamma} \int_{\mathbb{R}^n} \int_{\mathbb{R}^n} e^{i\xi\cdot (x-y)}\xi^\gamma D_x^{\alpha-\gamma}p(x,y,\xi) u(y)\,\dbar y\,\dbar\xi.
	\end{align*}
	
	By the second part of the proof of \cite[Theorem 1.1]{cordes}, integration by parts and Leibniz rule give that $x^\beta D_x^\alpha[\mathrm{Op}(p)u](x)$ is a linear combination of terms of the form
	\begin{equation}\label{terms_ST}
		\int_{\mathbb{R}^n} \int_{\mathbb{R}^n} e^{i\xi\cdot(x-y)} (x-y)^{\alpha_1-\gamma_1}
		\,D^{\gamma_2}_y(u(y)y^{\alpha_2})
		\,D_x^{\beta_2}D_y^{\gamma_3}p(x,y,\xi)\,\dbar y\,\dbar\xi,
	\end{equation}
	where $\alpha_1+\alpha_2=\alpha$, $\beta_1+\beta_2=\beta$, and $\gamma_1+\gamma_2+\gamma_3=\beta_1$. Then, the estimates obtained in the proofs of Propositions \ref{der_ST} and \ref{poly_ST}, as well as the bound \eqref{est_I}, give us constants $C,h>0$ such that the terms of the form \eqref{terms_ST} are bounded by
	\[Ch^{|\alpha|+|\beta|} M_{|\alpha|}M_{|\beta|}\langle x \rangle^{\lambda(|\beta|)-N},\]
	for an arbitrary, sufficiently large $N\in\mathbb{N}_0$. Then, we obtain
	\begin{equation*}
		\sup_{x\in\R^n}|x^\beta D_x^\alpha\mathrm{Op}(p)u(x)|\leq Ch_0^{|\alpha|+|\beta|}M_{|\alpha|}M_{|\beta|}\langle x\rangle^{n+\lambda(|\beta|)},
	\end{equation*}
	for a new constant $h_0>0$, that comprises the factor of the form $h_1^{|\alpha|+|\beta|}$ arising from the sum of the binomial coefficients of Leibniz rule. It follows
	that $\mathrm{Op}(p)$ is bounded from $\caS_{[\M]}(\R^n)$ to itself.
	
	Finally, the proof that $\Op(p)$ extends to a continuous operator from $\mathcal{S}'_{[\M]}(\R^n)$ to itself is analogous to the one of Proposition \ref{cont_SG_ext}. 
	The proof is complete.
\end{proof}

\begin{remark}\label{remark_ST}
	A simple adaptation of the proof of Theorem \ref{thm_ST} gives us that if we replace $\lambda(|\alpha|)$ with $\lambda(|\alpha|)+|\beta|+|\gamma|$ in the definition of the symbol classes, we still obtain continuous operators from $\mathcal{S}_{[\M]}(\mathbb{R}^n)$ to $\mathcal{S}_{[\M]}(\mathbb{R}^n)$.
\end{remark}


\subsection{Hypoellipticity and parametrices}\label{sec:SGABHypoell}
The definitions of ellipticity and hypoellipticity for ultradifferential SG-calculus are the same of the standard one.

\begin{definition}\label{def_elliptic}
	We say that a symbol $p\in {\SG}_{[\A,\B]}^{m,\mu}$ is elliptic if there exist $B,C>0$ such that
	\begin{equation}\label{elliptic_cond}
		|p(x,\xi)|\geq C\langle\xi\rangle^{\mu}\langle x\rangle^{m},
	\end{equation}
	for all $(x,\xi)\in Q_B^c$.
\end{definition}

\begin{definition}\label{def_hypo}
	We say that a symbol $p\in {\SG}_{[\A,\B]}^{m,\mu}$ is hypoelliptic if there exist $m'\leq m$, $\mu'\leq \mu$, and $B,A>0$ such that
	\begin{equation}\label{hypo_cond1}
		|p(x,\xi)| \geq A\langle\xi\rangle^{\mu'}\langle x\rangle^{m'},
	\end{equation}
	for all $(x,\xi)\in Q_B^c$, and there exist $C,h>0$ (respectively, for every $h>0$, there exists $C>0$) such that
	\begin{equation}\label{hypo_cond2}
		|\partial_\xi^\alpha\partial_x^\beta p(x,\xi)| \leq Ch^{|\alpha|+|\beta|}A_{|\alpha|}B_{|\beta|}|p(x,\xi)|\langle\xi\rangle^{-|\alpha|}\langle x\rangle^{-|\beta|},
	\end{equation}
	for all $\alpha,\beta\in\N_0^n$ and $(x,\xi)\in Q_B^c$.
\end{definition}

\begin{remark}
	Notice that every elliptic symbol is hypoelliptic. Indeed, condition \eqref{hypo_cond1} is satisfied for $m'=m$ and $\mu'=\mu$. 
	Also, if $p$ belongs to ${\SG}_{[\A,\B]}^{m,\mu}$ and satisfies \eqref{elliptic_cond}, there exist $C,h>0$ 
	(respectively, for every $h>0$, there exists $C>0$) such that
	\begin{align*}
		|\partial_\xi^\alpha\partial_x^\alpha p(x,\xi)| & \leq Ch^{|\alpha|+|\beta|}A_{|\alpha|}B_{|\beta|}\langle\xi\rangle^{\mu-|\alpha|}\langle x\rangle^{m-|\beta|}\\
		& = Ch^{|\alpha|+|\beta|}A_{|\alpha|}B_{|\beta|}\langle\xi\rangle^{\mu}\langle x\rangle^{m}\langle\xi\rangle^{-|\alpha|}\langle x\rangle^{-|\beta|}\\
		& \leq Ch^{|\alpha|+|\beta|}A_{|\alpha|}B_{|\beta|}|p(x,\xi)|\langle\xi\rangle^{-|\alpha|}\langle x\rangle^{-|\beta|},
	\end{align*}
	for all $\alpha,\beta\in\N_0^n$ and $Q_B^c$. Hence, $p$ satisfies \eqref{hypo_cond2}.
\end{remark}

\begin{lemma}\label{lemma_1.3.5}
	Let $p\in {\SG}_{[\A,\B]}^{m,\mu}(\R^n)$ be a hypoelliptic symbol with $m',\mu'$ as in Definition \ref{def_hypo}, and suppose that there exists $B>0$ such that $p(x,\xi)>0$ for $(x,\xi)\in Q_B^c$. Then, for every $N\in\Z$, there exist $C,h>0$ (respectively, for every $h>0$, there exists $C>0$) such that
	\begin{equation*}
		|\partial_\xi^\alpha\partial_x^\beta p(x,\xi)^N|\leq Ch^{|\alpha|+|\beta|}A_{|\alpha|}B_{|\beta|}|p(x,\xi)^N|\langle\xi\rangle^{-|\alpha|}\langle x\rangle^{-|\beta|},
	\end{equation*}
	for all $\alpha,\beta\in\N_0$ and $(x,\xi)\in Q_B^c$. Moreover, $p^N\in {\SG}_{[\A,\B]}^{Nm, N\mu}(\R^n)$ if $N\geq 0$, and 
	$p^N\in {\SG}_{[\A,\B]}^{Nm',N\mu'}(\R^n)$ if $N<0$.
\end{lemma}
\begin{proof}
	The claim follows by a straightforward variant of the argument in \cite[Lemma 1.3.5]{NicRod}.
\end{proof}

\begin{theorem}\label{thm_para}
	Let $p\in {\SG}_{[\A,\B]}^{m,\mu}(\R^n)$ be a hypoelliptic symbol and consider $m',\mu'$ as in condition \eqref{hypo_cond1}. 
	Then, there exists $q\in {\SG}_{[\A,\B]}^{-m', -\mu'}(\R^n)$ such that
	\begin{equation*}
		\Op(q)\Op(p)=I+R_1\quad\text{and}\quad \Op(p)\Op(q)=I+R_2,
	\end{equation*}
	where $I$ is the identity operator and $R_1,R_2$ are $[\M]$-regularizing operators.
\end{theorem}
\begin{proof}
	Since $p$ is hypoelliptic, let $q_0\in {\SG}_{[\A,\B]}^{-m',-\mu'}(\R^n)$ be such that $q_0(x,\xi)=p(x,\xi)^{-1}$ for $(x,\xi)\in Q_B^c$. 
	We claim that $\Op(q_0)\Op(p)=I+\Op(r_1)$, where $r_1\in {\SG}_{[\A,\B]}^{-1,-1}(\R^n)$. 
	The proof is an adaptation of the usual parametrix construction argument. For the sake of completeness, we provide some detail.
	
	By Theorem \ref{asympt_exp_comp}, $\Op(q_0)\Op(p)$ has symbol with asymptotic expansion
	\[1+\sum_{\alpha\neq 0}\frac{1}{\alpha!}\partial_\xi^\alpha [p(x,\xi)^{-1}] D_x^\alpha p(x,\xi).\] 
	
	Notice that, for each $\alpha\neq 0$, the symbol $\partial_\xi^\alpha [p(x,\xi)^{-1}] D_x^\alpha p(x,\xi)$ belongs to 
	${\SG}_{[\A,\B]}^{-|\alpha|,-|\alpha|}(\R^n)$. Indeed, since $p$ is hypoelliptic, by \eqref{hypo_cond2}, Lemma \ref{lemma_1.3.5}, 
	and the Leibniz rule, we obtain
	\begin{align*}
		|\partial_\xi^{\alpha'}\partial_x^\beta(\partial_\xi^\alpha &[p(x,\xi)^{-1}] D_x^\alpha p(x,\xi))|\\
		&\leq\   \sum_{\gamma\leq\alpha'}\binom{\alpha'}{\gamma} \sum_{\delta\leq\beta}\binom{\beta}{\delta} 
		|\partial_\xi^{\gamma+\alpha}\partial_x^{\delta} [p(x,\xi)^{-1}]| |\partial_\xi^{\alpha'-\gamma+\alpha}\partial_x^{\beta-\delta+\alpha} p(x,\xi)|\\
		&\leq\ \sum_{\gamma\leq\alpha'}\binom{\alpha'}{\gamma} \sum_{\delta\leq\beta}\binom{\beta}{\delta} 
		Ch^{|\alpha|+|\gamma|+|\delta|}\frac{A_{\alpha}A_{\gamma}B_{|\delta|}}{|p(x,\xi)|}\langle\xi\rangle^{-|\alpha+\gamma|}\langle x\rangle^{-|\delta|}\\
		& \times h^{|\alpha'-\gamma|+|\beta-\delta|+|\alpha|}
		A_{|\alpha'-\gamma|}B_{|\beta-\delta|}B_{|\alpha|}|p(x,\xi)|\langle\xi\rangle^{-|\alpha'-\gamma|}\langle x\rangle^{-|\beta-\delta|-|\alpha|}\\
		&\leq\ \sum_{\gamma\leq\alpha'}\binom{\alpha'}{\gamma} \sum_{\delta\leq\beta}\binom{\beta}{\delta} 
		Ch^{2|\alpha|+|\alpha'|+|\beta|}A_{|\alpha'|}B_{|\beta|}A_{|\alpha|}B_{|\alpha|}\langle\xi\rangle^{-|\alpha|-|\alpha'|}\langle x\rangle^{-|\alpha|-|\beta|}\\
		&\leq\ C_0h_0^{|\alpha'|+|\beta|}A_{|\alpha'|}B_{|\beta|}\langle\xi\rangle^{-|\alpha|-|\alpha'|}\langle x\rangle^{-|\alpha|-|\beta|},
	\end{align*}
	where $C_0=Ch^{2|\alpha|}A_{|\alpha|}B_{|\alpha|}$ and $h_0=2h$, for the inductive case. 
	For the projective case, for each $h>0$, we can choose $h/2$ in the symbol estimates.
	
	In this case, by Proposition \ref{prop_p_pj}, we take $r_1\in {\SG}_{[\A,\B]}^{-1,-1}(\R^n)$ such that
	\[r_1(x,\xi) \sim \sum_{\alpha\neq 0}\frac{1}{\alpha!}\partial_\xi^\alpha [p(x,\xi)^{-1}] D_x^\alpha p(x,\xi),\]
	which proves our claim. In particular, $\Op(q_0)\Op(p)$ is an operator with symbol in ${\SG}_{[\A,\B]}^{0,0}(\R^n)$.
	
	Now, we set $r_0(x,\xi)=1$ and, for every $j\in\N$, let $r_j(x,\xi)\in {\SG}_{[\A,\B]}^{-j,-j}(\R^n)$ be the symbol of $\Op(r_1)^j$. 
	Again by Proposition \ref{prop_p_pj}, let $r\in {\SG}_{[\A,\B]}^{0,0}(\R^n)$ be such that
	\[r(x,\xi) \sim \sum_{j\in\N_0}r_j(x,\xi).\]
	
	Then, by a standard argument,
	if we take $q\in {\SG}_{[\A,\B]}^{-m',-\mu'}(\R^n)$ to be the symbol of $\Op(r)\Op(q_0)$, we obtain $\Op(q)\Op(p)=I+R_1$, with $R_1$ $[\M]$-regularizing. 
	Analogously, we obtain $q'\in {\SG}_{[\A,\B]}^{-m',-\mu'}(\R^n)$ and $R_2$ $[\M]$-regularizing such that $\Op(p)\Op(q')=I+R_2$,
	and $\Op(q)-\Op(q')$ is $[\M]$-regularizing. The details are left for the reader.
\end{proof}

The (hypo)elliptic regularity result follows, as usual, by the previous parametrix construction.

\begin{corollary}
	If a symbol $p\in {\SG}_{[\A,\B]}^{m,\mu}(\R^n)$ is hypoelliptic, then $\Op(p)$ satisfies
	\begin{equation}\label{GH}
		u\in\caS_{[\M]}'(\R^n) \text{ and } \Op(p)u\in \caS_{[\M]}(\R^n)\ \Rightarrow\ u\in\caS_{[\M]}(\R^n).
	\end{equation}
	
	An operator satisfying \eqref{GH} is said to be $[\M]$-globally hypoelliptic.
\end{corollary}
%


\subsection{SG-classical ultradifferential symbols}\label{sec:sgucl}
We now provide definitions and main results concerning classical symbols in the $SG_{[\A,\B]}^{m,\mu}(\R^n)$ classes, as well as the definition, characterization and main properties of global ultradifferentiable wave front sets for temperate ultradistributions. The proofs of the claims in this section are omitted, since they are analogous to those found in Sections 3 and 4 of \cite{CapRod2006} for the standard SG-classes
(see also the characterization of classical SG-symbols
recalled, e.g., in \cite{EgSchu97}).

Let $\A=\Aj$ and $\B=\Bj$ be two non-quasianalytic weight sequences and consider the associated sequences $\{a_j\}_{j\in\N_0}$ and $\{b_j\}_{j\in\N_0}$ given by $a_0=b_0=0$ and $a_j=A_j/A_{j-1}$, $b_j=B_j/B_{j-1}$, for $j\geq 1$. The next definitions and propositions
provide adapted versions of the standard notion of asymptotic summation with respect to the $x$ and $\xi$ variables, in general and in terms of homogeneous
summands.

\begin{definition}\label{def:SGcl_xi}
	We define $\FS_{[\A,\B],\xi}^{m,\mu}(\R^n)$ as the space of formal sums $\sum_{j\in\N_0}p_j$, where $p_j\in C^\infty(\R^{2n})$ and
	\begin{equation*}
		\sup_{j\in\N_0}\sup_{\alpha,\beta\in\N_0^n}\sup_{\langle\xi\rangle \geq Ba_j} \frac{\langle\xi\rangle^{-\mu+j+|\alpha|}\langle x\rangle^{-m+|\beta|}|\partial_\xi^\alpha\partial_x^\beta p_j(x,\xi)|}{h^{|\alpha|+|\beta|+j} A_{|\alpha|}A_jB_{|\beta|}B_j} < +\infty.
	\end{equation*}
	for some $h>0$ (respectively, for every $h>0$).
\end{definition}
\begin{definition}\label{def:SGcl_xieq}
	We say that two formal sums $\sum_{j\in\N_0}p_j, \sum_{j\in\N_0}\tilde{p}_j \in \FS_{[\A,\B],\xi}^{m,\mu}(\R^n)$ are equivalent if there exist $B,h>0$ such that (respectively, there exists $B>0$ such that, for every $h>0$, we have)
	\begin{equation*}
		\sup_{N\in\N_0}\sup_{\alpha,\beta\in\N_0^n}\sup_{\langle\xi\rangle\geq Ba_j} \frac{\langle\xi\rangle^{-\mu+N+|\alpha|}\langle x\rangle^{-m+|\beta|} \left| \sum_{j<N}\partial_\xi^\alpha\partial_x^\beta(p_j(x,\xi)-\tilde{p}_j(x,\xi))\right|}{h^{|\alpha|+|\beta|+N} A_{|\alpha|}A_NB_{|\beta|}B_N} < +\infty.
	\end{equation*}
	In this case, we write $\sum_{j\in\N_0}p_j\sim_\xi \sum_{j\in\N_0}\tilde{p}_j$.
\end{definition}

\begin{definition}\label{def:SGcl_x}
	We define $\FS_{[\A,\B],x}^{m,\mu}(\R^n)$ as the space of formal sums $\sum_{j\in\N_0}p_j$, where each $p_j(x,\xi)$ belongs to $C^\infty(\R^{2n})$ and
	\begin{equation*}
		\sup_{j\in\N_0}\sup_{\alpha,\beta\in\N_0^n}\sup_{\langle\xi\rangle \geq Bb_j} \frac{\langle\xi\rangle^{-\mu+|\alpha|}\langle x\rangle^{-m+j+|\beta|}|\partial_\xi^\alpha\partial_x^\beta p_j(x,\xi)|}{h^{|\alpha|+|\beta|+j} A_{|\alpha|}A_jB_{|\beta|}B_j} < +\infty.
	\end{equation*}
	for some $h>0$ (respectively, for every $h>0$).
\end{definition}

\begin{definition}\label{def:SGcl_xeq}
	We say that two formal sums $\sum_{j\in\N_0}p_j, \sum_{j\in\N_0}\tilde{p}_j \in \FS_{[\A,\B],x}^{m,\mu}(\R^n)$ are equivalent if there exist $B,h>0$ such that (respectively, there exists $B>0$ such that, for every $h>0$, we have)
	\begin{equation*}
		\sup_{N\in\N_0}\sup_{\alpha,\beta\in\N_0^n}\sup_{\langle x\rangle\geq Bb_j} \frac{\langle\xi\rangle^{-\mu+|\alpha|}\langle x\rangle^{-m+N+|\beta|} \left| \sum_{j<N}\partial_\xi^\alpha\partial_x^\beta(p_j(x,\xi)-\tilde{p}_j(x,\xi))\right|}{h^{|\alpha|+|\beta|+N} A_{|\alpha|}A_NB_{|\beta|}B_N} < +\infty.
	\end{equation*}
	
	In this case, we write $\sum_{j\in\N_0}p_j \sim_x \sum_{j\in\N_0}\tilde{p}_j$.
\end{definition}

Using the same arguments of Proposition \ref{prop_p_pj}, one can prove the following.

\begin{proposition}
	Let $\sum_{j\in\N_0}p_j\in \FS_{[\A,\B],\xi}^{m,\mu}(\R^n)$ and $\sum_{j\in\N_0}q_j\in \FS_{[\A,\B],x}^{m,\mu}(\R^n)$ be two formal sums. Then, there exist symbols $p,q\in {\SG}_{[\A,\B]}^{m,\mu}(\R^n)$ such that
	\[p\sim_\xi \sum_{j\in\N_0}p_j\ \text{in}\ \FS_{[\A,\B],\xi}^{m,\mu}(\R^n) \quad\text{and}\quad q\sim_x\sum_{j\in\N_0}q_j\ \text{in}\ \FS_{[\A,\B],x}^{m,\mu}(\R^n).\]
\end{proposition}

Let $\chi,\phi\in\E_{[\M]}(\R^n)$ be two excision functions, that is $\chi,\phi\equiv 0$ in a neighbourhood of the origin, and $\chi,\phi\equiv 1$ outside a larger 
neighbourhood the origin.

\begin{definition}
	We denote by ${\SG}_{[\A,\B]}^{m,[\mu]}(\R^n)$ the space of all $p\in C^\infty(\R^n\times(\R^n\setminus\{0\}))$ satisfying the following properties:
	\begin{itemize}
		\item $\chi(\xi)p(x,\xi) \in {\SG}_{[\A,\B]}^{m,\mu}(\R^n)$;
		\item $p(x,\lambda\xi)=\lambda^{\mu}p(x,\xi)$, for all $\lambda>0$ and $(x,\xi)\in\R^{2n}$ with $\xi\neq 0$.
	\end{itemize}
	
	  Analogously, we define ${\SG}_{[\A,\B]}^{[m],\mu}(\R^n)$ by interchanging the roles of $x$ and $\xi$ and using $\phi(x)$ in place of $\chi(\xi)$. Finally, we set
	  \[{\SG}_{[\A,\B]}^{[m,\mu]}(\R^n) = {\SG}_{[\A,\B]}^{m,[\mu]}(\R^n)\cap {\SG}_{[\A,\B]}^{[m],\mu}(\R^n).\]
\end{definition}

As in the proof of Proposition \ref{prop_p_pj}, given $B>0$ and $j,k\in\N_0$, we set $\chi_0\equiv \phi_0\equiv 1$ and
\begin{equation*}
	\chi_j(\xi)=\chi\left(\frac{\xi}{Ba_j}\right),\quad \phi_k(x)=\phi\left(\frac{x}{Bb_k}\right), \quad j,k\geq 1.
\end{equation*}

Now, we can define various SG-classical ultradifferentiable symbols classes.

\begin{definition}
	We denote by ${\SG}_{[\A,\B],\mathrm{cl}(\xi)}^{m,\mu}(\R^n)$ the space of all $p\in {\SG}_{[\A,\B]}^{m,\mu}(\R^n)$ satisfying the following properties:
	\begin{itemize}
		\item there exists a formal sum $\sum_{j\in\N_0} \chi_j(\xi)p_j(x,\xi)\in \FS_{[\A,\B],\xi}^{m,\mu}(\R^n)$ such that, for each $j\in\N_0$, 
		$p_j$ belongs to ${\SG}_{[\A,\B]}^{m,[\mu-j]}(\R^n)$;
		\item $p(x,\xi)\sim_\xi \sum_{j\in\N_0}\chi_j(\xi)p_j(x,\xi)$.
	\end{itemize} 
	
	Analogously, we define ${\SG}_{[\A,\B],\mathrm{cl}(x)}^{m,\mu}(\R^n)$ by interchanging the role of $x$ and $\xi$ and using $\phi_k(x)$ in place of $\chi_j(\xi)$.
\end{definition}

\begin{definition}\label{def:classical_symbols}
	We denote by ${\SG}_{[\A,\B],\mathrm{cl}}^{m,\mu}(\R^n)$ the space of all $p\in {\SG}_{[\A,\B]}^{m,\mu}(\R^n)$ satisfying the following conditions:
	\begin{enumerate}
		\item[$(I)$] There exists $\sum_{j\in\N_0}\chi_j(\xi)p_j(x,\xi)\in \FS_{[\A,\B],\xi}^{m,\mu}(\R^n)$ such that:
		\begin{itemize}
			\item $p_j(x,\xi)\in {\SG}_{[\A,\B],\mathrm{cl}(x)}^{m,[\mu-j]}(\R^n)$ for all $j\in\N_0$,
			
			\item $p(x,\xi)\sim_\xi \sum_{j\in\N_0}\chi_j(\xi)p_j(x,\xi)$ in $\FS_{[\A,\B],\xi}^{m,\mu}(\R^n)$, and
			
			\item $p(x,\xi)-\sum_{j<N}\chi(\xi)p_j(x,\xi)\in {\SG}_{[\A,\B],\mathrm{cl}(x)}^{m,\mu-N}(\R^n)$, for all $N\in\N_0$.
		\end{itemize}
		
		\medskip
		
		\item[$(II)$] There exists $\sum_{k\in\N_0}\phi_k(x)q_k(x,\xi)\in \FS_{[\A,\B],x}^{m,\mu}(\R^n)$ such that:
		\begin{itemize}
			\item $q_k(x,\xi)\in {\SG}_{[\A,\B],\mathrm{cl}(\xi)}^{[m-k],\mu}(\R^n)$ for all $k\in\N_0$,
			
			\item $p(x,\xi)\sim_x \sum_{k\in\N_0}\phi_k(x)q_k(x,\xi)$ in $\FS_{[\A,\B],x}^{m,\mu}(\R^n)$, and
			
			\item $p(x,\xi)-\sum_{k<N}\chi_k(x)q_k(x,\xi)\in {\SG}_{[\A,\B],\mathrm{cl}(\xi)}^{m-N,\mu}(\R^n)$, for all $N\in\N_0$.
		\end{itemize}   
	\end{enumerate}
	
	An element $p\in {\SG}_{[\A,\B],\mathrm{cl}}^{m,\mu}(\R^n)$ is said to be a SG-classical ultradifferentiable symbol.
\end{definition}

We have the following straightforward inclusions:
\begin{equation}\label{incl_cl_SG}
	{\SG}_{[\A,\B],\mathrm{cl}(\xi)}^{[m],\mu}, {\SG}_{[\A,\B],\mathrm{cl}(x)}^{m,[\mu]} \subset {\SG}_{[\A,\B],\mathrm{cl}}^{m,\mu}.
\end{equation}

For any symbol $p$ admitting expansions in homogeneous terms in $x$, respectively, in $\xi$, as above, with arbitrary $j,k\in\N_0$,
we set $\sigma_\psi^{\mu-j}(p) = p_j$, respectively, $\sigma_e^{m-k}(p) = q_k$. By \eqref{incl_cl_SG}, for each $j,k\in\N_0$, we can also define
\[\sigma_e^{m-k}(\sigma_\psi^{\mu-j}(p))\quad\text{and}\quad \sigma_\psi^{\mu-j}(\sigma_e^{m-k}(p)).\]

One can show, as in the standard SG-symbols theory, that 
$\sigma_e^{m-k}(\sigma_\psi^{\mu-j}(p)) = \sigma_\psi^{\mu-j}(\sigma_e^{m-k}(p))$, and we denote this, as usual, by $\sigma_{\psi e}^{\mu-j,m-k}(p)$. 

\begin{definition}
	Let $p\in {\SG}_{[\A,\B],\mathrm{cl}}^{m,\mu}(\R^n)$. The triple $(p_\psi,p_e,p_{\psi e})=(\sigma_\psi^{\mu}(p),\sigma_e^{m}(p),\sigma_{\psi e}^{m,\mu}(p))$ 
	is called the principal (homogeneous) symbol of $p$. We also say that:
	\begin{enumerate}
		\item $p_\psi=\sigma_\psi^{\mu}(p)$ is the interior principal symbol of $p$;
		
		\item $p_e=\sigma_e^{m}(p)$, is the homogeneous exit principal symbol of $p$;
		
		\item $p_{\psi e}=\sigma_{\psi e}^{m,\mu}(p)$ is the bi-homogeneous principal symbol of $p$.
	\end{enumerate}
\end{definition}

By the above definitions, we see that
\begin{equation*}
	p-\chi_\xi\sigma_\psi^{\mu}(p)\in {\SG}_{[\A,\B],\cl}^{m,\mu-1}(\R^n),\ \  
	p-\phi_x\sigma_\psi^{\mu}(p)\in {\SG}_{[\A,\B],\cl}^{m-1,\mu}(\R^n),\ \text{and}
\end{equation*}
\begin{equation*}
	p-\chi_\xi\sigma_\psi^{\mu}(p) - \phi_x\sigma_e^{m}(p) + \chi_\xi\phi_x\sigma_{\psi e}^{m,\mu}(p) \in {\SG}_{[\A,\B],\cl}^{\mu-1,m-1}(\R^n).
\end{equation*}

Analogously to \cite[Theorem 3.2.9]{NicRod}, in the next Proposition \ref{prop:sgclell} we characterize elliptic SG-classical ultradifferential symbols.

\begin{proposition}\label{prop:sgclell}
	A symbol $p\in {\SG}_{[\A,\B],\cl}^{m,\mu}(\R^n)$ is elliptic if and only if the following conditions hold:
	\begin{enumerate}
		\item $\sigma_\psi^{\mu}(p)(x,\xi)\neq 0$, for all $x,\xi\in\R^n$ with $\xi\neq 0$;
		
		\item $\sigma_e^{m}(p)(x,\xi)\neq 0$, for all $x,\xi\in\R^n$ with $x\neq 0$;
		
		\item $\sigma_{\psi e}^{m,\mu}(p)(x,\xi)\neq 0$, for all $x,\xi\in\R^n$ with $x\neq 0$ and $\xi\neq 0$.
	\end{enumerate}
\end{proposition}

\subsection{Global wave-front sets for ultradifferentiable distributions}\label{sec:globwfs}

Now we assume $\A=\B=\M$ and denote ${\SG}_{[\A,\B],\cl}^{m,\mu}(\R^n)$ simply by ${\SG}_{[\M],\cl}^{m,\mu}(\R^n)$. We denote by $\Op {\SG}_{[\M],\cl}^{m,\mu}(\R^n)$ the space of operators with symbol in ${\SG}_{[\M],\cl}^{m,\mu}(\R^n)$ and set
\[\Op {\SG}_{\cl}^{[\M]}(\R^n) = \bigcup_{m,\mu\in\R}\Op {\SG}_{[\M],\cl}^{m,\mu}(\R^n).\]

\begin{definition}
	Given $x_0\in\R^n$, we denote by $\RR_{x_0}^{[\M]}(\R^n)$ the set of the functions $\phi\in\D_{[\M]}(\R^n)$ such that $0\leq\phi\leq 1$ and $\phi\equiv 1$ in a neighbourhood of $x_0$.
\end{definition}

\begin{definition}
	Given $\xi_0\in\R^n\setminus\{0\}$, we denote by $\ZZ_{\xi_0}^{[\M]}(\R^n)$ the set of the functions $\psi\in\E_{[\M]}(\R^n)$ satisfying the following properties:
	\begin{itemize}
		\item $\psi(\lambda\xi)=\psi(\xi)$ for all $\lambda\geq 1$ and $\xi\in\R^n$ away from the origin;
		
		\item $0\leq\psi\leq 1$ and $\psi\equiv 1$ in a conic neighbourhood $\mathcal{C}_{\xi_0}$ of $\xi_0$ and $\psi\equiv 0$ outside a conic neighbourhood $\mathcal{C}_{\xi_0}'\supset \mathcal{C}_{\xi_0}$ of $\xi_0$;
		
		\item There exists $C,h>0$ (respectively, for every $h>0$, there exists $C>0$) such that
		\begin{equation*}
			|\partial_\xi^\alpha\psi(\xi)|\leq Ch^{|\alpha|}M_{|\alpha|}\langle\xi\rangle^{-|\alpha|},
		\end{equation*}
		for all $\xi\in\R^n$ and $\alpha\in\N_0^n$. 
	\end{itemize}
\end{definition}

\begin{definition}
	For a temperate ultradistribution $u\in\caS_{[\M]}'(\R^n)$, we define the ultradifferentiable global wave front set $\mathrm{WF}_{[\M]}(u)$ as the disjoint union
	\begin{equation*}
		\mathrm{WF}_{[\M]}(u) = \mathrm{WF}_{[\M]}^{\psi}(u) \sqcup \mathrm{WF}_{[\M]}^e(u) \sqcup \mathrm{WF}_{[\M]}^{\psi e}(u),
	\end{equation*}
	where the three components are defined as follows:
	\begin{itemize}
		\item We say that $(x_0,\xi_0)\in\R^{n}\times(\R^n\setminus\{0\})$ does not belong to $\mathrm{WF}_{[\M]}^{\psi}(u)$ if there exist $\phi_{x_0}\in\RR_{x_0}^{[\M]}$ and $\psi_{\xi_0}\in \ZZ_{\xi_0}^{[\M]}$ such that $\Op(\psi_{\xi_0})(\phi_{x_0}u)\in\caS_{[\M]}(\R^n)$;
		
		\item We say that $(x_0,\xi_0)\in(\R^n\setminus\{0\})\times \R^{n}$ does not belong to $\mathrm{WF}_{[\M]}^{e}(u)$ if there exist $\psi_{x_0}\in\ZZ_{x_0}^{[\M]}$ and $\phi_{\xi_0}\in \RR_{\xi_0}^{[\M]}$ such that $\Op(\phi_{\xi_0})(\psi_{x_0}u)\in\caS_{[\M]}(\R^n)$;
		
		\item We say that $(x_0,\xi_0)\in(\R^n\setminus\{0\})\times(\R^n\setminus\{0\})$ does not belong to $\mathrm{WF}_{[\M]}^{\psi e}$ if there exist $\psi_{x_0}\in\ZZ_{x_0}^{[\M]}$ and $\psi_{\xi_0}\in\ZZ_{\xi_0}^{[\M]}$ such that $\Op(\psi_{\xi_0})(\psi_{x_0}u)\in\caS_{[\M]}(\R^n)$.
	\end{itemize}
\end{definition}

\begin{definition}
	Given $P= \Op(p)$, with $p\in\SG_{[\M],\mathrm{cl}}^{m,\mu}(\R^n)$, we define
	\begin{itemize}
		\item $\mathrm{Char}_\psi(P) = \{(x,\xi)\in \R^n\times (\R^n\setminus\{0\}) \,:\, [\sigma_\psi^{\mu}(p)](x,\xi)=0\}$;
		
		\item $\mathrm{Char}_e(P) = \{ (x,\xi)\in (\R^n\setminus\{0\})\times\R^n \,:\, [\sigma_e^{m}(p)](x,\xi)=0 \}$;
		
		\item $\mathrm{Char}_{\psi e}(P) = \{(x,\xi)\in (\R^n\setminus\{0\})\times (\R^n\setminus\{0\}) \,:\, [\sigma_{\psi e}^{m,\mu}(p)](x,\xi)=0\}$.
	\end{itemize}
\end{definition}

The following results are analogous to those obtained in \cite{CapRod2006} for Gelfand-Shilov spaces of Gevrey type (see also \cite{Capp2006}), and we omit the proofs.

\begin{theorem}
	Let $u\in\caS_{[\M]}'(\R^n)$. Then
	\begin{itemize}
		\item $\mathrm{WF}_{[\M]}^{\psi}(u) = \bigcap\mathrm{Char}_\psi(P)$,
		
		\item $\mathrm{WF}_{[\M]}^{e}(u) = \bigcap\mathrm{Char}_e(P)$,
		
		\item $\mathrm{WF}_{[\M]}^{\psi e}(u) = \bigcap\mathrm{Char}_{\psi e}(P)$,
	\end{itemize}
	where the intersection is taken over the $P\in\Op\SG_{[\M],\cl}^{0,0}(\R^n)$ such that $Pu\in\caS_{[\M]}(\R^n)$.
\end{theorem}

\begin{theorem}
	Let $u\in\caS_{[\M]}'(\R^n)$ and $P=\Op(p)$, with $p\in\SG_{[\A,\B],\mathrm{cl}}^{m,\mu}(\R^n)$. Then, we have the following inclusions:
	\begin{enumerate}
		\item $\mathrm{WF}_{[\M]}^{\psi}(Pu) \subset \mathrm{WF}_{[\M]}^{\psi}(u) \subset \mathrm{WF}_{[\M]}^{\psi}(Pu)\cup \mathrm{Char}_\psi(P)$;
		
		\item $\mathrm{WF}_{[\M]}^{e}(Pu) \subset \mathrm{WF}_{[\M]}^{e}(u) \subset \mathrm{WF}_{[\M]}^{e}(Pu)\cup \mathrm{Char}_e(P)$;
		
		\item $\mathrm{WF}_{[\M]}^{\psi e}(Pu) \subset \mathrm{WF}_{[\M]}^{\psi e}(u) \subset \mathrm{WF}_{[\M]}^{\psi e}(Pu)\cup \mathrm{Char}_{\psi e}(P)$.
	\end{enumerate}
\end{theorem}

\begin{proposition}
	Let $u\in\caS_{[\M]}'(\R^n)$. Then, $u\in\caS_{[\M]}(\R^n)$ if and only if
	\[\mathrm{WF}_{[\M]}^{\psi}(u) = \mathrm{WF}_{[\M]}^{e}(u) = \mathrm{WF}_{[\M]}^{\psi e}(u) = \varnothing.\]
\end{proposition}

Let $\Pi_1(x,\xi)=x$ be the projection on the first factor of $\R^{2n}=\R^n_x\times\R^n_\xi$. Then, we have
\[\Pi_x\left(\mathrm{WF}_{[\M]}^{\psi}(u)\right) = [\M]-\mathrm{sing\,supp}(u),\]
so the first component of $\mathrm{WF}_{[\M]}(u)$ comprises the ultradifferentiable regularity of $u$. The other two components encodes the directions 
on which $u$ does not decay exponentially. More precisely, we have the following result, with which we conclude this section.

\begin{proposition}
	If $x_0\notin \Pi_x\left(\mathrm{WF}_{[\M]}^{e}\cup \mathrm{WF}_{[\M]}^{\psi e} \right)$, then there exists $\psi\in\ZZ_{x_0}^{[\M]}(\R^n)$ 
	such that $\psi u\in\caS_{[\M]}(\R^n)$.
\end{proposition}


\section{Coordinate invariance}\label{sec_coord_inv}

In this and in the subsequent section, for $\varphi:V\to U$ diffeomorphism between two open subsets of $\mathbb{R}^n$ and any $u\in C^\infty(U)$, for brevity of notation we will write $\tilde u\in C^\infty(V)$ in place of $u\circ\varphi$. We begin with an extension to the ultradifferentiable setting of the well-known result, due to Schrohe \cite{Schrohe86}, of invariance of the class of SG-operators, under the action of suitable (admissible) diffeomorphisms. Such diffeomorphisms will then represent
change of variables between local charts of analytical versions of SG-manifolds or manifolds with ends (see the subsequent Section \ref{sec:compmfs}).

\begin{theorem}\label{coord_inv}
	Let $\varphi:V^\sharp\to U^\sharp$ be an analytic diffeomorphism between open subsets $U^\sharp,V^\sharp\subset\mathbb{R}^n$ and suppose that the two following conditions hold:
	\begin{itemize}
		\item[$(I1)$] there exist $C,h>0$ (respectively, for every $h>0$, there exists $C>0$) such that
		\begin{equation*}
			|\partial_x^\alpha\varphi(x)|\leq Ch^{|\alpha|}\alpha!\langle x\rangle^{1-|\alpha|},\ x\in V^\sharp,\quad\text{and}\quad |\partial_y^\alpha\varphi^{-1}(y)|\leq Ch^{|\alpha|}\alpha!\langle y\rangle^{1-|\alpha|},\ y\in U^\sharp,
		\end{equation*}
		for all $\alpha\in\mathbb{N}_0^n$.
		
		\item[$(I2)$] there exist open subsets $U,V$ of $U^\sharp,V^\sharp$, respectively, such that 
		the restriction $\varphi|_V$ is a diffeomorphism between $V$ and $U$ and there exists a constant $\varepsilon>0$ such that
		\[B(x;\varepsilon\langle x\rangle)\subset V^\sharp,\ \forall x\in V,\quad\text{and}\quad B(y;\varepsilon\langle y\rangle)\subset U^\sharp,\ \forall y\in U.\]
	\end{itemize}
	
	Let $\A=\Aj,\B=\Bj\preceq\M=\Mj$ be three non-quasianalytic weight sequences and let $p\in {\SG}^{m,\mu}_{[\A,\B]}(\mathbb{R}^n)$ and 
	$m,\mu\in\R$, be such that $\mathrm{supp}(p)\subset U\times\mathbb{R}^n$. Consider the operator
	\begin{equation*}
		\tilde P\tilde u(x) = [\mathrm{Op}(p)u](\varphi(x)),\quad \tilde u\in\caS_{[\M]}(V),\ \mathrm{supp}(\tilde{u})\subset V.
	\end{equation*}
	
	Then, there exists a symbol $q\in {\SG}^{m,\mu}_{[\A,\B]}(\mathbb{R}^n)$ with $\mathrm{supp}(p)\subset V\times\mathbb{R}^n$ such that $\tilde P=\mathrm{Op}(q)+R$, where $R$ is a $[\M]$-regularizing operator.
\end{theorem}

Before proving Theorem \ref{coord_inv}, we need some technical lemmas. From now on, we will always suppose that $\varphi$ is a diffeomorphism
satisfiying conditions $(I1)$ and $(I2)$ in the statement of Theorem \ref{coord_inv}.

\begin{lemma}\label{lemma_M}
	Given $x,y\in U$, with $|x-y|\leq\varepsilon\langle x\rangle$, $\varepsilon\in(0,1)$,
	define 
	\[
		\nabla\varphi(x,y)=\int_0^1\varphi'(y+t(x-y))\,\mathrm{d}t,
	\]
	where $\varphi'$ is the Jacobian matrix of $\varphi$. Then:
	\begin{enumerate}
		\item[(a)] $\varphi(x)-\varphi(y)=\nabla\varphi(x,y)\cdot(x-y)$;
		
		\item[(b)] there is a constant $0<k\leq\varepsilon$ such that $\nabla\varphi(x,y)$ is invertible for $|x-y|\leq k\langle x\rangle$;
		
		\item[(c)] there is a function $\chi(x,y)\in \AG_{[\A,\B]}^{0,0,0}(\mathbb{R}^n)$ such that
		\[\chi(x,y)\equiv 1,\text{ if } |x-y|\leq \frac{k}{2}\langle x\rangle \quad \text{and} \quad  \chi(x,y)\equiv 0,\text{ if } |x-y|> k\langle x\rangle.\]
	\end{enumerate}
\end{lemma}
\begin{proof}
	The proof is identical to that of \cite[Lemma 2.4]{Schrohe86}. Point (a) follows from the mean value theorem, while point (b) follows from the fact that 
	$\nabla\varphi(x,x)=\varphi'(x)$ and $(I1)$ holds. Finally, for point (c), we choose a function $g\in\mathcal{D}_{[\B]}(\mathbb{R})$ such that $g(t)=1$ 
	if $t\leq 1/2$ and $g(t)=0$ if $t>1$. Then, the function 
	\[
		\chi(x,y) = g\left(\frac{|x-y|}{k\langle x\rangle}\right)
	\]
	satisfies the desired properties.
\end{proof}
With $\chi(x,y)$ as in Lemma \ref{lemma_M}, write
\[
	p(x,\xi)=\chi(x,y)p(x,\xi)+(1-\chi(x,y))p(x,\xi)=p_1(x,y,\xi)+p_2(x,y,\xi).
\]
It is clear that $p_1\in \AG_{[\A,\B]}^{m,0,\mu}(\mathbb{R}^n)$. Also, we have that $\mathrm{Op}(p_2)=\tilde{p}(x,D)+R$, 
where $R$ is $[M_k]$-regularizing and $\tilde{p}\in {\SG}_{[\A,\B]}^{m,\mu}(\R^n)$ is such that
\[\tilde{p}\sim \sum_{j=0}^\infty \tilde{p}_j,\]
where
\[\tilde{p}_j(x,\xi) = \sum_{|\alpha|=j}\frac{1}{\alpha!}\partial_\xi^\alpha D_y^\alpha p_2(x,y,\xi)|_{y=x}.\]

Observing that
\[\partial_\xi^\alpha D_y^\alpha p_2(x,y,\xi)|_{y=x} = \partial_\xi^\alpha p(x,\xi) D_y^\alpha (1-\chi(x,y))|_{y=x} = 0,\]
in view of the fact that, by cnstruction, $1-\chi(x,y)$ equals zero in a neighbourhood of the diagonal $x=y$. Hence, $\tilde{p}\sim 0$, which gives us that $\mathrm{Op}(p_2)$ is $[\M]$-regularizing and has kernel $K\in\caS_{[\M]}(\mathbb{R}^{2n})$. Moreover, the operator $\tilde{P}_2\tilde{u}(x)=[\mathrm{Op}(p_2)u](\varphi(x))$ has kernel $\tilde{K}(x,y)=K(\varphi(x),\varphi(y))$, which also belongs to $\caS_{[\M]}(\mathbb{R}^{2n})$ by $(I1)$ and the chain rule. Then, we only need to study the behaviour of $\tilde{P}_1\tilde{u}(x)=[\mathrm{Op}(p_1)u](\varphi(x))$.

For $u\in\caS_{[\M]}(\R^n)$, $\mathrm{supp}(u)\subset U$, employing the well-known \textit{Kuranishi's trick}, we find
\begin{align*}
	& [\mathrm{Op}(p_1)u](\varphi(x)) =\\
	& = \int_{\mathbb{R}^n}\int_{\mathbb{R}^n} e^{i\eta\cdot(\varphi(x)-y')}p_1(\varphi(x),y',\eta)u(y')\,\mathrm{d} y'\dbar\eta\\
	& = \int_{\mathbb{R}^n}\int_{\mathbb{R}^n} e^{i\eta\cdot(\varphi(x)-\varphi(y))}p_1(\varphi(x),\varphi(y),\eta)u(\varphi(y))|\det\varphi'(y)|\,\mathrm{d} y\dbar\eta\\
	& = \int_{\mathbb{R}^n}\int_{\mathbb{R}^n} e^{i\eta\cdot(x-y)}p_1(\varphi(x),\varphi(y), \nabla\varphi(x,y)^{-T}\xi)|\det\varphi'(y)||\det \nabla\varphi(x,y)^{-T}| u(\varphi(y)) \,\mathrm{d} y\dbar\xi,
\end{align*}
where we used the substitutions $y'=\varphi(y)$ and $\xi= \nabla\varphi(x,y)^T\eta$, since $\nabla\varphi(x,y)$, given in Lemma \ref{lemma_M},
is invertible in the support of $p_1$, for a sufficiently small $k\in(0,\varepsilon]$. It remains to show that
\[q(x,y,\xi) = p_1(\varphi(x),\varphi(y), \nabla\varphi(x,y)^{-T}\xi) u(\varphi(y)) |\det\varphi'(y)| |\det \nabla\varphi(x,y)^{-T}|\]
is a symbol in $\ST_{[\A,\B]}(\mathbb{R}^n)$ and induces an operator with symbol in ${\SG}_{[\A,\B]}^{m,\mu}(\mathbb{R}^n)$.

\begin{definition}
	Given $t\in[0,1]$, we set:
	\begin{itemize}
		\item $F_t(x,y,\xi) = p_1(\varphi(x),\varphi(x+t(y-x)), \nabla\varphi(x,x+t(y-x))^{-T}\xi)$;
		\item $G_t(x,y) = |\det\varphi'(x+t(y-x))|$;
		\item $H_t(x,y) = |\det \nabla\varphi(x,x+t(y-x))^{-T}|$
	\end{itemize}
\end{definition}

\begin{lemma}\label{est_FGH}
	For every fixed $t\in[0,1]$, there exist $C,h>0$ 
	(respectively, for every $h>0$, there exists $C>0$) such that, for all $\alpha,\beta,\gamma\in\mathbb{N}_0^n$ and $x,y,\xi\in\mathbb{R}^n$, we have:
	\begin{enumerate}
		\item[(a)] $|\partial_x^\beta\partial_y^\gamma G_t(x,y)|\leq Ch^{|\beta|+|\gamma|}B_{|\beta|}B_{|\gamma|}\langle x\rangle^{-|\beta|-|\gamma|}\langle x-y\rangle^{|\beta|+|\gamma|}$;
		
		\item[(b)] $|\partial_x^\beta\partial_y^\gamma H_t(x,y)|\leq Ch^{|\beta|+|\gamma|}B_{|\beta|}B_{|\gamma|}\langle x\rangle^{-|\beta|-|\gamma|}\langle x-y\rangle^{|\beta|+|\gamma|}$;
		
		\item[(c)] $|\partial_\xi^\alpha\partial_x^\beta\partial_y^\gamma F_t(x,y,\xi)|\leq Ch^{|\alpha|+|\beta|+|\gamma|}A_{|\alpha|}B_{|\beta|}B_{|\gamma|} \langle\xi\rangle^{m_1-|\alpha|} \langle x\rangle^{m_2-|\beta|-|\gamma|}\langle x-y\rangle^{|\beta|+|\gamma|}$.
	\end{enumerate}
\end{lemma}
\begin{proof}
	First, by Peetre's inequality, we have that
	\begin{equation}\label{peetre}
		\langle x-t(y-x)\rangle^{-j} \leq 2^j\langle x\rangle^{-j}\langle x-y\rangle^{j},
	\end{equation}
	for all $t\in[0,1]$, $x,y\in\mathbb{R}^n$, and $j\geq 0$.
	
	\medskip\noindent
	(a) Write
	\[\varphi' = \begin{bmatrix} \partial_{x_1}\varphi_1 & \cdots & \partial_{x_n}\varphi_1 \\ \vdots & \ddots & \vdots \\ \partial_{x_1}\varphi_n & \cdots & \partial_{x_n}\varphi_n \end{bmatrix} = \begin{bmatrix} \varphi_1^{(1)} & \cdots & \varphi_1^{(n)} \\ \vdots & \ddots & \vdots \\ \varphi_n^{(1)} & \cdots & \varphi_n^{(n)} \end{bmatrix}.\]
	
	Then, we have
	\[\det\varphi' = \sum_{\sigma\in S_n}\mathrm{sgn}(\sigma)\varphi_1^{(\sigma(1))}\cdots\varphi_n^{(\sigma(n))},\]
	where $S_n$ is the symmetric group of order $n$, which has $n!$ elements. By $(I1)$, there exist $C,h>0$ such that
	\[|\partial_x^\alpha\varphi_\ell^{(\sigma(\ell))}(x)|\leq Ch^{|\alpha|}\alpha!\langle x\rangle^{-|\alpha|},\]
	for all $\alpha\in\mathbb{N}_0^n$, $x\in\mathbb{R}^n$, $\ell=1,\dots,n$, and $\sigma\in S_n$. Let $\alpha_1,\cdots,\alpha_n\in\mathbb{N}_0^n$ be such that $\alpha_1+\cdots+\alpha_n=\alpha$. Then, we have
	\begin{equation*}
		|\partial_x^{\alpha_1}\varphi_1^{(\sigma(1))}\cdots \partial_x^{\alpha_n}\varphi_n^{(\sigma(n))}| \leq (Ch^{|\alpha_1|}\alpha_1!\langle x\rangle^{-|\alpha_1|})\cdots (Ch^{|\alpha_n|}\alpha_n!\langle x\rangle^{-|\alpha_n|}) \leq C^nh^{|\alpha|}\alpha!\langle x\rangle^{-|\alpha|}.
	\end{equation*}
	
	Given $j,k\in\{1,\dots,n\}$, we also see that
	\[\partial_{x_j}\det\varphi'(x) = 
	\sum_{\sigma\in S_n}\mathrm{sgn}(\sigma)\sum_{\ell=1}^n \varphi_1^{(\sigma(1))}\cdots \partial_{x_j}\varphi_{\ell}^{(\sigma(\ell))}\cdots \varphi_n^{(\sigma(n))},\]
	and
	\[\partial_{x_k}\partial_{x_j}\det\varphi'(x) = \sum_{\sigma\in S_n}\mathrm{sgn}(\sigma)\sum_{\nu=1}^n \sum_{\ell=1}^n \varphi_1^{(\sigma(1))}\cdots \partial_{x_k}\varphi_{\nu}^{(\sigma(\nu))} \cdots \partial_{x_j}\varphi_{\ell}^{(\sigma(\ell))}\cdots \varphi_n^{(\sigma(n))}.\]
	
	By induction, it follows that $\partial_x^\alpha\det\varphi'(x)$ is a linear combination of $n!n^{|\alpha|}$ terms of the form
	\[\mathrm{sgn}(\sigma)\partial_x^{\alpha_1}\varphi_1^{(\sigma(1))}\cdots \partial_x^{\alpha_n}\varphi_n^{(\sigma(n))},\]
	with $\alpha_1+\cdots+\alpha_n=\alpha$. Hence, we conclude
	\[|\partial_x^\alpha\det\varphi'(x)|\leq C_0h_0^{|\alpha|}\alpha!\langle x\rangle^{-|\alpha|},\]
	for all $\alpha\in\mathbb{N}_0^n$ and $x\in\mathbb{R}^n$, where $C_0=C^nn!$ and $h_0=nh$.
	From this estimate and \eqref{peetre}, we obtain $C,h>0$ such that
	\[|\partial_x^\beta\partial_y^\gamma G_t(x,y)|\leq Ch^{|\beta|+|\gamma|}(\beta!\gamma!)\langle x\rangle^{-|\beta|-|\gamma|}\langle x-y\rangle^{|\beta|+|\gamma|},\]
	for all $\beta,\gamma\in\mathbb{N}_0^n$, $x,y\in\mathbb{R}^n$, and $t\in [0,1]$. Notice that the absolute value in the definition of $G_t(x,y)$ does not affect the smoothness of the map since $\varphi$ is a diffeomorphism. Finally, we apply Proposition \ref{incl_quasi_factorial} to obtain the desired estimate.
	
	\medskip\noindent
	(b) Using the notation $\varphi_{\ell}^{(k)}=\partial_{x_k}\varphi_\ell$ from the previous point, by $(I1)$ and \eqref{peetre} it follows
	\begin{align*}
		\left|\partial_x^\beta\int_0^1\varphi_{\ell}^{(\sigma(\ell))}(x+t(y-x))\,\mathrm{d}t\right| & 
		\leq \int_0^1\left|\partial_x^\beta \varphi_{\ell}^{(\sigma(\ell))}(x+t(y-x))\right|\,\mathrm{d}t \\
		& \leq C(2h)^{|\beta|}\beta!\langle x\rangle^{-|\beta|}\langle x-y\rangle^{|\beta|}.
	\end{align*}
	
	For the $y$-derivatives, we obtain analogous estimates. Then, the estimates for $H_t(x,y)$ follow from those obtained in (a) 
	and Proposition \ref{incl_quasi_factorial}.
	
	\medskip\noindent
	(c) Fix $t\in [0,1]$ and write $X=(X_1,\dots,X_n)$, $Y=(Y_1,\dots,Y_n)$, $\Xi=(\Xi_1,\dots,\Xi_n)$, where
	\[X_j=\varphi_j(x),\quad Y_j=\varphi_j(x+t(y-x)),\quad\text{and}\quad \Xi_j=\sum_{k=1}^n \nabla\varphi(x,y)_{jk}^{-T}\xi_j,\quad j=1,\dots,n.\]
	
	In the sequel, we write
	\[
		p_1^\sharp(x,y,\xi)=p_1(\varphi(x),\varphi(x+t(y-x)),\nabla\varphi(x,y)^{-T}\xi)=p_1(X,Y,\Xi).
	\]
	
	By the chain rule, 
	\begin{equation}\label{dxp}
	\begin{aligned}
		\partial_{x_j}p_1^\sharp &= \sum_{k=1}^n \frac{\partial p_1}{\partial X_k}\frac{\partial X_k}{\partial x_j} + 
		\sum_{k=1}^{n}\frac{\partial p_1}{\partial Y_k}\frac{\partial Y_k}{\partial x_j} + 
		\sum_{k=1}^{n}\frac{\partial p_1}{\partial \Xi_k}\frac{\partial \Xi_k}{\partial x_j},
	\\
		\partial_{y_j}p_1^\sharp &= \sum_{k=1}^{n}\frac{\partial p_1}{\partial Y_k}\frac{\partial Y_k}{\partial y_j} + \sum_{k=1}^{n}\frac{\partial p_1}{\partial \Xi_k}\frac{\partial \Xi_k}{\partial y_j},
	\\
		\partial_{\xi_j}p_1^\sharp &= \sum_{k=1}^n \frac{\partial p_1}{\partial\Xi_k}\nabla\varphi(x,y)_{kj}^{-T}.
	\end{aligned}
	\end{equation}
	
	Observe that, by $(I1)$, it holds
	\[
		\langle\varphi(x)\rangle \asymp \langle x\rangle\quad\text{and}\quad \langle \nabla\varphi(x,y)^{-T}\xi\rangle \asymp \langle\xi\rangle.
	\]
	
	Then, the fact that $p_1\in \AG_{[\A,\B]}^{m,0,\mu}(\mathbb{R}^n)$ and \eqref{peetre}, together with $(I1)$, imply
	\begin{equation*}
		\left|\frac{\partial p_1}{\partial X_j}\right| \leq Ch B_1\langle x\rangle^{m-1}  \langle \xi\rangle^{\mu},\quad 
		\left|\frac{\partial p_1}{\partial Y_j}\right| \leq Ch B_1 \langle x\rangle^{m-1} \langle x-y\rangle \langle \xi\rangle^{\mu},
	\end{equation*}
	and
	\begin{equation*}
		\left|\frac{\partial p_1}{\partial \Xi_j}\right| \leq Ch A_1 \langle x\rangle^{m} \langle\xi\rangle^{\mu-1}.
	\end{equation*}
	
	We also have, again by $(I1)$,
	\begin{equation*}
		\left|\frac{\partial X_j}{\partial x_j}\right| \leq Ch B_1 ,\quad 
		\left|\frac{\partial Y_j}{\partial x_j}\right| \leq Ch B_1, 
	\end{equation*}
	\begin{equation*}
		\left| \frac{\nabla\varphi(x,y)_{k\ell}^{-T}}{\partial x_j}\right| \leq Ch B_1\langle x\rangle^{-1}\langle x-y\rangle,
		\quad \left|\frac{\partial Y_j}{\partial y_j}\right| \leq Ch B_1
	\end{equation*}
	\begin{equation*}
		\left|\frac{\nabla\varphi(x,y)_{k\ell}^{-T}}{\partial y_j}\right| \leq Ch B_1\langle x\rangle^{-1}\langle x-y\rangle,\quad |\nabla\varphi(x,y)_{kj}^{-T}|\leq C.
	\end{equation*}
	
	Notice that the terms $\xi_\ell$ that arise when we differentiate $p_1^\sharp$ with respect to $x_j$ or $y_j$ are compensated by the decay given by 
	$\partial p_1^\sharp/\partial\Xi_\ell$. Then, combining the previous estimates, we have proven
	\begin{equation}\label{est_1st}
		\left| \partial^\alpha_x\partial^\beta_y\partial^\gamma_\xi p_1^\sharp(x,y,\xi) \right|\leq (C^23nh)h^1 A_{|\alpha|}B_{|\beta|}B_{|\gamma|} 
		 \langle x\rangle^{m-|\alpha|-|\beta|}\langle x-y\rangle^{|\alpha|+|\beta|} \langle\xi\rangle^{\mu-|\gamma|},
	\end{equation}
	for $|\alpha|+|\beta|+|\gamma|\leq 1$.
	
	If we differentiate \eqref{dxp} with respect to $x_\ell$, the Leibniz rule yields
	\begin{align*}\label{dxxp}
		\partial_{x_\ell}\partial_{x_j}p_1^\sharp = & \sum_{k=1}^n \left( \left(\partial_{x_\ell}\frac{\partial p_1}{\partial X_k}\right)\frac{\partial X_k}{\partial x_j}+\frac{\partial p_1}{\partial X_k}\frac{\partial^2 X_k}{\partial x_\ell\partial x_j}\right) + \sum_{k=1}^{n}\left( \left(\partial_{x_\ell}\frac{\partial p_1}{\partial Y_k}\right)\frac{\partial Y_k}{\partial x_j}+\frac{\partial p_1}{\partial Y_k}\frac{\partial^2 Y_k}{\partial x_\ell\partial x_j}\right)\\
		 + &\sum_{k=1}^{n}\left( \left(\partial_{x_\ell}\frac{\partial p_1}{\partial \Xi_k}\right)\frac{\partial \Xi_k}{\partial x_j}+\frac{\partial p_1}{\partial \Xi_k}\frac{\partial^2 \Xi_k}{\partial x_\ell\partial x_j}\right). 
	\end{align*}
	
	Since $p_1\in \AG_{[\A,\B]}^{m,0,\mu}$, it follows
	\[
	\frac{\partial p_1}{\partial X_k}(X,Y,\Xi)\in \AG_{[\A,\B]}^{m-1,0,\mu},\ 
	\frac{\partial p_1}{\partial Y_k}(X,Y,\Xi)\in \AG_{[\A,\B]}^{m,-1,\mu},\ 
	\frac{\partial p_1}{\partial \Xi_k}(X,Y,\Xi)\in \AG_{[\A,\B]}^{m,0,\mu-1}.
	\]
	
	Moreover, these symbols satisfy
	\begin{equation*}
		\left|\partial^\alpha_X\partial^\beta_Y\partial^\gamma_\Xi
		\frac{\partial p_1}{\partial X_k}\right|
		\leq Ch^{|\alpha|+|\beta|+|\gamma|+1} A_{|\alpha|}B_{|\beta|+1}B_{|\gamma|} 
		\langle x\rangle^{m-|\alpha|-1}\langle y\rangle^{-|\beta|} \langle \xi\rangle^{\mu-|\gamma|},
	\end{equation*}
	\begin{equation*}
		\left|\partial^\alpha_X\partial^\beta_Y\partial^\gamma_\Xi
		\frac{\partial p_1}{\partial Y_k}\right|
		\leq Ch^{|\alpha|+|\beta|+|\gamma|+1}A_{|\alpha|} B_{|\beta|}B_{|\gamma|+1} 
		\langle x\rangle^{m-|\alpha|}\langle y\rangle^{-|\beta|-1}\langle \xi\rangle^{\mu-|\gamma|},
	\end{equation*}
	\begin{equation*}
		\left|
		\partial^\alpha_X\partial^\beta_Y\partial^\gamma_\Xi
		\frac{\partial p_1}{\partial \Xi_k}\right|\leq Ch^{|\alpha|+|\beta|+|\gamma|+1}A_{|\alpha|+1} B_{|\beta|}B_{|\gamma|} 
		\langle x\rangle^{m-|\alpha|}\langle y\rangle^{-|\beta|}\langle \xi\rangle^{\mu-|\gamma|-1},
	\end{equation*}
	for $\alpha,\beta,\gamma\in\mathbb{N}_0^n$, where the constants $C,h>0$ are the same appearing in the previous estimates above.
	If we now apply \eqref{est_1st} to those symbols in place of $p_1^\sharp$, we obtain
	\begin{align*}
		\left| \partial_{x_\ell}\frac{\partial p_1}{\partial X_k} \right| &\leq (C^23nh)h^2  B_2 \langle x\rangle^{m-2}\langle \xi\rangle^{\mu},
	\\		
		\left| \partial_{x_\ell}\frac{\partial p_1}{\partial Y_k} \right| &\leq (C^23nh)h^2 B_2\langle x\rangle^{m-2}\langle x-y\rangle^2 \langle \xi\rangle^{\mu},
	\\
		\left| \partial_{x_\ell}\frac{\partial p_1}{\partial \Xi_k} \right| &\leq (C^23nh)h^2 A_1 B_1 \langle x\rangle^{m-1}\langle x-y\rangle^2.
	\end{align*}
	
	By $(I1)$ for the derivatives of $X_j$, $Y_j$, and $\Xi_j$, $\partial_{x_\ell}\partial_{x_j}p_1^\sharp$ is a sum of $2(3n)$ terms satisfying the desired estimates. 
	By an induction argument, a derivative of order $N$ is a sum of $2^{N-1}3n$ terms. Then, incorporating the factor $1/2$ within the constant $C_0=C^23nh$ and 
	$2^N$ within the constant $h$, we obtain the desired estimate for the $x$-derivatives. Similar arguments give the estimates for the derivatives with respect to 
	$y$ and $\xi$. This proves the estimates in the inductive case.
	
	The estimates for the projective case can be proved in a similar fashion, by choosing the constants 
	$h>0$ in an appropriate way, to absorb the extra constants that appear during the computations. The proof is complete.
\end{proof}

Observing that $q(x,y,\xi)=F_1(x,y,\xi)G_1(x,y)H_1(x,y)$, the Leibniz rule implies that there exist $C,h>0$ (respectively, for every $h>0$, there exists $C>0$) such that
\begin{equation*}
	|\partial_\xi^\alpha\partial_x^\beta\partial_y^\gamma q(x,y,\xi)| \leq Ch^{|\alpha|+|\beta|+|\gamma|} A_{|\alpha|}B_{|\beta|}B_{|\gamma|} \langle\xi\rangle^{\mu-|\alpha|} \langle x\rangle^{m-|\beta|-|\gamma|}\langle x-y\rangle^{|\beta|+|\gamma|},
\end{equation*}
for all $\alpha,\beta,\gamma\in\mathbb{N}_0^n$ and $(x,y,\xi)\in\mathbb{R}^{3n}$.

As discussed in \cite[Remark 2.11]{Schrohe86}, condition $(I1)$ implies that either $q(x,y,\xi)$ vanishes or $\langle x-y\rangle\lesssim \langle x\rangle$. 
Then, we conclude
\begin{equation*}
	|\partial_\xi^\alpha\partial_x^\beta\partial_y^\gamma q(x,y,\xi)| 
	\leq 
	Ch^{|\alpha|+|\beta|+|\gamma|} A_{|\alpha|} B_{|\beta|} B_{|\gamma|} \langle \xi \rangle^{\mu-|\alpha|} \langle x \rangle^{m}
	\leq 
	Ch^{|\alpha|+|\beta|+|\gamma|} A_{|\alpha|} B_{|\beta|} B_{|\gamma|} \langle \xi \rangle^{\mu} \langle x \rangle^{m},
\end{equation*}
for all $\alpha,\beta,\gamma\in\mathbb{N}_0^n$ and $(x,y,\xi)\in\mathbb{R}^{3n}$. This proves $q\in \ST_{[\A,\B]}(\mathbb{R}^n)$, by choosing 
$\kappa(j)\equiv \mu$ and $\lambda(j)\equiv m$, according to Definition \ref{ST_class}. Hence, $q$ is a symbol and defines a continuous linear 
operator from $\mathcal{S}_{[\M]}(\mathbb{R}^n)$ to itself.

Consider the Taylor expansion of $q(x,y,\xi)$ with base point $y=x$, namely,
\begin{equation}\label{taylor}
	q(x,y,\xi) = \sum_{|\alpha|\leq N}\frac{i^{|\alpha|}}{\alpha!}D_y^\alpha q(x,y,\xi)|_{y=x}(x-y)^\alpha + r_N(x,y,\xi),
\end{equation}
where
\begin{equation*}
	r_N(x,y,\xi) = i^{N+1}\frac{N+1}{\alpha!}\sum_{|\alpha|=N+1}(x-y)^\alpha\int_0^1 D_y^\alpha q(x,x+t(y-x),\xi)(1-t)^N\,\mathrm{d}t.
\end{equation*}

For each $\delta\in\mathbb{N}_0^n$, write
\begin{equation*}
	q_\delta(x,y,\xi) = \frac{i^{|\delta|}}{\delta!}D_\xi^\delta D_y^\delta q(x,y,\xi),
\end{equation*}
and
\begin{equation*}
	q_N(x,y,\xi)=i^{N+1}\frac{N+1}{\delta!}\sum_{|\delta|=N+1}\int_0^1 D_\xi^\delta D_y^\delta q(x,x+t(y-x),\xi)(1-t)^N\,\mathrm{d}t.
\end{equation*}

We will show that $q$ defines an operator with symbol $\tilde{q}(x,\xi)\sim \sum_{\delta}\tilde{q}_\delta(x,\xi)$, where
\[\tilde{q}_\delta(x,\xi) = q_\delta(x,y,\xi)|_{y=x}.\]

Since $q_\delta=i^{|\delta|}(\delta!)^{-1}D_\xi^\delta D_y^\delta (F_1G_1H_1)$, Lemma \ref{est_FGH} yields 
$\tilde{q}_\delta\in {\SG}_{[\A,\B]}^{m-|\delta|,\mu-|\delta|,}(\mathbb{R}^n)$. In particular, by Proposition 
\ref{prop_p_pj} there is a symbol $\tilde{q}\in {\SG}_{[\A,\B]}^{m,\mu}(\mathbb{R}^n)$ such that 
$\tilde{q}\sim \sum_{\delta}\tilde{q}_\delta$. Also, by Lemma \ref{est_FGH}, there exist $C,h>0$ (respectively, for every $h>0$, there exists $C>0$) such that
\begin{equation}\label{est_qN}
	|\partial_\xi^\alpha\partial_x^\beta\partial_y^\gamma q_N(x,y,\xi)|\leq C h^{|\alpha|+|\beta|+|\gamma|} A_{|\alpha|}B_{|\beta|}B_{|\gamma|} \langle\xi\rangle^{\mu-|\alpha|-N-1} \langle x\rangle^{m-|\beta|-|\gamma|-N-1} \langle x-y\rangle^{|\beta|+|\gamma|+N+1},
\end{equation}
for all $\alpha,\beta,\gamma\in\mathbb{N}_0^n$ and $(x,y,\xi)\in\mathbb{R}^{3n}$. These estimates imply estimates of the form considered in Remark \ref{remark_ST}, hence we can use $q_N$ as a symbol.

\begin{lemma}\label{lemma_cd}
	Let $\tilde{q}$ be a symbol with $\tilde{q}\sim \sum_{\delta}\tilde{q}_\delta$, and set $\tilde{q}_N = \sum_{|\delta|\leq N} \tilde{q}_\delta$. 
	Then $\mathrm{Op}(q)-\mathrm{Op}(\tilde{q})=\mathrm{Op}(r)$, where $r=(\tilde{q}-\tilde{q}_N)+q_N$, and there is $C,h>0$ such that
	\begin{equation}\label{est_qq_N}
		|\partial_\xi^\alpha\partial_x^\beta\partial_y^\gamma r(x,y,\xi)|\leq 
		C h^{|\alpha|+|\beta|+|\gamma|} A_{|\alpha|}B_{|\beta|}B_{|\gamma|} 
		\langle \xi\rangle^{\mu-|\alpha|-N} \langle x\rangle^{m-|\beta|-|\gamma|-N} \langle x-y\rangle^{|\beta|+|\gamma|+N},
	\end{equation}
	for all $\alpha,\beta,\gamma\in\mathbb{N}_0^n$ and $(x,y,\xi)\in\mathbb{R}^{3n}$.
\end{lemma}
\begin{proof}
	Write $\mathrm{Op}(q)-\mathrm{Op}(\tilde q) = \mathrm{Op}(q-\tilde{q}_N) - \mathrm{Op}(\tilde q-\tilde{q}_N)$. 
	By definition, we have $\tilde{q}-\tilde{q}_N \in {\SG}_{[\A,\B]}^{m-N,\mu-N}$, hence $\tilde{q}-\tilde{q}_N$ satisfies \eqref{est_qq_N}. 
	Moreover, for every $u\in\caS_{[\M]}(\mathbb{R}^n)$, we have
	\begin{equation*}
		\int_{\mathbb{R}^n}\int_{\mathbb{R}^n} e^{i\xi\cdot(x-y)}(q(x,y,\xi)-\tilde{q}_N(x,\xi))u(y)\,\mathrm{d}y\, \mathrm{d}\xi = 
		\int_{\mathbb{R}^n}\int_{\mathbb{R}^n} e^{i\xi\cdot(x-y)}q_N(x,y,\xi)u(y)\,\mathrm{d}y\, \mathrm{d}\xi,
	\end{equation*}
	by \eqref{taylor} and the definition of $q_N$. Finally, \eqref{est_qN} gives us that $q-\tilde{q}_N$ also satisfies \eqref{est_qq_N}.
\end{proof}

\begin{lemma}
	Let $r$ be the symbol defined in Lemma \ref{lemma_cd}. Then, as an operator on $\caS_{[\M]}(\mathbb{R}^n)$, $\mathrm{Op}(r)$ has kernel
	\begin{equation}\label{kernel_r}
		K(x,y)=\int_{\mathbb{R}^n} e^{i\xi\cdot (x-y)}r(x,y,\xi)\,\mathrm{d}\xi \in \caS_{[\M]}(\mathbb{R}^{2n}).
	\end{equation}
	
	That is, $\mathrm{Op}(r)$ is $[\M]$-regularizing.
\end{lemma}
\begin{proof}
	If we choose $N\in\mathbb{N}$ sufficiently large, the integral in \eqref{kernel_r} exists for every $x,y\in\mathbb{R}^n$. 
	Again, since we can choose $N$ arbitrarily large, we can differentiate $K$ under the integral sign. 
	By \eqref{est_qq_N}, we find $K\in\caS_{[\M]}(\mathbb{R}^{2n})$, which implies that $\mathrm{Op}(r)$ is $[\M]$-regularizing. 
\end{proof}

Finally, we have that $\mathrm{Op}(q)-\mathrm{Op}(\tilde q)=0$ modulo $[\M]$-regularizing operators. Since $\tilde q\in {\SG}_{[\A,\B]}^{m,\mu}(\mathbb{R}^n)$, 
we conclude that $\tilde{P}_1$ is an operator with symbol in ${\SG}_{[\A,\B]}^{m,\mu}(\mathbb{R}^n)$, and Theorem \ref{coord_inv} is proved.


\section{Admissible analytic manifolds}\label{sec:compmfs}

In this section we introduce the class of analytic manifolds which admit atlases compatible with the ultradifferential extensions
of the $\SG$-calculus constructed and studied in the previous sections, adapting the approach in \cite{Schrohe86}. 

Let $X$ be a $n$-dimensional manifold. Let $\{\Omega_k^\sharp\}_{k=1}^N$ be a finite open cover of $X$, and let 
$\varphi_k^\sharp:\Omega_k^\sharp\to U_k^\sharp\subset\mathbb{R}^n$ be homeomorphisms. 
We say that $\{(\Omega_k,\varphi_k)\}_{k=1}^N$ \textit{has a good shrinking} if there is another finite open cover $\{\Omega_k\}_{k=1}^N$ of $X$ 
with $\Omega_k\subset\Omega_k^\sharp$, $k=1,\dots,N$, with homeomorphisms
$\varphi_k=\varphi_k^\sharp|_{\Omega_k}:\Omega_k\to U_k\subset U_k^\sharp$, and a constant $\delta>0$ such that
\[B(y,\delta\langle y\rangle)\subset U_k^\sharp,\quad\forall y\in U_k,\ k=1,\dots,N.\]

We say that an atlas $\mathscr{A}$ for $X$ is admissible, or \textit{$\mathrm{SGA}$-compatible}, if the three following conditions hold:
\begin{enumerate}
	\item[$(SGA1)$] $\mathscr{A}$ is finite and analytic;
	\item[$(SGA2)$] $\mathscr{A}$ has a good shrinking;
	\item[$(SGA3)$] the changes of coordinates among the local charts of $\mathscr{A}$ satisfy $(I1)$ from Theorem \ref{coord_inv}.
\end{enumerate}

If conditions $(SGA1)-(SGA3)$ hold, then the changes of coordinates will satisfy condition $(I2)$ from Theorem \ref{coord_inv}, too, 
see \cite[Lemma 3.3]{Schrohe86}. For simplicity, we will write $\varphi_k$ instead of $\varphi_k^\sharp$.

We say that $X$ is an SGA-manifold if it admits an $\mathrm{SGA}$-compatible atlas. Trivially, $\R^n$ is an SGA-manifold. 
In the next subsection we present a relevant subclass of SGA-manifolds. In the next section, we show that also on the so-called
\textit{scattering (or asymptotically Euclidean) manifolds} (see \cite{Melrose_GST}) the same construction is possible, of course
under the hypotheses that the underlying manifold us analytic.

\subsection{Manifolds with ends}\label{subs:mwe}
Manifolds with ends are a well-known class of non-compact manifolds (see, e.g., \cite{Melrose_APS,Melrose_GST,MP2002,Schrohe86}), admitting 
$\SG$-admissible atlases. We here show that they also admit SGA-compatible atlases.

\begin{definition}\label{def_mfwe}
	Let $X$ be an $n$-dimensional analytic manifold that can be decomposed into
	\begin{equation}\label{eq:mwe}
		X=X_0\cup E,
	\end{equation}
	where:
	\begin{enumerate}
		\item $X_0$ is a $n$-dimensional, compact, submanifold of $X$ with boundary $\partial X_0$;
		
		\item the boundary $\partial X_0$ is a ($(n-1)$-dimensional, analytically embedded) connected submanifold of $X$;
		
		\item $E$ is a $n$-dimensional submanifold with boundary that is analytically diffeomorphic to $[1,+\infty)\times \partial X_0$,
		its boundary being identified with $\{1\}\times\partial X_0$;
		
		\item $\cup$ in \eqref{eq:mwe} means gluing by identification along the boundaries.
	\end{enumerate}
	
	Then, $X$ is called a manifold with one cylindrical end.
\end{definition}

\begin{center}
	\includegraphics[scale=0.3]{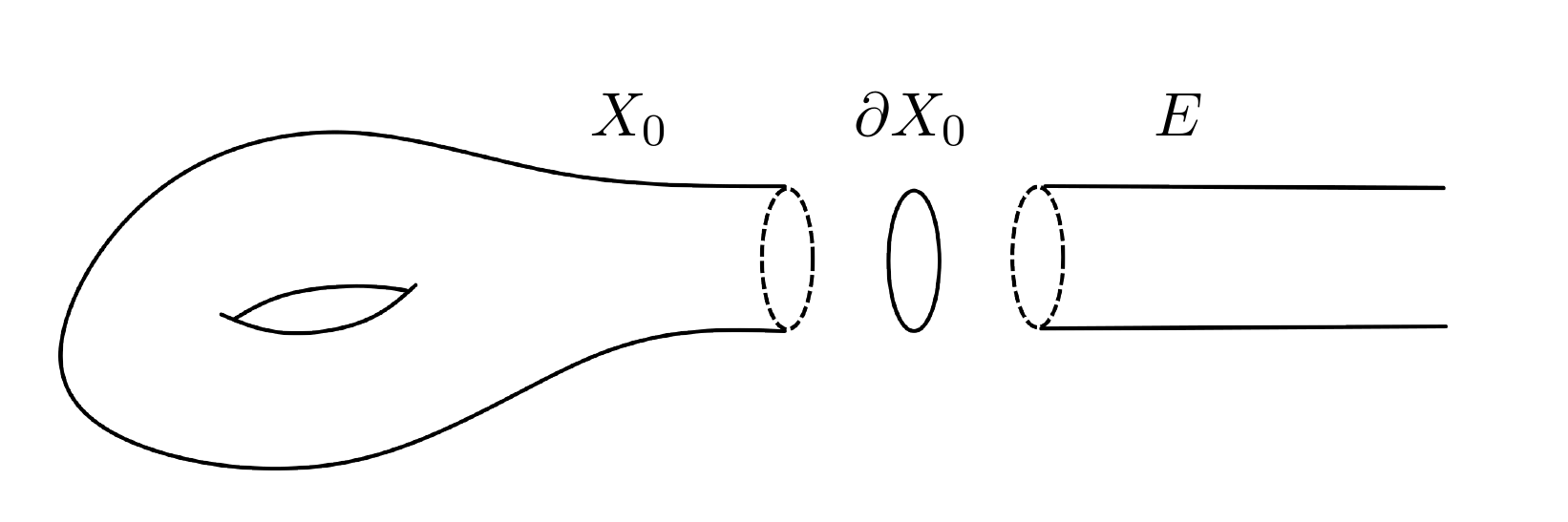}
	
	\textbf{Figure 1:} A manifold with one cylindrical end.
\end{center}

Let $\mathscr{A}=\{(\Omega_k^\sharp,\varphi_k)\}_{k\in K}$ be an analytic atlas on $X$ and $\mathscr{A}_0=\{(\Omega_k^\sharp,\varphi_k)\}_{k=1}^{N}$ be a finite subset of $\mathscr{A}$ consisting of local charts that cover the compact submanifold with boundary $X_0\cup\{\partial X_0\times [1,3]\}$. By shrinking the domains if necessary, we can assume that the $\Omega_k^\sharp$ are relatively compact.

Since $X_0$ is embedded, for each $p\in\partial X_0$, we can find $k\in\{1,\dots, N\}$ such that $p\in\Omega_k^\sharp$, and a subset $\Omega_k^0\subset \Omega_k^\sharp$ such that the restriction $\varphi_k|_{\Omega_k^0\cap\partial X_0}$ is a chart for $\partial X_0$. In this way, we can construct an atlas for $\partial X_0$, which is finite and analytic, in view of the fact that $\partial X_0$ is compact and the embedding in $X$ is analytic. We denote this atlas by $\tilde{\mathscr{A}}=\{(\tilde{\Omega}^\sharp_k,\tilde{\varphi}_k)\}_{k=1}^{N_0}$, for a suitable $N_0\in\mathbb{N}$.

Employing the atlas $\tilde{\mathscr{A}}$ on $\partial X_0$, we can introduce coordinates on the interior of $E$, identified with $(1,+\infty)\times \partial X_0$, as follows:
\[
	\psi_k: (1,+\infty)\times\tilde{\Omega}_k^\sharp \to \mathbb{R}^n,\quad \psi_k(t,y)=(t,t\varphi_j(y)),k=1,\dots,N_0.
\]

By compactness, for each $k=1,\dots,N_0$, we can obtain an open subset $\tilde{\Omega}_k\subset\tilde{\Omega}_k^\sharp$ such that $\{\tilde{\Omega}_k\}_{k=1}^{N_0}$ still covers $\partial X_0$ and each $\tilde{U}_k^\sharp=\tilde{\varphi}_k(\tilde{\Omega}_k^\sharp)$ contains an $\varepsilon$-neighbourhood of $\tilde{U}_k=\tilde{\varphi}_k(\tilde{\Omega}_k)$, for some fixed $\varepsilon>0$. Then, $(2,+\infty)\times \tilde{\Omega}_k$ is a \textit{good shrinking} of $(1,+\infty)\times\tilde{\Omega}_k^\sharp$. For the remaining charts, we obtain the \textit{good shrinking} by using compactness arguments. Hence, we have proved that 
the constructed atlas fulfills properties (SGA1) and (SGA2).

Arguing as in \cite[Example 3.4]{Schrohe86}, if $y=(\tilde{\varphi}_j\circ\tilde{\varphi}_k^{-1})(x)$ is a change of coordinates on $\partial X_0$, the corresponding coordinate changes on $E$ are of the form
\[(t,ty)\mapsto (t,x(y)t),\]
and the corresponding derivatives satisfy (SGA3), where the $\langle x\rangle$-estimates are given by the derivatives on $t$ and the 
constants $C,h>0$ comes from the analyticity of $x(y)$. The unbounded charts on $E$ are called \textit{exit charts}.

Finally, if $\varphi_k$ is a bounded analytic chart and $\psi_\ell$ is a coordinate map on the end $E$, in the intersection of the domains, which is contained in $E$, we can write $\varphi_k=(\varphi_{k,1},\varphi_{k,2})$, where $\varphi_{k,1}$ and $\varphi_{k,2}$ are analytic in an open subset of $\mathbb{R}$ and $\partial X_0$, respectively. Hence, it is clear that $\varphi_k\circ\psi_\ell^{-1}$ and $\psi_\ell\circ\varphi_k^{-1}$ are analytic. Since the intersection of the domains in this case is bounded, the requested estimates trivially hold true in this case as well. Then, (SGA3) also holds, and the constructed atlas is SGA-compatible.

The construction above extends to $n$-dimensional analytic manifolds $X$ that can be decomposed as a union
\begin{equation}\label{eq:mwne}
	X=X_0\cup E_1\cup\cdots\cup E_m,
\end{equation}
where $X_0$ is a $n$-dimensional, compact submanifold with boundary such that
\[\partial X_0 = \partial E_1\cup\cdots\cup\partial E_m,\]
where the boundary components $\partial E_1,\cdots,\partial E_m$ are disjoint, analytically embedded, compact, connected, $(n-1)$-dimensional submanifolds, each $E_\ell$ is 
analytically diffeomorphic to $ [1,+\infty)\times\partial E_\ell$, its boundary identified with $\{1\}\times \partial E_\ell$, $\ell=1,\dots,m$, and $\cup$
in \eqref{eq:mwne} means gluing along the corresponding boundary components. In this case, we say that $m$ is the number of ends of $X$.

\subsection{Partition of unity} 
We show how to construct a suitable partition of unity for an admissible atlas $\mathscr{A} = \{(\Omega_k^\sharp,\varphi_k^\sharp)\}_{k=1}^N$,
compatible with the analytic structure and the ultradifferential extension of the SG-calculus. 

With $\B=\Bj$ a non-quasianalytic weight sequence, choose a function $\Phi \in \mathcal{D}_{[\B]} (\mathbb{R}^n)$ such that $\int\Phi=1$.
For each $\varepsilon>0$, we set
\[\Phi_\varepsilon(x) = \varepsilon^{-n}\Phi(x/\varepsilon).\]

Then, each $\Phi_\varepsilon$ also belongs to $\mathcal{D}_{[\B]}(\R^n)$ and satisfies $\int\Phi_\varepsilon=1$. Given an open set $U\subset\mathbb{R}^n$ and a constant $\delta>0$, denote by $U_\delta$ the set $\bigcup_{x\in U}B(x,\delta\langle x\rangle)$, where $B(x,\delta\langle x\rangle)$ is the open ball centered in $x\in U$ with radius $\delta\langle x\rangle$.

\begin{lemma}\label{lem:anpartun}
	Let $U\subset\mathbb{R}^n$ be an open subset and $\chi_U$ be its characteristic function. Define $\Phi_U:(0,+\infty)\times\mathbb{R}^n\to\mathbb{R}$ by
	\[\Phi_U(\varepsilon,x) = \chi_U * \Phi_\varepsilon(x) = \int_{\mathbb{R}^n}\chi_U(y)\Phi_\varepsilon(x-y)\,\mathrm{d}y.\]
	Then, the following properties hold true:
	\begin{enumerate}[(a)]
		\item for every fixed $\varepsilon>0$, $\Phi_U(\varepsilon,\cdot)\in\mathcal{E}_{[\B]}(\mathbb{R}^n)$ and $0\leq\Phi_U(\varepsilon,x)\leq 1$ for all $x\in\mathbb{R}^n$;
		\item $\Phi_U(\varepsilon,0)=0$, if $\mathrm{dist}(x,U)\geq\varepsilon\langle x\rangle$;
		\item $\Phi_U(\varepsilon,x)=1$, if $B(x,\varepsilon\langle x\rangle)\subset U$;
		\item for every fixed $\varepsilon>0$, there exist $C,h>0$ (respectively, for every $h>0$, there exists $C>0$) such that
		\begin{equation}\label{est_phiU}
			|\partial_x^\beta\Phi_U(x)|\leq Ch^{|\alpha|}B_{|\beta|} \langle x \rangle^{-|\beta|},
		\end{equation}
		for all $\beta\in\mathbb{N}_0^n$. 
	\end{enumerate}
\end{lemma}
\begin{proof}
	Point (a) is consequence of the regularizing property of the convolution, while Points (b) and (c) are straightforward. 
	Concerning Point (d), since $\Phi_\varepsilon$ is compactly supported, we have
	\begin{align*}
		|\partial_x^\beta\Phi_U(x)| & \leq \int_{\mathbb{R}^n}\chi_U(y)|\partial_x^\alpha \Phi_\varepsilon(x-y)|\,\mathrm{d}y\\
		& \leq \sup_{y\in\mathbb{R}^n}|\partial_x^\alpha\Phi_\varepsilon(x-y)|\int_{\text{supp}(\Phi_\varepsilon(x-y))}\,\mathrm{d}y\\
		& \leq C_\varepsilon\sup_{x\in\mathbb{R}^n}|\partial_x^\alpha\Phi_\varepsilon(x)|,
	\end{align*}
	for all $\beta\in\N_0^n$, where $C_\varepsilon>0$ is the volume of the support of $\Phi_\varepsilon$. 
	Hence, it is enough to show that $\Phi_\varepsilon$ satisfies \eqref{est_phiU}. 
	Recall that $\Phi_\varepsilon$ decays exponentially. Then, for every $\beta\in\mathbb{N}_0^n$, we have that 
	$|\partial_x^{\beta}\Phi_\varepsilon|\leq \langle x\rangle^{-|\beta|}$ outside a fixed compact set independent of $\beta$, while inside such compact set we have
	\[|\partial_x^{\beta}\Phi_\varepsilon|\leq Ch^{|\beta|}B_{|\beta|},\]
	since $\Phi_\varepsilon\in\mathcal{E}_{[\B]}(\R^n)$. Combining both estimates, we obtain \eqref{est_phiU}.
\end{proof}

Given an open subset $U\subset\mathbb{R}^n$ and a constant $\varepsilon>0$, set $U_\varepsilon=\bigcup_{x\in U}B(x,\varepsilon\langle x\rangle)$. 
The next Corollary \ref{coro_Phi} is a straightforward consequence of Lemma \ref{lem:anpartun}.

\begin{corollary}\label{coro_Phi}
	Let $\Phi(x)=\Phi_{U_\varepsilon}(\varepsilon,x)$, for some open subset $U\subset\mathbb{R}^n$ and a constant $\varepsilon>0$. 
	Then, the following properties hold true:
	\begin{enumerate}[(a)]
		\item $\Phi\in {\SG}_{[\A,\B]}^{0,0}(\mathbb{R}^n)$ and $0\leq\Phi\leq 1$;
		\item $\Phi\equiv 1$ on $U$ and $\Phi\equiv 0$ outside $U_{2\varepsilon}$.
	\end{enumerate}
\end{corollary}

\begin{theorem}\label{thm_part}
	Let $X$ be a manifold with an SGA-admissible atlas $\mathscr{A}=\{(\Omega_k^\sharp,\varphi_k)\}_{k=1}^N$. 
	Then, there exists a partition of unity $\{\Phi_k\}_{k=1}^N$ of $X$, subordinate to the cover $\{\Omega_k^\sharp\}_{k=1}^N$, 
	such that, in local coordinates, it holds
	\begin{equation}\label{est_Phij}
		|\partial_x^\beta (\Phi_k)_*(x)|\leq Ch^{|\beta|}B_{|\beta|}\langle x\rangle^{-|\beta|},
	\end{equation}
	that is, $(\Phi_k)_*\in {\SG}_{[\A,\B]}^{0,0}(\R^n)$, where
	\begin{equation*}
		(\Phi_k)_*(x) = \begin{cases}
			\Phi_k(\varphi_k^{-1}(x)),& x\in U_k^\sharp\\
			0,& x\notin U_k^\sharp,
		\end{cases}
	\end{equation*}
\end{theorem}
\begin{proof}
	For each $k=1,\dots,N$, define on $U_k^\sharp$ the function $\Psi_k(x)=\Phi_{(U_k)_\varepsilon}(\varepsilon,x)$, 
	where $\varepsilon=\delta/3$. Here, $\delta>0$ is the constant from condition $(SGA2)$. Then, using the local charts, set
	\[\Psi_k^*(x) =
	\begin{cases}
		\Psi_k(\varphi_k(x)),& x\in\Omega_k^\sharp, \\
		0,& x\notin\Omega_k^\sharp.
	\end{cases}\]
	and, for each $k=1,\dots,N$, define
	\begin{equation*}
		\Phi_k:X\to \mathbb{R}\colon x\mapsto \Phi_k(x) = \frac{\Psi_k^*(x)}{\sum_{\ell=1}^N\Psi_\ell^*(x)}.
	\end{equation*}
	
	Notice that each $\Phi_k$ is well-defined, since the denominator is always bigger than $1$. Moreover, 
	it is clear, by the definition of $\Phi_k$, that $\sum_{k=1}^N\Phi_k=1$. Moreover, $0\leq \Phi_k\leq 1$, by Corollary \ref{coro_Phi}, (b). 
	Then, $\{\Phi_k\}_{k=1}^N$ is a partition of unity subordinated to $\{\Omega_k^\sharp\}_{k=1}^N$. Finally, \eqref{est_Phij} 
	is proved employing condition $(I1)$ and Corollary \ref{coro_Phi}, (a).
\end{proof}

\begin{corollary}\label{coro_Theta}
	Let $\{\Phi_k\}_{k=1}^N$ the partition of unity constructed in Theorem \ref{thm_part}. Then, there exist functions $\{\Theta_k\}_{k=1}^N$ such that:
	\begin{enumerate}[(a)]
		\item $0\leq \Theta_k\leq 1$, $\Theta_k\equiv 1$ on $\mathrm{supp}(\Phi_k)$, and $\Theta_k\equiv 0$ outside $\Omega_k^\sharp$;
		\item In local coordinates, $\Theta_k$ belongs to $ {\SG}_{[\A,\B]}^{0,0}(\mathbb{R}^n)$, that is, satisfies estimates of the form \eqref{est_Phij}.
	\end{enumerate}
	In particular, $\Theta_k\Phi_k=\Phi_k$, $k=1,\dots,N$.
\end{corollary}
\begin{proof}
	The proof is analogous to that of \cite[Lemma 3.10]{Schrohe86}.
\end{proof}


\section{Gelfand-Shilov spaces on SGA-manifolds}\label{sec:GSspcsGSAmf}

Let $X$ be a manifold with SGA-admissible atlas $\mathscr{A}=\{(\Omega_k^\sharp,\varphi_k)\}_{k=1}^N$. For each $k=1,\dots,N$, 
write $U_k^\sharp = \varphi_k(\Omega_k^\sharp)$. Also, let $\A=\Aj,\B=\Bj,\M=\Mj$ be three non-quasianalytic weight sequences 
such that $\A,\B\preceq\M$.
\begin{definition}
	Let $m,\mu\in\mathbb{R}$. We say that a function $p\in C^\infty(U_k^\sharp\times\mathbb{R}^n)$ belongs to ${\SG}_{[\A,\B]}^{m,\mu}(U_k)$ if, 
	for every $\Phi\in {\SG}_{[\A,\B]}^{0,0}(\mathbb{R}^n)$ with $\mathrm{supp}(\Phi)\subset U_k^\sharp\times\R^n$, setting $\tilde{p}(x,\xi)=\Phi(x)p(x,\xi)$,
	we have $\tilde{p}\in {\SG}_{[\A,\B]}^{m,\mu}(\mathbb{R}^n)$. In this case, we define 
	\[{\SG}_{[\A,\B]}^{m,\mu}(X)=\{(p_1,\dots,p_N)\,:\, p_k\in {\SG}_{[\A,\B]}^{m,\mu}(U_k),\ k=1,\dots,N\}.\]
\end{definition}

By Theorem \ref{inv_analytic}, the symbol classes ${\SG}_{[\A,\B]}^{m,\mu}(X)$ are well defined.

Let $\{\Omega_k\}_{k=1}^N$ be a \textit{good shrinking} of the cover $\mathscr{A}$ and write $\varphi_k(\Omega_k) = U_k\subset\mathbb{R}^n$. 
Its SGA-admissible structure, described in the previous section, allows us to (invariantly) define Gelfand-Shilov spaces on the manifold $X$.

\begin{definition}
	For each $h>0$, we define $\caS_{\M}(X;h)$ as the space of all $f\in C^\infty(X)$ such that
	\[\|f\|_{h}=\sum_{k=1}^N\|f\circ\varphi_k^{-1}\|_{h,U_k}<+\infty,\]
	where
	\begin{equation}\label{gelf_manif}
		\|f\|_{h,U_k} = \sup_{x\in U_k}\sup_{\alpha,\beta\in\mathbb{N}_0^n} 
		\frac{|x^\beta\partial_x^\alpha (f\circ\varphi_k^{-1})(x)|}{h^{|\alpha|+|\beta|}M_{|\alpha|}M_{|\beta|}} < +\infty,\quad k=1,\dots,N.
	\end{equation}
	
	We have that each $\caS_{\M}(X;h)$ is a Banach space with the norm $\|\cdot \|_{h}$. Then, we define
	\[\caS_{\{\M\}}(X) = \underset{h>0}{\mathrm{ind\,lim}}\, \caS_{\M}(X;h),\]
	\[\caS_{(\M)}(X) = \underset{h>0}{\mathrm{proj\,lim}}\, \caS_{\M}(X;h).\]
\end{definition}

\begin{remark}
	The spaces $\caS_{[\M]}(X)$ are well-defined as a consequence of condition $(I1)$ and the Fa\`a di Bruno formula. We clearly recover the usual Gelfand-Shilov spaces when $X=\R^n$.
\end{remark}

As in the case when $X=\R^n$, one can show that the inclusion map $\caS_{\M}(X;h)\hookrightarrow \caS_{\M}(X;h_+)$ is continuous if $h\leq h_+$ and compact if $h<h_+$. This gives us that $\caS_{\{\M\}}(X)$ and $\caS_{(\M)}(X)$ are, respectively, DFS and FS spaces. We denote by $\caS_{[\M]}'(X)$ the topological dual of $\caS_{[\M]}(X)$. Notice that $\caS_{\{\M\}}'(X)$ and $\caS_{(\M)}'(X)$ are FS and DFS spaces, respectively.

Clearly we have
\[\mathcal{D}_{[\M]}(X)\subset \mathcal{S}_{[\M]}(X)\subset \mathcal{E}_{[\M]}(X)\quad\text{and}\quad \mathcal{E}_{[\M]}'(X)\subset \mathcal{S}_{[\M]}'(X)\subset\mathcal{D}_{[\M]}'(X).\]

Alternatively, in view of Proposition \ref{equiv_GS}, one can define the spaces $\caS_{\{\M\}}(X)$ and $\caS_{(\M)}(X)$ as the inductive and projective limits, respectively, of the Banach spaces
\[\caS_{\M}(X;h,\varepsilon) = \{f\in C^\infty(X)\,:\, \|f\|_{h,\varepsilon}<+\infty\},\quad h,\varepsilon>0,\]
where
\[\|f\|_{h,\varepsilon} = \sum_{k=1}^N\sup_{x\in U_k}\sup_{\alpha\in\N_0^n}\frac{e^{\omega_{\M}(\varepsilon|x|)}|\partial_x^\alpha (f\circ\varphi_k^{-1})(x)|}{h^{|\alpha|}M_{|\alpha|}}.\]

\subsection{Gelfand-Shilov spaces on analytic manifolds with ends}

Observe that every manifold with ends (as defined in the previous section) can be identified with the interior of a compact manifold with boundary. For this particular class of manifolds, we can obtain an alternative characterization of $\mathcal{S}_{[\M]}(X)$. 

Let $\X$ be a compact manifold with boundary $\partial X$ and assume that the weight sequence $\M=\Mj$ satisfies also $(M2)$ and $(M4)$. In this case, we can also assume that the weight function $\omega_\M$ satisfies the properties of $\omega_0$ in Proposition \ref{prop_wf_E}.

\begin{definition}
	A boundary defining function is a smooth function $\varrho:\overline{X}\to [0,+\infty)$ such that $\varrho^{-1}(0)=\partial X$ and $\mathrm{d}\varrho_{y}\neq 0$ for every $y\in\partial X$.
\end{definition}

Similarly to the analogous result proved in \cite{CKMT_GH_Gevrey} in the Gevrey setting, we now construct an ultradifferentiable boundary defining function for analytic manifolds with boundary.

\begin{proposition}\label{bdf_ultra}
	Let $\X$ be a smooth compact manifold with boundary endowed with a finite analytic atlas $\{(\Omega_k,\varphi_k)\}_{i=1}^N$. 
	Given a non-quasianalytic sequence $\M=\Mj$, there exists a boundary defining function $\varrho:\X\to[0,+\infty)$ that belongs to $\E_{[\M]}(\X)$.
\end{proposition}
\begin{proof}
	In an interior chart $(\Omega_k,\varphi_k)$, we just set $\varrho_k\equiv 1$. In a boundary chart $(\Omega_k,\varphi_k)$, with $\varphi_k(\Omega_k)\subset\mathbb{R}^{n-1}\times[0,+\infty)$, we consider $\varrho_k(x_1,\dots,x_n)=x_n$. Then, $\varrho_k=0$ in $\partial X$ and $\varrho_k>0$ outside $\partial X$. Let $\{\Phi_k\}_{k=1}^N$ be a partition of unity of order subordinated to the cover $\{\Omega_k\}_{k=1}^N$, such that each $\Phi_k$ belongs to $\D_{[\M]}(\X)$. Since each $\varrho_k$ is analytic on its domain,
	\[\varrho(x) = \sum_{k=1}^N \Phi_k(x)(\varrho_k\circ\varphi_k^{-1})(x)\]
	belongs to $\E_{[\M]}(\X)$. Moreover, $\mathrm{d}\varrho$ does not vanish in any $y\in \partial X$ as shown in the proof of \cite[Proposition 5.41]{Lee_smooth}.
\end{proof}

Then, close to $\partial X$, the boundary defining function induces coordinates $(\varrho,y)$ on a collar neighbourhood of the boundary $C = [0,\delta)\times\partial X$, for some $\delta>0$. Here, we take $\delta>0$ to be sufficiently small such that $C$ is contained in the union of the exit charts.

Then, in analogy with Melrose's definition of the space $\dot{\mathcal{C}}^\infty(X)$ (see \cite{Melrose_GST}), we define
\begin{equation}\label{eq:GSspacesmwe}
\dot{\E}_{\{\M\}}(X) = \bigcup_{\varepsilon>0}e^{-\omega_\M(\varepsilon /\varrho)}\E_{\{\M\}}(\X) \quad\text{and}\quad \dot{\E}_{(\M)}(X) = \bigcap_{\varepsilon>0}e^{-\omega_\M(\varepsilon /\varrho)}\E_{(\M)}(\X).
\end{equation}
Hence, if $\X$ is an analytic manifold with ends as in the previous section, by radial compactification, we can see $X$ as the interior of a compact manifold with boundary $\X$. Moreover, the boundary defining function constructed in Proposition \ref{bdf_ultra} induces the diffeomorphism $E\simeq (\delta,+\infty)\times\partial X$ as required in Definition \ref{def_mfwe}. Indeed, notice that the transformation $t = 1/\varrho$ induces a diffeomorphism $C\setminus\partial X \simeq (1/\delta,+\infty)$.
\begin{theorem}
	Let $X$ be a manifold with ends equipped with a SGA-atlas and $\M=\Mj$ be a non-quasianalytic weight sequence. Then $\mathcal{S}_{[\M]}(X) = \dot{\E}_{[\M]}(X)$.
\end{theorem}
\begin{proof}
	We just need to check the behaviour of the functions in the exit charts, where we have coordinates of the form $(t,ty)$, where $t>\delta$ and $y \in \tilde{U}_k\subset\R^{n-1}$, for $k=1,\dots,N$ and $\delta\geq 1$. Here, $t=1/\varrho$ and $\tilde{U}_k$ is a local chart of the closed manifold $\partial X_0$. By the compactness of $\partial X_0$, we can assume that each $\tilde{U}_k$ is bounded. In this case, we have
	\[|(t,ty)| = (t^2+t^2y^2)^{\frac{1}{2}} = t\langle y\rangle \asymp t,\]
	that is, we can obtain a constant $\kappa>0$ such that
	\begin{equation}\label{ty_asymp_t}
		\kappa^{-1}|(t,ty)| \leq t \leq \kappa|(t,ty)|. 
	\end{equation}
	
	First, suppose that $f\in\mathcal{S}_{\{\M\}}(X)$ and let $h,\varepsilon>0$ be such that
	\[\sup_{(t,y)\in U_k}\sup_{\alpha\in\N_0^n}\frac{e^{\omega_\M(\varepsilon|(t,ty)|)}|\partial^\alpha f(t,y)|}{h^{|\alpha|}M_{|\alpha|}} < +\infty.\]
	
	Write $f(x) = e^{-\omega_\M(\varepsilon/(4H\kappa\varrho))} e^{\omega_\M(\varepsilon/(4H\kappa\varrho))} f(x)$, where $H\geq 1$ is given by $(M2)'$ and $\kappa>0$ is given by \eqref{ty_asymp_t}.
	
	By Propositions, \ref{exp_EM}, \ref{E_cl_prod_diff}, and \ref{prop_wf_E}, we have $e^{\omega_\M(\varepsilon/(4H\kappa\varrho))} f(x)=g(x)\in\E_{\{\M\}}(X)$. In any exit chart, by Leibniz Rule and Fa\`a di Bruno Formula (Lemma \ref{faa}), we obtain
	\begin{align*}
		|\partial_t^\alpha\partial_y^\beta g(t,y)| & \leq \sum_{k=0}^\alpha\binom{\alpha}{k}|\partial_t^{\alpha-k}\partial_y^\beta f(t,y)||\partial_t^k e^{\omega_\M\left(\tfrac{\varepsilon}{4H\kappa}t\right)}|\\
		& \leq \sum_{k=0}^\alpha\binom{\alpha}{k}e^{\omega_\M\left(\tfrac{\varepsilon}{4H\kappa}t\right)}|\partial_t^{\alpha-k}\partial_y^\beta f(t,y)| \sum_{\gamma\in\Delta(k)}\frac{k!}{\gamma!}\prod_{\ell=1}^k\left|\frac{\partial_t^\ell \omega_\M\left(\tfrac{\varepsilon}{4H\kappa}t\right)}{\ell!}\right|^{\gamma_\ell},
	\end{align*}
	for all $\alpha\in\N_0$ and $\beta\in\N_0^{n-1}$, where the set $\Delta(k)$ is defined in \eqref{Delta_k}.
	
	Notice that, since $\M$ satisfies $(M4)$, the sequence $\{M_j/j!\}_{j\in\N_0}$ satisfies $(M1)$. Then, by Proposition \ref{prop_wf_E}, Lemmas \ref{lemma_sum} and \ref{lemma_prod_M}, and the fact that $t^k\leq k!e^t$, for all $t>0$ and $k\in\N_0$, we obtain $C,h>0$ such that
	\begin{align*}
		\sum_{\gamma\in\Delta(k)}\frac{k!}{\gamma!}\prod_{\ell=1}^k\left|\frac{\partial_t^\ell \omega_\M\left(\tfrac{\varepsilon}{4H\kappa}t\right)}{\ell!}\right|^{\gamma_\ell} & \leq \sum_{\gamma\in\Delta(k)}\frac{k!}{\gamma!}\prod_{\ell=1}^k\left|\frac{C\varepsilon^\ell (4H\kappa)^{-\ell}h^\ell M_\ell \omega_\M\left(\tfrac{\varepsilon}{4H}t\right)}{\ell!}\right|^{\gamma_\ell}\\
		& \leq h_0^k \sum_{\gamma\in\Delta(k)} \frac{k!}{\gamma!} C^{|\gamma|} \omega_\M\left(\tfrac{\varepsilon }{4H\kappa}t\right)^{|\gamma|}\frac{M_k}{k!}\frac{|\gamma|!}{M_{|\gamma|}}\\
		& \leq h_0^k \sum_{\gamma\in\Delta(k)} \frac{1}{\gamma!} C^{|\gamma|}|\gamma|! e^{\omega_\M \left(\tfrac{\varepsilon }{4H\kappa}t\right)} M_k\frac{|\gamma|!}{M_{|\gamma|}}\\
		& \leq C_\M h_0^k M_k e^{\omega_\M\left(\tfrac{\varepsilon}{4H\kappa}t\right)} \sum_{\gamma\in\Delta(k)} \frac{|\gamma|!}{\gamma!} C^{|\gamma|}\\
		& \leq C_\M h_0^k M_k e^{\omega_\M \left(\tfrac{\varepsilon}{4H\kappa}t\right)} C(1+C)^{k-1}\\
		& \leq C_\M(h_0C_0)^kM_k e^{\omega_\M \left(\tfrac{\varepsilon }{4H\kappa}t\right)},
	\end{align*}
	where $h_0=h\varepsilon(4H\kappa)^{-1}$ and $C_0=1+C$. Here, the constant $C_\M>0$ is such that $|\gamma|!/M_{|\gamma|}\leq C_\M$, for all $\gamma$, which exists by Proposition \ref{incl_quasi_factorial}.
	
	Hence, combining the previous estimates and using Proposition \ref{prop_wf_E}, we obtain
	\begin{align*}
		|\partial_t^\alpha\partial_y^\beta g(t,y)| & \leq C_\M \sum_{k=0}^\alpha\binom{\alpha}{k}  e^{\omega_\M \left(\tfrac{\varepsilon}{4H}t\right)} e^{\omega_\M \left(\tfrac{\varepsilon}{4H}t\right)}|\partial_t^{\alpha-k}\partial_y^\beta f(t,y)|(C_0h_0)^k M_k\\
		& \leq C_\M \sum_{k=0}^\alpha\binom{\alpha}{k}  e^{\omega_\M(\varepsilon t)}|\partial_t^{\alpha-k}\partial_y^\beta f(t,y)|(C_0h_0)^k M_k\\
		& \leq C_\M \sum_{k=0}^\alpha\binom{\alpha}{k}  e^{\omega_\M(\varepsilon (t+|y|))}|\partial_t^{\alpha-k}\partial_y^\beta f(t,y)|(C_0h_0)^k M_k\\
		& \leq C_\M \sum_{k=0}^\alpha\binom{\alpha}{k} C_fh_f^{\alpha-k+|\beta|}M_{\alpha-k+|\beta|}(C_0h_0)^kM_k\\
		& \leq C_1h_1^{\alpha-k+|\beta|}M_{\alpha+|\beta|},
	\end{align*}
	for all $\alpha\in\N_0$ and $\beta\in\N_0^{n-1}$, where the constants $C_f,h_f>0$ are given by the fact that $f\in\mathcal{S}_{\{\M\}}(X)$, $C_0=C_\M C_f$, and $h_1=2\max\{1,h_f,C_0h_0\}$. Hence $g\in\E_{\{\M\}}(\X)$ and we have $\mathcal{S}_{\{\M\}}(X) \subset \dot{\E}_{\{\M\}}(X)$.
	
	On the other hand, suppose that $f\in \dot{\E}_{\{\M\}}(X)$, that is, $f=e^{-\omega_\M(\varepsilon/\varrho)}g$, for some $g\in \E_{\{\M\}}(\X)$ and $\varepsilon>0$. By Leibniz Rule, Fa\`a di Bruno Formula, Proposition \ref{prop_wf_E}, and the estimate \eqref{ty_asymp_t}, we have
	\begin{align*}
		e^{\omega_\M\left(\tfrac{\varepsilon}{2H\kappa}|(t,ty)|\right)}|\partial_t^\alpha\partial_y^\beta f(t,y)| \leq\ & e^{\omega_\M \left(\tfrac{\varepsilon}{2H} t\right)} \sum_{k=0}^\alpha\binom{\alpha}{k} |\partial_t^{\alpha-k}\partial_y^\beta g(t,y)|e^{-\omega_\M(\varepsilon t)}\\
		& \times  \sum_{\gamma\in\Delta(k)}\frac{k!}{\gamma!}\prod_{\ell=1}^k\left|\frac{\partial_t^\ell \omega_\M(\varepsilon t)}{\ell!}\right|^{\gamma_\ell}\\
		\leq\ & e^{\omega_\M\left(\tfrac{\varepsilon}{2H}t\right)} \sum_{k=0}^\alpha\binom{\alpha}{k}  C_gh_g^{\alpha-k+|\beta|}M_{\alpha-k+|\beta|} \\
		& \times e^{-\omega_\M\left(\tfrac{\varepsilon}{2H} t\right)}e^{-\omega_\M\left(\tfrac{\varepsilon}{2H} t\right)} \sum_{\gamma\in\Delta(k)}\frac{k!}{\gamma!}\prod_{\ell=1}^k\left|\frac{\partial_t^\ell \omega_\M(\varepsilon t)}{\ell!}\right|^{\gamma_\ell}.
	\end{align*}
	
	Similarly to the computations of the first inclusion, we obtain
	\begin{align*}
		\sum_{\gamma\in\Delta(k)}\frac{k!}{\gamma!}\prod_{\ell=1}^k\left|\frac{\partial_t^\ell \omega_\M(\varepsilon t)}{\ell!}\right|^{\gamma_\ell} & \leq \sum_{\gamma\in\Delta(k)}\frac{k!}{\gamma!}\prod_{\ell=1}^k\left|\frac{C\varepsilon^\ell h^\ell M_\ell \omega_\M(\varepsilon t)}{\ell!}\right|^{\gamma_\ell}\\
		& \leq (h\varepsilon)^k \sum_{\gamma\in\Delta(k)} \frac{k!}{\gamma!} C^{|\gamma|}\omega_\M(\varepsilon t)^{|\gamma|}\frac{M_k}{k!}\frac{|\gamma|!}{M_{|\gamma|}}\\
		& = (h \varepsilon)^k \sum_{\gamma\in\Delta(k)} \frac{k!}{\gamma!} C^{|\gamma|}\omega_\M\left(\tfrac{\varepsilon}{2H}2H t\right)^{|\gamma|}\frac{M_k}{k!}\frac{|\gamma|!}{M_{|\gamma|}}\\
		& \leq (h \varepsilon )^k\sum_{\gamma\in\Delta(k)} \frac{k!}{\gamma!} C^{|\gamma|}K_\M\omega_\M\left(\tfrac{\varepsilon}{2H} t\right)^{|\gamma|}\frac{M_k}{k!}\frac{|\gamma|!}{M_{|\gamma|}}\\
		& \leq K_\M (h\varepsilon)^k \sum_{\gamma\in\Delta(k)} \frac{1}{\gamma!} C^{|\gamma|}|\gamma|!e^{\omega_\M\left(\tfrac{\varepsilon}{2H} t\right)} M_k\frac{|\gamma|!}{M_{|\gamma|}}\\
		& \leq K_\M C_\M (h\varepsilon)^k M_k e^{\omega_\M\left(\tfrac{\varepsilon}{2H} t\right)} \sum_{\gamma\in\Delta(k)} \frac{|\gamma|!}{\gamma!} C^{|\gamma|}\\
		& \leq K_\M C_\M h^k M_k e^{\omega_\M\left(\tfrac{\varepsilon}{2H} t\right)}C(1+C)^{k-1}\\
		& \leq K_\M C_\M (h\varepsilon C_0)^k M_k e^{\omega_\M\left(\tfrac{\varepsilon}{2H} t\right)},
	\end{align*}
	where $C_0=1+C$, $K_\M>0$ is such that $\omega_\M(2Ht)\leq K_\M\omega_\M(t)$ for all $t\geq 0$, which exists by condition $(W1)$, and $C_\M>0$ is given by Proposition \ref{incl_quasi_factorial}.
	
	Combining the the estimates above, we obtain
	\begin{align*}
		e^{\omega_\M\left(\tfrac{\varepsilon}{2H\kappa}|(t,ty)|\right)}|\partial_t^\alpha\partial_y^\beta f(t,y)| \leq\ & e^{\omega_\M\left(\tfrac{\varepsilon}{2H}t\right)}  \sum_{k=0}^\alpha\binom{\alpha}{k}  C_gh_g^{\alpha-k+|\beta|}M_{\alpha-k+|\beta|} \\
		& \times e^{-\omega_\M\left(\tfrac{\varepsilon}{2H} t\right)}e^{-\omega_\M\left(\tfrac{\varepsilon}{2H} t\right)}K_\M C_\M(hC_0)^kM_ke^{\omega_\M\left(\tfrac{\varepsilon}{2H} t\right)}\\
		\leq\ & C_0 h_0^{\alpha+|\beta|} M_{\alpha+|\beta|},
	\end{align*}
	where $C_0=K_\M C_\M C_g$, $h_0 = 2\max\{1,h\varepsilon C_0,h_g\}$, and the constants $C_g,h_g>0$ are given by the fact that $g\in\E_{\{\M\}}(\overline{X})$. Hence $f\in\mathcal{S}_{\{\M\}}(X)$ and $\dot{\E}_{\{\M\}}(X) \subset\mathcal{S}_{\{\M\}}(X)$. The projective case can be proved by similar arguments.
\end{proof}
\begin{remark}
	There are different definitions of manifold with corners, see \cite{Melrose_AMwC}, and, for instance, \cite{Joyce,M-RO92}.
	Adopting, as above, an approach based on local charts, let $Y$ be a $n$-dimensional paracompact Hausdorff space. 
	
A $n$-dimensional chart with corners (of codimension $k$) on $Y$ is a pair $(U,\phi)$, where $U$ is an open subset of $[0,+\infty)^k \times \R^{n-k}$ 
for $0 \le k \le n$, and $\phi\colon  U \to \phi(U) \subset Y$ is a homeomorphism. If $k = 1$, we obtain a usual ``boundary chart''. 
Defining compatibility between charts and an atlas of charts, the definitions of manifold with boundary and manifold with corner are obtained.

Observe that, for every manifold with corners $Y$ of dimension $n$, there exists a $n$-dimensional smooth manifold $\widetilde{Y}$ without boundary 
such that $Y\subset \widetilde{Y}$ and the interior of $Y$ is open in $\widetilde{Y}$. Then, all smooth objects and structures on $Y$ are defined by restriction
to $Y$ of the analogous ones on $\widetilde{Y}$.

Assume $Y$ to be compact and analytic, and that there is a finite collection of analytic functions on $Y$, $\{\varrho_j\}_{j\in J}$, called boundary defining functions,
such that $Y=\cap_{j\in J}\{p\in\widetilde{Y},\varrho_j(p)\ge0\}$, and at every point where $\varrho_k =0$ for every $k\in K\subset J$,
the differentials of such $\varrho_k$ are linearly independent. In particular, $d\varrho_j\not= 0$ when $\varrho_j = 0$. 
Then, it is possible to work in local coordinates $x \colon p \mapsto (\varrho_1,\dots,\varrho_k,x_1,...,x_{n-k})(p)$, where $k$ is the number of boundary defining functions.
In \cite[Remark 2.11]{Joyce} these are called compact manifolds with embedded corners. By adapting \cite[Proposition 2.15]{Joyce} to the analytic case, 
taking into account our argument in Proposition \ref{bdf_ultra} above, it follows that, locally, a boundary defining function in $\mathcal{E}_{[\mathcal{M}]}(Y)$ 
always exists, and the property that all corners are embedded implies, as in the smooth case, that a global boundary defining function in that same space 
exists as well.

Hence, definition \eqref{eq:GSspacesmwe} could be extended, in principle, to an analytic manifold with corners $Y$. Clearly, it would be interesting
also to consider spaces of the types
\begin{align*}
\dot{\E}_{\{\M\}}(Y) &= \bigcup_{\varepsilon>0}e^{-\omega_\M(\varepsilon /\varrho_1)+\dots-\omega_\M(\varepsilon /\varrho_k)}\E_{\{\M\}}(\overline{Y}) 
\\
&\text{and}
\\
\dot{\E}_{(\M)}(Y) &= \bigcap_{\varepsilon>0}e^{-\omega_\M(\varepsilon /\varrho_1)+\dots-\omega_\M(\varepsilon /\varrho_k)}\E_{(\M)}(\overline{Y}).
\end{align*}

The study of Gelfand-Shilov spaces in this more complicated geometric situation will appear elsewhere.
\end{remark}


\section{The calculus of ultradifferential SG-operators on SGA-manifolds}\label{sec:SGABmf}

Let $\{\Phi_k\}_{k=1}^N$ be a partition of unity as constructed in Theorem \ref{thm_part} and consider the family $\{\Theta_k\}_{k=1}^N$ 
as in Corollary \ref{coro_Theta}. If $P:\mathcal{S}_{[\M]}(X)\to \mathcal{S}_{[\M]}(X)$ is a linear operator, we can write
\[P = \sum_{k=1}^N\Theta_k P\Phi_k + \sum_{k=1}^N(1-\Theta_k) P\Phi_k.\]

\begin{definition}
	We say that a linear operator $P:\mathcal{S}_{[\M]}(X)\to \mathcal{S}_{[\M]}(X)$ is a (ultradifferential SG-)pseudodifferential operator on $X$
	with symbol $p=(p_1,\dots,p_N)\in {\SG}_{[\A,\B]}^{m,\mu}(X)$ if $(1-\Theta_k)P\Phi_k$ has ultradistributional kernel 
	in $\mathcal{S}_{[\M]}(\mathbb{R}^{2n})$, and, in local coordinates, $(\Theta_k P \Phi_k)_* = (\Theta_k)_*\mathrm{Op}(p_k)(\Phi_k)_*$, $k=1,\dots,N$.
\end{definition}

\begin{remark}
	The definition of a pseudodifferential operator with symbol in ${\SG}_{[\A,\B]}^{m,\mu}(X)$ does not depend on the choice of the cut-off functions
	$\Theta_k$. Indeed, if $\{\tilde{\Theta}_k\}$ has the same properties, then $\Theta_k-\tilde{\Theta}_k$ vanishes on the support of 
	$\Phi_k$ and we can apply Proposition \ref{disj_supp}. Similarly, the definition does not depend on the chosen partition of unity or on the choice
	of the SGA-admissible atlas.
\end{remark}

We can then apply the local results proved above to extend to $X$ the properties of ultradifferential SG-operators on $\mathbb{R}^n$.

\begin{proposition}
	Let $p\in {\SG}_{[\A,\B]}^{m,\mu}(X)$ and $q\in {\SG}_{[\A,\B]}^{m',\mu'}(X)$. 
	Then $\mathrm{Op}(p)\mathrm{Op}(q)$ is a pseudodifferential operator with symbol $s\in {\SG}_{[\A,\B]}^{m'',\mu''}(X)$, 
	with $m''=m+m'$ and $\mu''=\mu+\mu'$, that has asymptotic expansion of the form \eqref{asympt_exp_comp} in local coordinates.
\end{proposition}

Now, let us assume that $X$ is endowed with a Riemannian metric $g$. In each local chart $U_k$, we can write
\begin{equation*}
	g = \sum_{\ell,\nu=1}^n g_{\ell,\nu}^{(k)}(x)\,\mathrm{d}x_\ell\otimes\mathrm{d}x_\nu,
\end{equation*}
where the coefficients $g_{\ell,\nu}^{(k)}(x)$ are smooth functions on $U_k$ satisfying $g_{\ell,\nu}^{(k)}(x)=g_{\nu,\ell}^{(k)}(x)$. Moreover, we can assume that the coefficients $g_{\ell,\nu}^{(k)}$ belong to $\E_{[\M]}(U)$. This is possible since every smooth manifold admits an analytic metric by Grauert's Theorem \cite{Grauert}, in addition to Proposition \ref{incl_quasi_factorial}, since $\M$ is non-quasianalytic.

\begin{remark}
	If $X$ is the interior of an analytic compact manifold with boundary $\X$, we can procceed as in \cite{CKMT_GH_Gevrey} and use the ultradifferentiable boundary defining function $\varrho$ constructed in Proposition \ref{bdf_ultra} and a ultradifferentiable partition of unity to obtain an ultradifferentiable scattering metric on $X$, that is, a Riemannian metric $g$ that has coefficients in $\E_{[\M]}$ in any analytic chart and has the form
	\[g = \dfrac{\mathrm{d}\varrho^2}{\varrho^4}+\dfrac{g'}{\varrho^2}\]
	in a collar neighbourhood of $\partial X$, where $g'$ restricts to a metric in $\partial X$.
\end{remark}

Suppose also that $X$ is orientable. In this case, we can assume that the matrix $\left[g_{\ell,\nu}^{(k)}(x)\right]$ is positive definite for every $x\in U_k$. Hence $g$ induces a volume form $\mathrm{d}V$ on $X$ that is given by
\[g^{(k)}(x)\,\mathrm{d}x_1\wedge\cdots\wedge \mathrm{d}x_n\]
in each local chart, where
\[g^{(k)}(x) = \sqrt{\det \left[g_{\ell,\nu}^{(k)}(x)\right]}.\]

We refer the reader to \cite[Proposition 7.7]{Kumano-go} for the details of this construction.

\begin{definition}
	Given a symbol $p\in {\SG}_{[\A,\B]}^{m,\mu}(X)$, we define the formal adjoint of $\mathrm{Op}(p)$ as the operator $\mathrm{Op}(p)^*$ that satisfies
	\begin{equation*}
		\int_X \mathrm{Op}(p)u(x) \overline{v(x)}\,\mathrm{d}V = \int_X u(x) \overline{\mathrm{Op}(p)^* v(x)}\,\mathrm{d}V,
	\end{equation*}
	for all $u,v\in\mathcal{S}_{[\M]}(X)$. If $\mathrm{Op}(p)=\mathrm{Op}(p)^*$, we say that $\mathrm{Op}(p)$ is formally self-adjoint.
\end{definition}

\begin{theorem}
	Given a symbol $p\in {\SG}_{[\A,\B]}^{m,\mu}(X)$, the operator $\mathrm{Op}(p)^*$ is a (ultradifferential SG-)pseudodifferential operator with symbol 
	$p^*\in {\SG}_{[\A,\B]}^{m,\mu}(X)$ that has, in each local chart $U_k$, $k=1,\dots,N$, asymptotic expansion
	\begin{equation*}
		p_k^*(x,\xi)\sim \sum_{j=0}^\infty\sum_{|\alpha|=j}\frac{1}{\alpha!}\tilde{g}^{(k)}(x)^{-1}\partial_\xi^\alpha D_x^\alpha(\tilde{g}^{(k)}(x) \overline{p_k(x,\xi)}),
	\end{equation*}
	where $\tilde{g}^{(k)}$ is an extension of ${g}^{(k)}$ to a slightly larger domain.
\end{theorem}
\begin{proof}
	The proof follows the same lines of the one of \cite[Theorem 7.9]{Kumano-go}, employing Proposition \ref{asympt_exp_adj}.
\end{proof}

By the invariance properties and the fact that the ellipticity of a local symbol on $\R^n$ is stable under perturbations by symbols of strictly lower order
(in both components), the next Definition \ref{def:ellX} makes sense.

\begin{definition}\label{def:ellX}
	Let $p\in {\SG}_{[\A,\B]}^{m,\mu}(X)$. We say that $p$ is elliptic if, in local coordinates, it is elliptic in the sense of Definition \ref{def_elliptic}. 
	Similarly, we say that $p$ is hypoelliptic if, in every local chart, $p$ satisfies \eqref{hypo_cond1} and \eqref{hypo_cond2} for the same constants $\mu',m'\in\R$.
\end{definition}

We can now prove the global version on the SGA-manifold $X$ of the elliptic regularity for ultradifferential SG-operators.

\begin{theorem}
	If $p\in {\SG}_{[\A,\B]}^{m,\mu}(X)$ is a hypoelliptic symbol, then
	\begin{equation}\label{GH_X}
		u\in\mathcal{S}_{[\M]}'(X),\ \Op(p)u\in \mathcal{S}_{[\M]}(X)\ \Rightarrow\ u\in \mathcal{S}_{[\M]}(X).
	\end{equation}
\end{theorem}
\begin{proof}
	By Theorem \ref{thm_para}, on each local chart we can obtain a local parametrix, from which we can obtain a global parametrix by using the partition of unity 
	and the cutoff functions that we constructed in Theorem \ref{thm_part} and Corollary \ref{coro_Theta}, respectively. This implies \eqref{GH_X}.
\end{proof}

\section{Ultradifferential SG-classical operators on analytic manifolds with ends}\label{sec:SGABmwe}
In this concluding section we shortly discuss how, on the subclass of admissible analytic manifolds introduced in Section \ref{subs:mwe},
even the calculus of ultradifferential SG-classical operators, having local symbols in the classical subclasses described in Section \ref{sec:sgucl}, 
can be invariantly defined.  

This follows by arguments similar to those employed to prove the analogous result in the standard SG-classical case on the manifolds
with ends studied in \cite{MP2002} by employing \textit{admissible charts} of the form $\varphi_j$ described in Section \ref{subs:mwe}.
Indeed, it is not restrictive to assume that the charts $\varphi_j$ take value in $\mathbb{S}^{n-1}$,
so that the change of coordinates on the end $E\simeq[1,+\infty)\times\partial X_0$ are actually of the polar coordinates form 
\[
	F\colon\rho\omega\equiv(\rho,\rho\omega)\mapsto(\rho,\rho\phi(\omega))\equiv\rho\phi(\omega), \quad\rho\in(1,+\infty), \; \omega,\phi(\omega)\in\mathbb{S}^{n-1},
\]
like those employed in \cite{MP2002}. This can be achieved, for instance, by covering $\partial X_0$ with (a finite number of) counterimages, via the
homeomorphisms $\varphi_j^{-1}$, of open balls in $\varphi_j(\Omega_j)\subseteq\mathbb{R}^{n-1}$ with suitably small radius, and then mapping such balls 
to the interior of the upper semisphere $\mathbb{S}^{n-1}_+$ by stereographic projection. This preserves the analiticity of the change of coordinates 
$\phi\colon\mathbb{S}^{n-1}\to\mathbb{S}^{n-1}$.

Now, consider an operator $A=\Op(a)$ with $a\in{\SG}_{[\A,\B],\mathrm{cl}}^{m,\mu}(\R^n)$ supported in the image of an exit chart of $E$. If we consider
the pull-back $B=F^*A$ by the change of coordinates introduced above, we find $B=\Op(b)$ with $b\in{\SG}_{[\A,\B],\mathrm{cl}}^{m,\mu}(\R^n)$.
Indeed, that $B$ is still a SG-ultradifferential operator with the same order components follows by the invariance results in Section \ref{sec_coord_inv}.
That $b$ is also SG-classical follows by iterating the analogous argument in \cite{MP2002}. Indeed, the mentioned result implies that $b$ admits
a homogeneous principal triple $(b_\psi, b_e, b_{\psi e})$, obtained by $(a_\psi, a_e, a_{\psi e})$ via the same transformation laws proved there.
Since 
\[
	a_{m-1,\mu-1}=a-(\chi_\xi a_\psi+\phi_x a_e - \chi_\xi\phi_x a_{\psi e}) \in {\SG}_{[\A,\B],\mathrm{cl}}^{m-1,\mu-1}(\R^n),
\]
we can apply the same result to $a_{m-1,\mu-1}$, obtain a corresponding principal homogeneous triple of orders $m-1,\mu-1$, 
and iterate the procedure an arbitrary number of times. 
The two sequences of homogeneous components with respect to $\xi$ and to $x$ thusly obtained provide the sequences needed in Definition 
\ref{def:classical_symbols}, proving that $b$ is also SG-classical, as claimed.

Moreover, the transformation laws of $a_\psi, a_e, a_{\psi e}$, given in \cite{MP2002}, show that these objects have an invariant meaning, as 
local representations of functions on suitable vector bundles. We collect the consequences of the above results in the next Theorem \ref{thm:sguclmwe},
whose remaining proof details we leave for the reader.

\begin{theorem}\label{thm:sguclmwe}
Let $\A=\Aj,\B=\Bj,\M=\Mj$ be three non-quasianalytic weight sequences such that $\A,\B\preceq\M$, and let a be a linear operator 
$A\colon\mathcal{S}_{[\M]}(X)\to\mathcal{S}_{[\M]}(X)$ on the manifold with ends $X$, equipped with admissible charts on the ends
$E_1,\dots,E_m$. $A$ is a SG-classical operator on $X$, with orders $m,\mu\in\R$, and we write $A\in\Psi_{\A,\B}^{m,\mu}(X)$, if
\begin{enumerate}
	\item[(a)] it is a classical pseudodifferential operator of order $\mu$ in $X_0\cup E_0$, where $E_0\subset E$ is an embedded
	submanifold analytically diffeomorphic to $[1,2)\times\partial X_0$;
	\item[(b)] it is a pseudodifferential operator defined by local symbols in ${\SG}_{[\A,\B],\mathrm{cl}}^{m,\mu}(\R^n)$ in all local charts of $E_1,\dots,E_m$.  
\end{enumerate}

By the local calculus properties, it follows
\[
	\Psi_{\A,\B}^{m,\mu}(X)\circ\Psi_{\A,\B}^{s,\sigma}(X) \subseteq \Psi_{\A,\B}^{m+s,\mu+\sigma}(X), \quad m,s,\mu,\sigma\in\R.
\]

The local principal homogeneous symbol then associated with $A\in\Psi_{\A,\B}^{m,\mu}(X)$, namely 
\[
	\sigma(A)=(\sigma_\psi(A),\sigma_e(A),\sigma_{\psi e}(A))=(a_\psi,a_e,a_{\psi e}),
\]
are local representations of smooth functions on suitable subbundles of $T^*X$, which we denote by the same symbols and names, defined, 
respectively\footnote{We recall that the notation $T^*_YM$, for $Y\subset M$ submanifold of the manifold $M$, denotes the restriction of the
cotangent bundle $T^*M$ to $Y$, that is, $T^*_YM=\{(y,\eta)\colon y\in Y,\ \eta\in T^*_y M\}$, while $T^*M\setminus0$ denotes the cotangent bundle
without its $0$-section.}:
\begin{enumerate}
	\item[(c)] on $T^*X\setminus0$, providing the ``standard'', \textit{interior homogeneous principal symbol} of $A$;
	\item[(d)] on $T^*_{\partial X_0}X$ and $T^*_{\partial X_0}X\setminus0$, collectively providing the so-called \textit{exit principal symbol} of $A$.  
\end{enumerate}

The principal symbol triple behaves multiplicatively, component by component, under composition, that is, for any
$A\in\Psi_{\A,\B}^{m,\mu}(X)$, $B\in\Psi_{\A,\B}^{s,\sigma}(X)$,
\[
	\sigma(A\circ B)=\sigma(A)\cdot\sigma(B)=(\sigma_\psi(A)\cdot\sigma_\psi(B),\sigma_e(A)\cdot\sigma_e(B),\sigma_{\psi e}(A)\cdot\sigma_{\psi e}(B)).
\]

An operator $A\in\Psi_{\A,\B}^{m,\mu}(X)$ is elliptic if and only if the three components of $\sigma(A)$ are everywhere non-vanishing
on the respective domains. In such case, $A$ admits a parametrix $P\in\Psi_{\A,\B}^{-m,-\mu}(X)$, that is
\begin{equation*}\label{eq:paramX}
	AP=I+R_1,\quad PA=I+R_2,\quad R_1,R_2\in\Psi_{\A,\B}^{-\infty,-\infty}(X)=\bigcap_{m,\mu\in\R}\Psi_{\A,\B}^{m,\mu}(X),
\end{equation*}
and
\[
	\sigma(P)=\sigma(A)^{-1}=(\sigma_\psi(A)^{-1},\sigma_e(A)^{-1},\sigma_{\psi e}(A)^{-1})
\]

Assuming $\A=\B=\M$, for any $A\in\Psi_{\M}^{m,\mu}(X)$ we have well-defined characteristic sets 
\[
	\mathrm{char}(A)=(\mathrm{char}_\psi(A), \mathrm{char}_e(A), \mathrm{char}_{\psi e}(A)),
\]
that is, the subsets of $T^*X\setminus0$, $T^*_{\partial X_0}X$ and $T^*_{\partial X_0}X\setminus0$ where,
respectively, the three corresponding components of $\sigma(A)$ vanish. Equivalently to the above definition in this
case, $A\in\Psi_{\M}^{m,\mu}(X)$ is elliptic if and only if $\mathrm{char}(A)=(\emptyset,\emptyset,\emptyset)$.
\end{theorem}
\begin{remark}\ 
	\begin{enumerate}
		\item[  i)] 	Considering extensions of the operators $A\in\Psi_{\M}^{m,\mu}(X)$, $m,\mu\in\R$, to the dual of $\mathcal{S}_{[\M]}(X)$,
				also the notion of (global) wave-front sets and the corresponding results from Section \ref{sec:globwfs} extend to the case of manifolds
				with ends. 
		\item[ ii)] 	Completely similar conclusions hold on an analytic scattering manifold (see \cite{Melrose_GST}),
				allowing to define the operator spaces $\Psi^{m,\mu}_{[\A,\B]}(X)$ there as well (again locally given by operators
				with symbol in ${\SG}_{[\A,\B],\mathrm{cl}}^{m,\mu}(\R^n)$), the principal symbol maps, characteristic sets, and (global) wave-front sets.
		\item[iii)]	Ultradifferential $\SG$-classical operators operators acting on suitable sections of analytic vector bundles on analytic manifolds with ends
				have analogous properties, with straightforward modifications of the statements in Theorem \ref{thm:sguclmwe}.
	\end{enumerate}
\end{remark}

\bibliographystyle{plain}
\bibliography{references}	
	
\end{document}